\documentclass[preprint, 12pt, sort&compress]{elsarticle}

\usepackage{latexsym}
\usepackage{amsfonts}
\usepackage{graphicx,amssymb,natbib}
\usepackage{amsmath}
\usepackage{amssymb}
\usepackage[mathscr]{eucal}
\usepackage{enumerate}
\usepackage{amsthm}

\usepackage{bm}

\usepackage{color, soul}

\begin{document}

\newtheorem{tm}{Theorem}[section]
\newtheorem{pp}[tm]{Proposition}
\newtheorem{lm}[tm]{Lemma}
\newtheorem{df}[tm]{Definition}
\newtheorem{tl}[tm]{Corollary}
\newtheorem{re}[tm]{Remark}
\newtheorem{eap}[tm]{Example}

\newcommand{\pof}{\noindent {\bf Proof} }
\newcommand{\ep}{$\quad \Box$}

\newcommand{\al}{\alpha}
\newcommand{\be}{\beta}
\newcommand{\var}{\varepsilon}
\newcommand{\la}{\lambda}
\newcommand{\de}{\delta}
\newcommand{\str}{\stackrel}
\newcommand{\rmn}{\romannumeral}

\renewcommand{\proofname}{\bf Proof}

\allowdisplaybreaks

\begin{frontmatter}

\title{Properties of fuzzy set spaces with $L_p$ metrics
\tnoteref{usc}
 }
\tnotetext[usc]{Project supported by
 Natural Science Foundation of Fujian Province of China (No. 2020J01706) }
\author{Huan Huang}
 \ead{hhuangjy@126.com }
\address{Department of Mathematics, Jimei
University, Xiamen 361021, China}

\date{}

\begin{abstract}
This paper discusses the properties of
 the spaces of fuzzy sets in a metric space with $L_p$-type $d_p$ metrics, $p\geq 1$.
The $d_p$ metrics are well-defined
if and only if the corresponding Haudorff
distance functions are measurable.
In this paper, we give
some fundamental conclusions on the measurability of these Haudorff
distance functions.
Then we give the characterizations of compactness in fuzzy set space with $d_p$ metrics.
At last we
present the completions of fuzzy set spaces with $d_p$ metrics.

\end{abstract}


\end{frontmatter}

Please refer to the version 5 of this paper (arXiv:2209.12978v5) for the
history of the results in this paper.

\section{Introduction}

The $L_p$-type metrics are widely used in theoretical research and practical applications.
The class of $d_p$ metrics are a kind of $L_p$-type metrics.
The $d_p$ metrics are commonly used metrics
on
fuzzy sets \cite{da3, wu, du, trutschnig}.

The $d_p$ metrics are well-defined
if and only if the corresponding Haudorff distance functions are measurable.
So it is important to discuss the measurability of these Haudorff distance functions.
In \cite{huang17}, we give some conclusions on this topic.
In this paper, we give proofs to the positive conclusions and counterexamples to the negative conclusion.
Further, we make great improvements to these conclusions.
The conclusions on measurability of the function in this paper
pointed out the cases in which the $d_p$ metrics are well-defined,
and, therefore, indicate the properties of the $L_p$-type metrics can be used in these cases.
So
these conclusions are fundamental for relevant studies of the $d_p$ metrics.

Compactness is a fundamental property in both theory and applications \cite{gh3, kelley, fs}.
The characterizations of compactness for fuzzy set space with $d_p$ metrics
have attracted attentions of scholars \cite{ma,zhao,wu2}.
The completion of a given metric space is an important topic in analysis.

In \cite{huang97}, we give the characterizations of total boundedness, relative compactness and compactness
for fuzzy set space with $d_p$ metrics.
We also give the completions of various fuzzy set spaces with $d_p$ metrics.
All the fuzzy sets involved in the conclusions in \cite{huang97} are fuzzy sets
in
 the $m$-dimensional Euclidean space $\mathbb{R}^m$.

It is natural to consider spaces of
fuzzy sets in a metric space
\cite{jarn, greco, greco3}.
In this paper, we discuss the properties spaces of fuzzy sets in a general metric space with $d_p$ metrics.

We
discuss the properties of the $d_p^*$ metric ($d_p$ metrics) and the $H_{\rm end}$ metric
including the
relationship between them.
Based on these results and the characterizations of total boundedness, relative compactness and compactness
in fuzzy set space with $d_p$ metrics given in \cite{huang19c}.
We present
 the characterizations of total boundedness, relative compactness and compactness
 for space of fuzzy sets in a metric space with $d_p$ metrics.
These results reveal a connection between
 a set being total bounded (respectively, relatively
compact, compact) in the sense of the $d_p$ metrics and this set being total bounded (respectively, relatively
compact, compact) in the sense of the endograph metrics.
The results on the characterizations of total boundedness, relative compactness and compactness in this paper generalize the corresponding results in \cite{huang97}.
Furthermore, using the results
in this paper,
we give new characterizations of total boundedness, relative compactness and compactness
 for space of fuzzy sets in $\mathbb{R}^m$.

We construct completions of fuzzy set spaces in a metric space with $d_p$ metrics.
These conclusions on the completions of the spaces of fuzzy set in a metric space $(X,d)$
apply to not only the cases that $X$ is a complete metric space
but also the cases that $X$ is an incomplete metric space.
The fuzzy sets involved in the corresponding results in \cite{huang97} is fuzzy sets in $\mathbb{R}^m$, which is a complete space.
The results on completions of fuzzy set spaces in this paper improve the corresponding results in \cite{huang97}.

The remainder of this paper is organized as follows.
In Section \ref{bas}, we recall and give some basic notions and results related to fuzzy sets
and
 metrics on them.
In Section \ref{meau}, we give fundamental conclusions on
the measurability of the Hausdorff distance functions.
 In Section \ref{per},
 we discuss the properties and relation of the $d_p^*$ metric and the $H_{\rm end}$ metric.
In Section \ref{cng}, we give the characterizations of total boundedness, relative compactness and compactness in
$(F^1_{USCG} (X)^p, d_p)$.
In Section \ref{cnr},
 we pointed out that the results in Section \ref{cng} generalize the corresponding results in \cite{huang97} for $(F^1_{USCG} (\mathbb{R}^m)^p, d_p)$.
Furthermore, by using results in Sections \ref{per} and \ref{cng},
we give new
characterizations of total boundedness, relative compactness and compactness in $ (F^1_{USCG} (\mathbb{R}^m)^p, d_p)$.
In Section 7, we construct completions of spaces of fuzzy sets in a metric space according to $d_p$ metrics.
In Section 8, we draw our conclusions.

\section{Fuzzy sets and metrics on them} \label{bas}

In this section, we recall and give some notions and results related to fuzzy sets
and
metrics on them.
Readers
can refer to \cite{wu, da3, du, wang2, wa, rojas, garcia, kloeden, kloeden2}
for studies and applications of fuzzy theory.

Let $\mathbb{N}$ be the set of all positive integers,
and let
$\mathbb{R}^m$
be the $m$-dimensional Euclidean space ($\mathbb{R}^1$ is also written as $\mathbb{R}$).
We use \bm{$\rho_m$} to denote the Euclidean metric on $\mathbb{R}^m$.
We use $\mathbb{R}^+ $ to denote the set $\{ x\in \mathbb{R}: x\geq 0\}$.

In this paper,
if not specifically mentioned, we suppose that $X$ is \emph{a metric space endowed with a metric}
$d$. For simplicity, we also use $X$ to denote the metric space $(X,d)$.

Let
$F(X)$
denote the set of
all fuzzy sets in $X$. A fuzzy set $u\in F(X)$ can be seen as a function $u:X \to [0,1]$.
A
subset $S$ of $X$ can be seen as a fuzzy set in $X$.
If there is no confusion,
 the fuzzy set in $X$ corresponding to $S$ is often denoted by $\chi_{S}$; that is,
\[ \chi_{S} (x) = \left\{
                    \begin{array}{ll}
                      1, & x\in S, \\
                      0, & x\in X \setminus S.
                    \end{array}
                  \right.
\]
For simplicity,
for
$x\in X$, we will use $\widehat{x}$ to denote the fuzzy set  $\chi_{\{x\}}$ in $X$.
In this paper, if we want to emphasize a specific metric space $X$, we will write the fuzzy set in $X$ corresponding to $S$ as
$S_{F(X)}$, and the fuzzy set in $X$ corresponding to $\{x\}$ as $\widehat{x}_{F(X)}$.

For
$u\in F(X)$, let $[u]_{\al}$ denote the $\al$-cut of
$u$, i.e.
\[
[u]_{\al}=\begin{cases}
\{x\in X : u(x)\geq \al \}, & \ \al\in(0,1],
\\
{\rm supp}\, u=\overline{    \{x: u(x) > 0 \}    }, & \ \al=0,
\end{cases}
\]
where $\overline{S}$
denotes
the topological closure of $S$ in $(X,d)$.

For
$u\in F(X)$, $u$ is said to be normal if $[u]_1 \not= \emptyset$.
We use $F^1(X)$ to denote the family of all normal fuzzy
sets in $X$.

For
$u\in F(X)$,
define
\begin{gather*}
{\rm end}\, u:= \{ (x, t)\in  X \times [0,1]: u(x) \geq t\},
\\
{\rm send}\, u:= \{ (x, t)\in  X \times [0,1]: u(x) \geq t\} \cap  ([u]_0 \times [0,1]).
\end{gather*}
$
{\rm end}\, u$ and ${\rm send}\, u$
 are called the endograph of $u$ and the sendograph of $u$, respectively.

If $X$ is replaced by a nonempty set $Y$ in the above four paragraphs, then
the definitions and notations in the above four paragraphs still apply except
for the notations $[u]_0$ and ${\rm send}\, u$.
If $(X,d)$ is replaced by a topological space $(Y, \tau)$ in the above four paragraphs, then
the definitions and notations in the above four paragraphs still apply.

Let
$(X,d)$ be a metric space.
 We
use $\bm{H}$ to denote the \emph{\textbf{Hausdorff distance}}
on
 $C(X)$ induced by $d$, i.e.,
\begin{equation} \label{hau}
\bm{H(U,V)}  =   \max\{H^{*}(U,V),\ H^{*}(V,U)\}
\end{equation}
for each $U,V\in C(X)$,
where
\begin{equation*}\label{haus}
  d\, (u,V) = \inf_{v\in
V}d\, (u,v),\  H^{*}(U,V)=\sup\limits_{u\in U}\,d\, (u,V) =\sup\limits_{u\in U}\inf\limits_{v\in
V}d\, (u,v).
\end{equation*}
We call $H^*$ the \emph{Hausdorff pre-distance} related to $H$.

In this paper, if we want to emphasize that $H$ and $H^*$  are induced by a specific metric $\lambda$, we will
write $H$ and $H^*$ as $H_\lambda$ and $H^*_{\lambda}$, respectively.

The metric $\overline{d}$ on $X \times [0,1]$ is defined
as
$$  \overline{d } ((x,\al), (y, \beta)) = d(x,y) + |\al-\beta| .$$
If not specifically mentioned, we suppose by default that the metric on $X \times [0,1]$ is $\overline{d}$. For simplicity, if there is no confusion,
we also use $X \times [0,1]$ to denote the metric
space $(X \times [0,1], \overline{d})$, and
also use $H$ to denote the Hausdorff distance on $C(X\times [0,1])$ induced by $\overline{d}$ on $X \times [0,1]$.

\begin{re}
 {\rm
Let $Y$ be a nonempty set.
$\rho$ is said to be a \emph{metric} on $Y$ if $\rho$ is a function from $Y\times Y$ into $\mathbb{R}$
satisfying
positivity, symmetry and triangle inequality. At this time, $(Y, \rho)$ is said to be a metric space.
  $\rho$ is said to be an \emph{extended metric} on $Y$ if $\rho$ is a function from $Y\times Y$ into $\mathbb{R} \cup \{+\infty\} $
satisfying
positivity, symmetry and triangle inequality. At this time, $(Y, \rho)$ is said to be an extended metric space. Clearly a metric space is an extended metric space.

Let $(Y, \rho)$ be an extended metric space,
$y\in Y$ and $\varepsilon>0$. We use
$B_{(Y,\rho)}(y, \varepsilon)$ to denote the set
$\{z\in Y: \rho(y,z) < \varepsilon\}$
and use
$\overline{B}_{(Y,\rho)}(y, \varepsilon)$ to denote the set
$\{z\in Y: \rho(y,z) \leq \varepsilon\}$.
If there is no confusion, we will write
$B_{(Y,\rho)}(y, \varepsilon)$ as $B(y, \varepsilon)$
and write
$\overline{B}_{(Y,\rho)}(y, \varepsilon)$ as $\overline{B}(y, \varepsilon)$.

Let $(Y, \rho)$ be an extended metric space.
$\{B(y, \varepsilon): y\in Y, \varepsilon>0\}$
is a basis for the topology induced by $\rho$ on $Y$.
The topological closure of a set $A$ in $(Y, \rho)$, denoted by
$\overline{A}$,
refers to the closure of $A$ in $Y$ according to the topology induced by $\rho$ on $Y$.
Then $x\in \overline{A}$
if and only if
there is a sequence   $\{x_n\}$ in $A$
such that
$\rho(x_n, x) \to 0$.
So $x\in \overline{A}$ if and only if $\rho(x, A) =0$.

Let $(Y, \rho)$ be an extended metric space
and $S$ a nonempty set in $Y$.
we use $\rho|_{S}$ to denote the induced metric on $S$ by $\rho$.
$(S, \rho|_{S})$ is called
 a subspace of $(Y, \rho)$.
If there is no confusion we \emph{also use} $\rho$ \emph{to denote} $\rho|_{S}$.
Obviously, if $(Z, \lambda)$ is an extended metric space (respectively, metric space) and $A$ is a subset of $Z$,
then $(A, \lambda|_A)$ is also an extended metric space (respectively, metric space).

 }
\end{re}

\begin{re}\label{haum}
{\rm
Let $(X,d)$ be a metric space. The Hausdorff distance $H$ on $K(X)$ induced by $d$ on $X$ is a metric. The Hausdorff distance $H$ on $C(X)$ induced by $d$
on $X$
is an extended metric.
As $(X \times [0,1], \overline{d})$ is a metric space,
the Hausdorff distance $H$ on $K(X\times [0,1])$ induced by $\overline{d}$ on $X \times [0,1]$
is a metric and the Hausdorff distance $H$ on $C(X\times [0,1])$ induced by $\overline{d}$
 on $X \times [0,1]$ is an extended metric.

Each of $H$ on $C(X)$ induced by $d$ on $X$ and $H$ on $C(X\times [0,1])$ induced by $\overline{d}$
 on $X \times [0,1]$
 does not need to be a metric.
Clearly $H$ on $C(X)$ induced by $d$ on $X$ is a metric
if and only if $H$ on $C(X\times [0,1])$ induced by $\overline{d}$
 on $X \times [0,1]$ is a metric.

Clearly $H(\{0\}, [0,+\infty))=+\infty$. So $H$ on $C(\mathbb{R}^m)$ is an extended metric but not a metric,
and then this is also true for $H$ on $C(\mathbb{R}^m\times [0,1])$.

 When the Hausdorff distance $H$ is an extended metric, it is also called
the Hausdorff extended metric.
When the Hausdorff distance $H$ is a metric, it is also called the Hausdorff
metric.
In this paper, for simplicity,
 we refer to both the Hausdorff extended metric and the Hausdorff metric as the \textbf{Hausdorff metric} unless
there is a need to specifically indicate what they are.
}
\end{re}

Let $(Y,\rho)$ be an extended metric space.
The symbols $K(Y)$ and
 $C(Y)$ are used to
 denote the set of all non-empty compact subsets of $(Y,\rho)$ and the set of all non-empty closed subsets of $(Y,\rho)$, respectively.

Let $u\in F(Y)$. We say that $u$ is an upper semi-continuous fuzzy set in $(Y,\rho)$
if $u(x) \geq \limsup_{y\to x} u(y)$ for each $x\in Y$.
The following conditions are equivalent:
(\rmn1) $u$ is an upper semi-continuous fuzzy set in $(Y,\rho)$,
(\rmn2) for each $\al\in \mathbb{R}$, $\{x\in Y: u(x)\geq \al\} \in C(Y)\cup \{\emptyset\} $, and (\rmn3) for each $\al\in (0,1]$, $[u]_\al \in C(Y)\cup \{\emptyset\} $. (\rmn1)$\Leftrightarrow$(\rmn2) is well-known.
Given $v\in F(Y)$.
If $\al>1$ then $\{x\in Y: v(x)\geq \al\} = \emptyset \in C(Y)\cup \{\emptyset\} $. If
$\al\leq 0$ then $\{x\in Y: v(x)\geq \al\} = Y \in C(Y)\cup \{\emptyset\}$.
So (\rmn2)$\Leftrightarrow$(\rmn3).
Thus (\rmn1)$\Leftrightarrow$(\rmn2)$\Leftrightarrow$(\rmn3).

Let
$F_{USC}(Y)$
denote
the family of all upper semi-continuous fuzzy sets in $(Y,\rho)$
and let $F^1_{USC}(Y)$ denote
the family of normal fuzzy sets in $F_{USC}(Y)$,
i.e.,
\begin{gather*}
 F_{USC}(Y) :=\{ u\in F(Y) : [u]_\al \in  C(Y)\cup \{\emptyset\}  \  \mbox{for all} \   \al \in (0,1]   \},
\\
 F^1_{USC}(Y) :=  F_{USC}(Y)\cap F^1(Y)=\{ u\in F(Y) : [u]_\al \in  C(Y)   \  \mbox{for all} \   \al \in (0,1]   \}.
\end{gather*}
Obviously for each $u\in F(Y)$, $[u]_0\in C(Y)\cup \{\emptyset\}$,
and
for each $u\in F^1_{USC}(Y)$, $[u]_0\in C(Y)$.
So
$\al \in (0,1]$ in the above definitions of $F_{USC}(Y)$ and $F^1_{USC}(Y)$ can be replaced
by $\al \in [0,1]$.

In this paper, we mainly discuss
  the normal and upper semi-continuous fuzzy sets in a metric space.
Some of the discussion will also involve more general fuzzy sets.
Define
\begin{gather*}
F_{USCB}(X):=\{ u\in  F(X): [u]_\al \in K(X)\cup\{\emptyset\}  \ \mbox{for all} \   \al\in [0,1] \},
\\
F_{USCG}(X):=\{ u\in  F(X): [u]_\al \in K(X)\cup\{\emptyset\} \ \mbox{for all} \   \al\in (0,1] \},
\\
F^1_{USCB}(X):= F^1(X)\cap F_{USCB}(X)=\{ u\in  F(X): [u]_\al \in K(X)\ \mbox{for all} \   \al\in [0,1] \},
\\
F^1_{USCG}(X):= F^1(X)\cap F_{USCG}(X)=\{ u\in  F(X): [u]_\al \in K(X) \ \mbox{for all} \   \al\in (0,1] \}.
 \end{gather*}
Clearly
$F_{USCB}(X) = \{ u\in  F_{USC}(X): [u]_0 \in K(X)\cup\{\emptyset\}  \}$
(``$\subseteq$'' is obvious.
Let $u\in  F_{USC}(X)$ with $[u]_0 \in K(X)\cup\{\emptyset\} $.
Then for each $\al\in (0,1]$,
$[u]_\al$ is a closed subset of the compact set $[u]_0$,
and therefore $[u]_\al \in K(X)\cup\{\emptyset\} $. Thus
 $u\in  F_{USCB}(X)$.
Hence
 ``$\supseteq$'' holds.).
So $F^1_{USCB}(X) =F^1(X) \cap F_{USCB}(X) = \{ u\in  F^1_{USC}(X): [u]_0 \in K(X) \}$.
Obviously
$F_{USCB}(X)\subseteq F_{USCG}(X) \subseteq F_{USC}(X)$.

The supremum metric $d_\infty$,
the endograph metric $H_{\rm end} $ and the sendograph metric $H_{\rm send} $
on $F^1_{USC}(X)$ are defined as follows.
Kloeden \cite{kloeden2} introduced the endograph metric $H_{\rm end}$.
For each $u,v \in F^1_{USC}(X)$,
\begin{align*}
\bm{d_\infty (u,v)} &:=   \sup\{ H([u]_\al, [v]_\al) :\al\in [0,1]  \},
\\
\bm{  H_{\rm end}(u,v)    } &: =  H({\rm end}\, u,  {\rm end}\, v ),
\\
\bm{  H_{\rm send}(u,v)    } &: =  H({\rm send}\, u,  {\rm send}\, v ),
  \end{align*}
where the first $H$ refers to
the Hausdorff
metric on $C(X)$ and the other two
 $H$s refer to
the Hausdorff
metric on $C(X \times [0,1])$ induced by $\overline{d}$ on $X \times [0,1]$.
Clearly
\begin{equation}\label{bre}
  \mbox{ for each $u,v \in F^1_{USC} (X)$,
$
  d_\infty(u,v) \geq H_{\rm send}(u,v) \geq H_{\rm end}(u,v).
$}
\end{equation}

\begin{re}
  {\rm

We can see that
each one of
$d_\infty$ and $H_{\rm send}$ on $F^1_{USC}(X)$ is an extended metric but does not need to be a metric.
 Both $d_\infty$
and
$H_{\rm send}$ are metrics on $F^1_{USCB}(X)$.
By \eqref{bre} and
Example \ref{pefium},
the $d_\infty$ metric and $H_{\rm send}$ metric on $F^1_{USCG}(\mathbb{R}^m)$ could take
the value $+\infty$.
So both $d_\infty$ and $H_{\rm send}$ on $F^1_{USCG}(\mathbb{R}^m)$ are not metrics, they are extended metrics.
 $H_{\rm end}$ is a metric on $F^1_{USC}(X)$ with $H_{\rm end}(u,v) \leq 1$ for all $u,v \in F^1_{USC}(X)$.
See also
Remark 3.3 in \cite{huang17} (We made a misprint in the last sentence of Remark 3.3 in \cite{huang17}. The ``$H_{\rm end}$'' must be deleted from this sentence).

For simplicity,
in this paper, we call
$H_{\rm send}$ on $F^1_{USC}(X)$ the $H_{\rm send}$ metric or the sendograph metric $H_{\rm send}$,
and call
$d_\infty$ on $F^1_{USC}(X)$ the $d_\infty$ metric or the supremum metric $d_\infty$.

}
\end{re}

We will often use the following known facts directly without quoting.
\\
Fact 1: Let $f$ be a function from an interval $[\mu,\nu]$ to $\mathbb{R}^+ \cup \{0\}$
and $p>0$.
Then the following conditions are equivalent to each other: (\rmn1) $f$ is measurable on $[\mu,\nu]$,
(\rmn2) $f^p$ is measurable on $[\mu,\nu]$, and (\rmn3) $(\int_\mu^\nu {f(\al)}^p \,d\al)^{1/p}$
is well-defined.
\\
Fact 2: Let $p>0$. For each $i=1,\ldots, n$, let $f_i$ be functions from an interval $[\mu_i,\nu_i]$ to $\mathbb{R}^+ \cup \{0\}$.
Then
the following conditions are equivalent to each other: (\rmn1) for each $i=1,\ldots, n$, $f_i$ is measurable on $[\mu_i,\nu_i]$,
(\rmn2) for each $i=1,\ldots, n$, $(\int_{\mu_i}^{\nu_i} {f_i (\al)}^p \,d\al)^{1/p}$
is well-defined,
(\rmn3) $\sum_{i=1}^n (\int_{\mu_i}^{\nu_i} {f_i (\al)}^p \,d\al)^{1/p}$
is well-defined.
(By Fact 1, (\rmn1)$\Leftrightarrow$(\rmn2).
Note that for each $i=1,\ldots, n$, if
$(\int_\mu^\nu {f_i(\al)}^p \,d\al)^{1/p}$
is well-defined, then $(\int_\mu^\nu {f_i(\al)}^p \,d\al)^{1/p}\geq 0$.
So (\rmn2)$\Rightarrow$(\rmn3).
(\rmn3)$\Rightarrow$(\rmn2) is obvious.
So (\rmn1), (\rmn2) and (\rmn3) are equivalent to each other.)

For each $u,v\in F^1_{USC}(X)$, the $d_p$ distance of $u,v$ given by
$$\bm{d_p(u,v)}= \left(\int_0^1 H([u]_\al, [v]_\al)^p  \,   d\al   \right)^{1/p}$$
 is well-defined if and only if
$H([u]_\al, [v]_\al)$ is a measurable function of $\al$ on $[0, 1]$.
Here we see
 $H([u]_\al, [v]_\al)$ as a function of $[0,1]$ into $\widehat{\mathbb{R}}$ with $\al$ as the independent variable.
We suppose that, in the sequel, ``$p$'', which appears in the mathematical expressions such as $d_p$, etc., satisfies
 $p \geq 1$.

In Section \ref{meau}, we will give some fundamental conclusions for the measurability of the function
$H([u]_\al, [v]_\al)$ of $\al$ on $[0,1]$.
Since $H([u]_\al, [v]_\al)$ could be a non-measurable function
of $\al$ on
$[0,1]$ (see Example \ref{nmf}),
we introduce the
 $d_p^*$ distance on $F^1_{USC} (X)$, $p\geq 1$, in
\cite{huang17},
which is defined by
\begin{tiny}
 \begin{gather*}
\bm{d_p^*(u,v) }:= \inf \left\{  \  \left(\int_0^1 f(\al)^p  \,   d\al   \right)^{1/p}
 :
\ f \mbox{ is a measurable function from } [0,1] \mbox{ to } \mathbb{R}^+ \cup \{+\infty\}    \mbox{ satisfying }
  f(\al)  \geq H([u]_\al, [v]_\al) \mbox{ for all } \al\in [0,1]  \  \right\}
\end{gather*}
\end{tiny}
for each
$u,v \in F^1_{USC} (X)$.

Conclusions (\rmn1) and (\rmn2)
below have been shown
in \cite{huang17} (see Theorem 3.1 and
Remarks 3.2 and 3.3 in \cite{huang17}.).
\\
(\rmn1)
$d_p^*$ on $F^1_{USC}(X)$ is an extended metric but does not need to be a metric.
\\
(\rmn2)
For each $u,v \in F^1_{USC} (X)$,
if $d_p(u,v)$ is well-defined then clearly
$d_p^*(u,v) = d_p(u,v)$.
So the $d_p^*$ distance is
 an expansion of the $d_p$ distance on
$F^1_{USC} (X)$.

The $d_p$ distance is well-defined on $F^1_{USC}(\mathbb{R}^m)$ (see Theorem \ref{rcm}).
The
 $d_p$ distance is well-defined on $F^1_{USCG}(X)$ (see Section 6 of \cite{huang19c} or Proposition \ref{lcpm}).
In Section \ref{meau}, we also give further results on the measurability of the function
$H([u]_\al, [v]_\al)$ of $\al$ on $[0,1]$.

\begin{re}
{\rm

  In this paper, for simplicity,
 we also call the $d_p^*$ distance as the
$d_p^*$ metric when the $d_p^*$ distance is an extended metric,
and
also call the $d_p$
distance as the
$d_p$ metric when the $d_p$ distance is an extended metric.

}
\end{re}

In \cite{huang19c},
we showed that for each $u,v\in F^1_{USC}(X)$,
$
  d_\infty(u,v) \geq d_p^*(u,v).
$

The following example shows that
the $d_p$ metric and $H_{\rm send}$ metric on $F^1_{USCG}(\mathbb{R}^m)$ could take
the value $+\infty$. This example comes from Section 6 of \cite{huang19c}.
The symbol $\|\cdot\|$ is used to denote the Euclidean norm on $\mathbb{R}^m$.

\begin{eap} \label{pefium}
  {\rm
Denote the origin of $\mathbb{R}^m$ by $o$.
Clearly $\widehat{o}_{F(\mathbb{R}^m)} \in F^1_{USCB}(\mathbb{R}^m) \subsetneqq F^1_{USCG}(\mathbb{R}^m)$.
Let $u\in F^1_{USCG}(\mathbb{R}^m)$ be defined by putting
$
[u]_\al=
\{x\in \mathbb{R}^m: \|x\|\leq n\} \mbox{ for each } n\in \mathbb{N} \mbox{ and each } \al\in (1/(n+1), 1/n].
$
Then
$d_p(u, \widehat{o}_{F(\mathbb{R}^m)}) = + \infty$
(this fact has been given in \cite{huang19c}).
Given $n\in \mathbb{N}$.
Pick an
 $x_n \in \mathbb{R}^m$ with $\|x_n\|=n$.
Then $x_n\in [u]_{1/n}$. This means that $(x_n, 1/n) \in {\rm send}\, u$.
Note that $\overline{\rho_m}((x_n, 1/n), \, {\rm send}\, \widehat{o}_{F(\mathbb{R}^m)}) = \overline{\rho_m}((x_n, 1/n), \, \{o \}\times [0,1])=n$.
Thus
$+\infty\geq H^*({\rm send}\, u, {\rm send}\,\widehat{o}_{F(\mathbb{R}^m)}) \geq \sup_{n=1}^{+\infty} \overline{\rho_m}((x_n, 1/n), \, \{o \}\times [0,1])=+\infty$.
Hence $H^*({\rm send}\, u, {\rm send}\,\widehat{o}_{F(\mathbb{R}^m)})= +\infty$,
and so
$H_{\rm send} (u, \widehat{o}_{F(\mathbb{R}^m)}) =+\infty$.
}
\end{eap}

We introduce the following subset of $ F^1_{USC}(X)$
\begin{itemize}
  \item
 $F^1_{USCG} (X)^p  :=  \{ u\in  F^1_{USCG}(X):   d_p(u, \widehat{x_0} ) = (\int_0^1 H([u]_\al, \{x_0\}   )^p  \,   d\al)^{1/p} <  +\infty     \}$,
where $x_0$ is a point in $X$.
\end{itemize}
 Let $u\in  F^1_{USCG}(X)$. If there is an $x\in X$ such that $ d_p(u, \widehat{x}) < +\infty$, then for each $y\in X$,
$d_p(u, \widehat{y}) < +\infty$ since
$d_p(u, \widehat{y}) \leq d_p(u, \widehat{x}) + d_p(\widehat{x}, \widehat{y}) = d_p(u, \widehat{x}) + d(x,y) < +\infty$. So the definition of $F^1_{USCG} (X)^p$ does not depend on the choice of $x_0$. And
$F^1_{USCG} (X)^p  =  \{ u\in  F^1_{USCG}(X):   d_p(u, \widehat{x} )  <  +\infty   \mbox{ for each } x\in X  \}$.

Let $x\in X$.
Clearly, for each $u,v\in F^1_{USCG} (X)^p $, $d_p(u,v) \leq  d_p(u, \widehat{x}) + d_p(v, \widehat{x}) < +\infty$.
So the
$d_p$ distance is a metric on $F^1_{USCG} (X)^p$.

\begin{re}\label{dfu}
  {\rm
Let $x\in X$. Then for each $\al\in [0,1]$, $[\widehat{x}]_\al = \{x\} \in K(X)$. Thus $\widehat{x} \in F^1_{USCB}(X) \subseteq F^1_{USCG}(X)$.
Note that the
 $d_p$ distance is well-defined on $F^1_{USCG}(X)$ (see Section 6 of \cite{huang19c} or Proposition \ref{lcpm}).
So $d_p(u, \widehat{x_0} )$ in the definition of $F^1_{USCG} (X)^p$ is well-defined.
This fact also follows from Theorem \ref{cme}
as $u\in  F^1_{USCG}(X)\subseteq F^1_{USC}(X)$ and $\widehat{x_0} \in F^1_{USCG}(X)$.
This fact also follows from Proposition \ref{gmn}.

}
\end{re}

Define $\widehat{X}:=\{\widehat{x}: x\in X\}$. Clearly
$$
\widehat{X}\subseteq F^1_{USCB}(X) \subseteq  F^1_{USCG} (X)^p   \subseteq  F^1_{USCG}(X)   \subseteq  F^1_{USC}(X).
$$

In this paper both the $\lim$ quantities and the $\liminf$ (see Page \pageref{limf}) quantities are allowed to take the values $-\infty$ and $+\infty$.

\section{Measurability of function $H([u]_\al, [v]_\al)$} \label{meau}

For each $u,v\in F^1_{USC}(X)$,
$d_p(u,v)$
 is well-defined if and only if
$H([u]_\al, [v]_\al)$ is a measurable function of $\al$ on $[0, 1]$.
So it is important to discuss the measurability of the function
$H([u]_\al, [v]_\al)$ of $\al$ on $[0,1]$.
In this section, we give some fundamental conclusions on this topic.

If not specifically mentioned,
we
uniformly use $H$ to denote
the Hausdorff
metric on $C(Y)$ induced by $d_Y$ on $Y$, where $(Y, d_Y)$ is a certain metric space.
The meaning of
$H$ can be judged
according to the context.

We have pointed out
 the following statements (\rmn1)--(\rmn4) on the measurability
of the function
$H([u]_\al, [v]_\al  )$ in \cite{huang17} (See also \cite{huang9}, which was submitted on 2019.07.06).
\\
(\romannumeral1)
  For $u \in  F^1_{USC}(X)$ and $x_0 \in X$,  $H([u]_\al, \{x_0\}) $ is a measurable function of $\al$ on $[0,1]$.
\\
(\romannumeral2)   For    $u,v \in  F^1_{USC} (\mathbb{R}^m)$, $H([u]_\al, [v]_\al) $ is a measurable function of $\al$ on $[0,1]$.
\\
(\romannumeral3) For $u,v \in  F^1_{USCG}(X)$, $H([u]_\al, [v]_\al) $ is a measurable function of $\al$ on $[0,1]$.
\\
(\romannumeral4)  There exists a metric space $X$ and $u,v\in F^1_{USC}(X)$ such that $H([u]_\al, [v]_\al) $ is a non-measurable function
of $\al$ on
$[0,1]$.
\\
Statement (\rmn3) was shown in \cite{huang19c} (see Proposition \ref{lcpm}), which is an easy corollary of a conclusion in \cite{huang17} (see Remark \ref{uen}).
We prove the other three statements in this section.
 Statement (\rmn1) is Proposition \ref{gmn}.
Statement (\rmn2) is Theorem \ref{rcm}.
Example \ref{nmf} shows statement (\rmn4).
We submitted the proofs of the first three statements in \cite{huang31}.
Further,
we give some improvements of these statements, which are listed below.
\\
(\rmn5) The function
$H([u]_\al, [v]_\al)$ of $\al$ on $[0,1]$ is measurable when $u\in F^1_{USC} (X)$ and $v\in F^1_{USCG} (X)$
(Theorem \ref{cme}). Theorem \ref{cmeg} is an improvement
of
  Theorem \ref{cme}.
\\
(\rmn6)
Theorem \ref{regn} is an improvement of Theorem \ref{reg}.
Theorem \ref{reg} is an improvement of
Theorem \ref{rcm}.

We say a function $f: [0,1] \to \widehat{\mathbb{R}}$ is decreasing
if
$f(x) \geq f(y)$ when $x < y$.

\begin{pp}  \label{gmn}
   Let $u \in  F^1_{USC}(X)$ and $x_0 \in X$. Then $H([u]_\al, \{x_0\}) $
 is a decreasing function of $\al$ on $[0,1]$.
So $H([u]_\al, \{x_0\}) $
is a measurable function of $\al$ on $[0,1]$, which is equivalent to that $d_p(u, \widehat{x_0})$ is well-defined.
\end{pp}

\begin{proof}

Let $\al,\beta$ with $0\leq \al \leq \beta \leq 1$. Then $[u]_\al\supseteq [u]_\beta$ and hence
$$H([u]_\al, \{x_0\}) =\sup_{x\in [u]_\al} d(x, x_0)\geq \sup_{x\in [u]_\beta} d(x, x_0) = H([u]_\beta, \{x_0\}).$$
By (a) in Remark \ref{bsm}, the above two ``='' hold.
  So $H([u]_\al, \{x_0\}) $ is a decreasing function of $\al$ on $[0,1]$.
  Thus
$H([u]_\al, \{x_0\}) $ is a measurable function of $\al$ on $[0,1]$.

$d_p(u, \widehat{x_0})$ is well-defined if and only if
$H([u]_\al, [\widehat{x_0}]_{\al})$ is a
 measurable function of $\al$ on $[0,1]$.
Clearly for each $\al\in[0,1]$, $H([u]_\al, [\widehat{x_0}]_{\al}) = H([u]_\al, \{x_0\}) $.
So
$d_p(u, \widehat{x_0})$ is well-defined if and only if
$H([u]_\al, \{x_0\})$ is a
 measurable function of $\al$ on $[0,1]$.
\end{proof}

\begin{re}\label{bsm}
{\rm
It is easy to see
\\
(a) Let $(X,d)$ be a metric space. Let $A\in C(X)$ and $x\in X$. Then $H(A, \{x\}) = H^*(A, \{x\}) = \sup_{y\in A} d(y,x)$.

Observe that
\begin{gather*}
  H^*(A, \{x\})= \sup_{y\in A} d(y,\{x\}) = \sup_{y\in A} d(y,x)=\sup_{y\in A} d(x,y), \\
  H^*(\{x\}, A) =  d(x, A) =\inf_{y\in A} d(x,y),
\end{gather*}
 so
$H^*(A, \{x\})\geq H^*(\{x\}, A)$.
Hence
 $H(A, \{x\}) = \max\{H^*(A, \{x\}),   H^*(\{x\}, A) \}= H^*(A, \{x\}) = \sup_{y\in A} d(y,x)$. Thus (a) is proved.

}
\end{re}

A set $S$ is said to be an \emph{interval} if it is a set in $\mathbb{R}$
which has the following property: if $x,y\in S$ satisfying $x<y$,
then $[x,y]\in S$. Clearly
$\emptyset$ and singleton sets in $\mathbb{R}$ are intervals.
For each $r\in \mathbb{R}$, $(r,r]=[r,r)=\emptyset$.

Let $\widehat{\mathbb{R}} = \mathbb{R} \cup \{-\infty\} \cup \{+\infty\}$.
Let $f$ be a function from an interval $[\mu, \nu]$ to $\widehat{\mathbb{R}}$
and
$\al\in (\mu, \nu]$.
\bm{$\liminf_{\gamma\to \al-} f(\gamma)$} \label{limf} is defined by
\begin{gather*}
       \liminf_{\gamma\to \al-} f(\gamma) :=\inf S_{f,\al}, \mbox{ where}
 \\               \bm{S_{f,\al}}= \{x\in \widehat{\mathbb{R}}:  \mbox{ there is a sequence } \{\gamma_n\} \mbox{ in } [\mu, \nu]
 \mbox{ such that } \gamma_n\to\al-  \mbox{ and } x= \lim_{n\to +\infty}  f(\gamma_n) \}.
\end{gather*}

\begin{re}\label{lfe}
  {\rm
(\rmn1) $S_{f,\al}\not= \emptyset$ and so
 $\liminf_{\gamma\to \al-} f(\gamma) $ is well-defined;
\\
(\rmn2)
 $\liminf_{\gamma\to \al-} f(\gamma) $
 could take the values $\pm\infty$;
\\
(\rmn3)
if $\lim_{\gamma\to \al-} f(\gamma) $ exists, then
$\lim_{\gamma\to \al-} f(\gamma) = \liminf_{\gamma\to \al-} f(\gamma) $;
\\
(\rmn4) $\liminf_{\gamma\to \al-} f(\gamma) = \min S_{f,\al}$.

(\rmn1)--(\rmn4) should be known. (\rmn3) is obvious. For the completeness of this paper, we give the proofs
or illustration
of (\rmn1), (\rmn2) and (\rmn4) in the following.

To show $S_{f,\al} \not= \emptyset$,
pick a sequence $\{\gamma_n\}$ in $[\mu, \nu]$ with $\gamma_n\to \al-$.

If $\{f(\gamma_n)\}$ is a bounded set,
 then there is an $x_0\in \mathbb{R}$ and a subsequence $\{f(\gamma_{n_k})\}$ of $\{f(\gamma_n)\}$
 such that $x_0 = \lim_{k\to +\infty}  f(\gamma_{n_k})$.
 Obviously
 $\gamma_{n_k}\to \al-$ as $k\to +\infty$. Thus
  $x_0\in S_{f,\al}$. Hence $S_{f,\al}\not=\emptyset$.

If $\{f(\gamma_n)\}$ is an unbounded set,
 then there is a subsequence $\{f(\gamma_{n_k})\}$ of $\{f(\gamma_n)\}$
 such that
 $+\infty= \lim_{k\to +\infty}  f(\gamma_{n_k})$
 or $-\infty= \lim_{k\to +\infty}  f(\gamma_{n_k})$.
Clearly
 $\gamma_{n_k}\to \al-$ as $k\to +\infty$. Thus
 $+\infty$ or $-\infty$ belongs to $S_{f,\al}$. Hence $S_{f,\al}\not=\emptyset$.
   So (\rmn1) is true.

The following simple examples illustrate (\rmn2).

  Consider $f: [0,1]\to \widehat{\mathbb{R}}$ defined by $f(y)=+\infty$ (respectively, $-\infty$), for all $y\in [0,1]$. Then $\liminf_{\gamma\to\al-} f(\gamma) = \lim_{\gamma\to\al-} f(\gamma)= +\infty$ (respectively, $-\infty$) for each $\al\in (0,1]$.

   Consider $f: [0,1]\to \widehat{\mathbb{R}}$ defined by $f(1)=0$, and $f(y)=\frac{1}{1-y}$ (respectively, $\frac{1}{y-1}$), for all $y\in [0,1)$. Then $\liminf_{\gamma\to 1-} f(\gamma) =\lim_{\gamma\to 1-} f(\gamma) =  +\infty$ (respectively, $-\infty$).
So (\rmn2) is true.

To show (\rmn4),
set $z:=\liminf_{\gamma\to \al-} f(\gamma)$.
It suffices to show that
$z\in S_{f,\al}$.

Case 1. $z=+\infty$. In this case, obviously $z\in S_{f,\al}$.

Case 2. $z\in \mathbb{R}$. Given $n\in \mathbb{N}$. Then there is a $\xi_n\in [z-\frac{1}{n},z]\cap S_{f,\al}$.
As $\xi_n\in S_{f,\al}$, it follows that
 there exists a $\gamma_n$ such that $\gamma_n\in (\al-\frac{1}{n}, \al)\cap [\mu, \nu]$
 and $|f(\gamma_n) - \xi_n|< \frac{1}{n}$.
 Thus
 $|f(\gamma_n) - z| \leq |f(\gamma_n) - \xi_n| + |\xi_n-z|< \frac{2}{n}$.
Hence the sequence $\{\gamma_n\}$ in $[\mu, \nu]$ satisfies that
 $\gamma_n \to \al-$ as $n\to +\infty$
 and $z= \lim_{n\to +\infty}  f(\gamma_n)$.
 So $z\in S_{f,\al}$.

  Case 3. $z=-\infty$. Given $n\in \mathbb{N}$.
  Then there is a $\xi_n\in (\{-\infty\} \cup (-\infty,-n])\cap S_{f,\al}$.
Since $\xi_n\in S_{f,\al}$,
we have that there exists a $\gamma_n$ such that $\gamma_n\in (\al-\frac{1}{n}, \al)\cap [\mu,\nu]$
 and $f(\gamma_n) < (-n) + \frac{1}{n}$.
 Thus the sequence $\{\gamma_n\}$ in $[\mu,\nu]$ satisfies that
 $\gamma_n \to \al-$ as $n\to +\infty$
and
 $f(\gamma_n) \to -\infty$ as $n\to +\infty$.
Hence $z=-\infty\in S_{f,\al}$. So (\rmn4) is true.

}
\end{re}

Let
 $f$ be a function from an interval $[\mu, \nu]$ to $\widehat{\mathbb{R}}$
 and
  $\al\in (\mu, \nu]$. We say that $f$ is \textbf{\emph{left lower semicontinuous at}}
\bm{$\al$} if
   $f(\al) \leq \liminf_{\gamma\to \al-} f(\gamma) $.
We say that $f$ is \textbf{\emph{left lower semicontinuous on}}
\bm{$(\mu,\nu]$} if $f$ is left lower semicontinuous at each
$\xi \in (\mu,\nu]$.
Let $r\in \mathbb{R}$.
The symbol $\bm{\{f>r\}}$ is used to denote the set $ \{\xi\in [\mu, \nu]: f(\xi) > r\}$.

\begin{lm} \label{fcm}
Let
 $f$ be a function from an interval $[\mu, \nu]$ to $\widehat{\mathbb{R}}$
 and
  $\al\in (\mu, \nu]$.
Then the following properties are equivalent:
\\
(\romannumeral1) \
  For each $r\in \mathbb{R}$, if $\al\in \{f > r\}  $,
then there exists $\delta (\al) > 0$
such that $[\al-\delta(\al), \al]  \subseteq \{ f > r \}$;
\\
(\romannumeral2) \   $f$ is left lower semicontinuous at
$\al$; that is, $f(\al) \leq \liminf_{\gamma\to \al-} f(\gamma) $.
\end{lm}

\begin{proof}
The desired result may be known.
Here we give a proof for the completeness of this paper.
 Let $r\in \mathbb{R}$.
 Obviously $f(\al) > r$ means $\al\in \{f>r\}$.

Assume that (\romannumeral1) is true.
If $f(\al) = -\infty$, then (\romannumeral2) is true.
Now suppose that $f(\al) > -\infty$.
Given $r\in \mathbb{R}$ with $f(\al) > r$.
Then by (\rmn1),
 there exists $\delta (\al) > 0$
such that $[\al-\delta(\al), \al]  \subseteq \{f>r\}$,
thus
$\liminf_{\gamma\to \al-} f(\gamma) \geq r$.
Since $r\in (-\infty, f(\al))$ is arbitrary, we conclude that $ \liminf_{\gamma\to \al-} f(\gamma)\geq f(\al) $.
So (\romannumeral2) is true.

Assume that (\romannumeral2) is true. Let
 $r\in \mathbb{R}$. If
$\al\in \{f>r\}  $, then $\liminf_{\gamma\to \al-} f(\gamma) >r $.
To show (\rmn1) is true, we only need to prove
\\
 (a)
there exists $\delta (\al) > 0$
such that $[\al-\delta(\al), \al]  \subseteq \{f>r\}$.

Suppose that (a) is not true; that is, for each $\delta < \al$,
 $[\delta, \al]  \nsubseteq \{f>r\}$.
 Let $n\in \mathbb{N}$.
Then $(\al-\frac{1}{n})\vee \mu <\al$. Thus $[(\al-\frac{1}{n})\vee \mu, \al]  \nsubseteq \{f>r\}$. This means that
there is an $\al_n$ such that $\al_n\in[(\al-\frac{1}{n})\vee \mu, \al] $
but  $\al_n\notin\{f>r\}$.
We can see that the sequence $\{\al_n\}$ in $[\mu, \nu]$
satisfies that $\al_n \to \al-$ and $f(\al_n) \leq r $ for all $n\in \mathbb{N}$.

Case 1. $\{f(\al_n)\}$ is bounded. Then there is a subsequence $\{f(\al_{n_k})\}$ of $\{f(\al_n)\}$ and an $x_0\in \mathbb{R}$
such that
$x_0=\lim_{k\to\infty} f(\al_{n_k})$.
Since $\al_{n_k} \to \al-$ as $k\to\infty$, we have that
  $x_0\in S_{f,\al}$.
Clearly $x_0\leq r$.
Thus
 $\liminf_{\gamma\to \al-} f(\gamma) \leq r $. This
contradicts $\liminf_{\gamma\to \al-} f(\gamma) >r$.

Case 2.
$\{f(\al_n)\}$ is unbounded. Then $\{f(\al_n)\}$ has a subsequence $\{f(\al_{n_k})\}$
such that $-\infty=\lim_{k\to\infty} f(\al_{n_k})$.
Obviously $\al_{n_k} \to \al-$ as $k\to\infty$. So
  $-\infty\in S_{f,\al}$.
Thus
 $\liminf_{\gamma\to \al-} f(\gamma) = -\infty$. This
contradicts $\liminf_{\gamma\to \al-} f(\gamma) >r$.

As there are contradictions in both cases, it follows that (a) is true.
So (\rmn1) is true. This completes the proof.

\end{proof}

\begin{pp} \label{rsf}
Let
 $f$ be a function from an interval $[\mu, \nu]$ to $\widehat{\mathbb{R}}$.
If $f$ is left lower semicontinuous on
$(\mu,\nu]$; that is,
   $f(\al) \leq \liminf_{\gamma\to \al-} f(\gamma) $ for all $\al\in (\mu, \nu]$,
then $f $ is a measurable function on $[\mu, \nu]$.
\end{pp}

\begin{proof}

To show the desired result,
we only need to show
that
for each $r\in \mathbb{R}$, the set $\{f>r\}$ is a measurable set.
We split the proof into three steps.

\textbf{Step 1} \ Prove statement (\romannumeral1) For each $r\in \mathbb{R}$,
if $\al>\mu$ and $\al\in \{f>r\}  $,
then there exists $\delta (\al) > 0$
such that $[\al-\delta(\al), \al]  \subseteq \{f>r\}$.

By Lemma \ref{fcm}, statement (\rmn1) is equivalent to the statement
that
  $f(\al) \leq \liminf_{\gamma\to \al-} f(\gamma) $ for all $\al\in (\mu, \nu]$, which is our assumption. So statement (\rmn1) is true.

\textbf{Step 2} \ Prove statement (\romannumeral2)
For each $r\in \mathbb{R}$ with $\{f>r\}\setminus \{\mu\}  \not= \emptyset$,
 $\{f > r\}\setminus \{\mu\}  $ is a union of disjoint
positive length intervals.

For each $x \in \{f>r\}\setminus \{\mu\}   $,
 define $\overbrace{x} = \bigcup \{A : A\in S_x\}$,
 where
 $S_x = \{ A: A \mbox{ is an interval with } x \in A \subseteq \{f>r\}\setminus \{\mu\}   \}$.
 We claim that
 \\
 (a) For each $x \in \{f>r\}\setminus \{\mu\}$, $\overbrace{x}$ is the largest interval in $S_x$;
 \\
  (b) For each $x \in \{f>r\}\setminus \{\mu\}$,
 $\overbrace{x}$ is a positive length interval;
 \\
 (c) For each
 $x,y \in \{f>r\}\setminus \{\mu\} $ with $\overbrace{x} \cap \overbrace{y} \not= \emptyset$, $\overbrace{x} = \overbrace{y}$.

First, we show (a).
Clearly $\overbrace{x}$ is an interval
(To show this,
 let $y,z\in \overbrace{x}$ with $y< z$.
It suffices to
 verify $[y,z]\subseteq \overbrace{x}$.
Choose $A_1$ and $A_2$ in $S_x$
 such that $y\in A_1$ and $z\in A_2$.
As $x\in A_1\cap A_2\not=\emptyset$,
$A_1\cup A_2$ is an interval.
Observe that
 $A_1\cup A_2\in S_x$, so $[y,z]\subseteq A_1\cup A_2 \subseteq \overbrace{x}$.
  Thus $\overbrace{x}$ is an interval.).
Obviously
 $[x,x]\in S_x$, i.e., $x\in \overbrace{x}$. Also $ \overbrace{x} \subseteq\{f>r\}\setminus \{\mu\}  $.
So $\overbrace{x}\in S_x$. Then from the definition of $\overbrace{x}$,
(a) is true.

Next we show (b). By statement (\romannumeral1), there is a $\delta(x)>0$ such that $[x-\delta(x),x]\subseteq \{f>r\}$.
Note that $\mu \leq x-\delta(x)$.
Thus $x\in (x-\delta(x),x]\subseteq \{f>r\} \setminus \{\mu\}$, i.e., $(x-\delta(x),x]\in S_x$.
Hence $(x-\delta(x), x] \subseteq \overbrace{x}$. So (b) is true.

Now we show (c). Observe that
 $\overbrace{x} \cup \overbrace{y} $ is an interval
 with $\{x,y\}\subseteq \overbrace{x} \cup \overbrace{y} \subseteq \{f>r\}\setminus \{\mu\}$. In other words,
 $\overbrace{x} \cup \overbrace{y}  \in S_x\cap S_y$.
 By (a),
this means that
 $\overbrace{x} =\overbrace{x} \cup \overbrace{y} = \overbrace{y}$.
So (c) is true.

 As $\{f>r\}\setminus \{\mu\} = \bigcup\{\overbrace{x}: x \in \{f>r\}\setminus \{\mu\} \}$,
it follows
from (b) and (c) that statement (\rmn2) is true.

From the proof of (a) and (c), it can be seen that (a) and (c) are true regardless of whether
$f$ is left lower semicontinuous on
$(\mu,\nu]$ or not.
Let $x,y\in \{f>r\}\setminus \{\mu\}$.
If $f$ is not left lower semicontinuous on
$(\mu,\nu]$,
then $\overbrace{x}$ and $\overbrace{y}$ may both be singleton sets, and
in this case, $\{x,y\}= \overbrace{x} \cup \overbrace{y} $.
If $f$ is left lower semicontinuous on
$(\mu,\nu]$,
then of course $\{x,y\}\subsetneqq \overbrace{x} \cup \overbrace{y}$.

\textbf{Step 3} \ Prove statement (\romannumeral3) For each $r\in \mathbb{R}$, $\{f>r\}$ is a measurable set.

Note that a set $S$ of disjoint positive length intervals is at most countable (For each $B\in S$,
choose a $q_B$ in $B\cap \mathbb{Q}$. Then $q_{B_1} \not= q_{B_2}$ for each $B_1$ and $B_2$ in $S$ with $B_1\not= B_2$. Thus $\overline{\overline{S}}= \overline{\overline{\{q_B: B\in S\}}}\leq \overline{\overline{\mathbb{Q}}}$ and so $S$ is at most countable, where $\overline{\overline{C}}$ denotes the cardinality of a set $C$.).
So
by statement (\romannumeral2), $\{f>r\} \setminus \{\mu\}$
 is an at most countable union of intervals.
  Hence $\{f>r\} \setminus \{\mu\}$ is a measurable set.
 As $\{f>r\}=\{f>r\} \setminus \{\mu\}$ or $\{f>r\}=(\{f>r\} \setminus \{\mu\})\cup \{\mu\}$,
it follows that
$\{f>r\}$ is a measurable set.

  Here we mention that if
  ($\{f>r\} \setminus \{\mu\}= \emptyset$, then it is an at most countable union of intervals as it
 can be seen as an empty union of intervals.
Of course, if $\{f>r\} \setminus \{\mu\}= \emptyset$, then obviously $\{f>r\} \setminus \{\mu\}$ is measurable.

\end{proof}

\begin{re}\label{ref}
  {\rm   Let $f$ be a function from an interval $[\mu, \nu]$ to $\widehat{\mathbb{R}}$, $r\in \mathbb{R}$,
 and $\{f>r\}\setminus \{\mu\} \not= \emptyset$.

(\rmn1) Let $f$ be left lower semicontinuous on $(\mu, \nu]$. Then
 (\rmn1-a) for each
 $x \in \{f>r\}\setminus \{\mu\}$,
 $\overbrace{x}$ is an interval $(b,c]$ or $(b,c)$ for some $b<c$;
  (\rmn1-b) setting $\lambda:=\inf\{y: y\in \{f>r\}\setminus \{\mu\}\}$, we have
  $ \lambda \notin \{f>r\}\setminus \{\mu\}$.

   Set $a:=\inf\overbrace{x}$. If $a\in \overbrace{x}$, then
  by Lemma \ref{fcm}, there is a $\delta>0$ with $(a-\delta, a]\subseteq \overbrace{x}$, which contradicts $a=\inf\overbrace{x}$. Thus $a\notin \overbrace{x}$.
Putting together this fact and below simple example we see that
  (\rmn1-a) is true.
The proof of (\rmn1-b) is similar to that of $a\notin \overbrace{x}$.
  (\rmn1-b) can also be seen as an easy consequence of (\rmn1-a).
  If $\lambda\in \{f>r\}\setminus \{\mu\}$, then by (\rmn1-a),
  $\overbrace{\lambda}$ is an interval
  $(\lambda_b, c]$ or $(\lambda_b, c)$.
  Clearly $\lambda_b<\lambda$.
  So $(\lambda_b,\lambda]\subseteq \{f>r\}\setminus \{\mu\}$. This contradicts the definition of $\lambda$.
  Thus (\rmn1-b) is true.

 Consider
   $f: [0,1] \to [0,1]$ defined by $f(y)=1$ for all $y\in [0,1)$
  and $f(1)=0$. Then $f$ is left lower semicontinuous on $(0,1]$.
  Clearly $\{f>1/2\}\setminus \{0\} = (0,1)$ and
  $\overbrace{\xi} (\mbox{according to } r=1/2)=(0,1)$ for all $\xi\in (0,1)$;
   $\{f>-1\}\setminus \{0\} = (0,1]$ and
  $\overbrace{\xi} (\mbox{according to } r=-1)=(0,1]$ for all $\xi\in (0,1]$.

(\rmn1-a) has essentially been given in Remark 1.3 of
\cite{huang31}.

  (\rmn2)
Essentially, the proof of Proposition \ref{rsf}, except its step 1, has already been given in
the proof of Proposition 6.1 in \cite{huang19c}.
In the proof of Proposition 6.1 in \cite{huang19c},
``$x\in X$'' should be replaced by ``$x\in \{f>r\}\setminus \{0\}$''(we were a little careless);
``$[\al- \delta(\al), \al]\subseteq \{f>r\}\setminus \{0\}$'' can also be written as ``$(\al- \delta(\al), \al]\subseteq \{f>r\}\setminus \{0\}$''
or ``$[\al- \delta(\al), \al]\subsetneqq \{f>r\}\setminus \{0\}$''
(By (\rmn1-b), $[\al- \delta(\al), \al]\not= \{f>r\}\setminus \{0\}$. The example in (\rmn1) shows that $(\al- \delta(\al), \al]= \{f>r\}\setminus \{0\}$ is possible.).
Indeed,
 this modified proof of Proposition 6.1 in \cite{huang19c}
will be a proof for Proposition \ref{rsf} if we
replace $0$ by $\mu$, $1$ by $\nu$, and ``since $f$ is left-continuous at $\al$'' by ``, by Lemma \ref{fcm} and the fact that $f$ is left lower semi continuous at $\al$''.

The following
fact (\rmn2-a) is known. Fact (\rmn2-b)
follows easily from (\rmn2-a).
\\
  (\rmn2-a) Let $A$ be
an interval and $x\in A$. Then $A=\bigcup \{[a,b] : x\in [a,b]\subseteq A\}$.
\\
(\rmn2-b)
For each $x \in \{f>r\}\setminus \{\mu\}   $,
 $\overbrace{x} = \bigcup \{[a,b] : [a,b]\in S_x\}$.

First we show (\rmn2-a).
Set $\xi:=\inf A$ and $\eta:=\sup A$.
For each $n\in \mathbb{N}$, define
\[a_n
=
\left\{
  \begin{array}{ll}
    \xi, & \mbox{if } \xi \in A, \\
 \xi + \frac{1}{n} (x-\xi), & \mbox{if }  \xi \notin A, \ \xi\in \mathbb{R},\\
 x -n, & \mbox{if }  \xi=-\infty,
  \end{array}
\right.
\
b_n
=
\left\{
  \begin{array}{ll}
   \eta, & \mbox{if } \eta \in A, \\
 \eta - \frac{1}{n} (\eta-x), & \mbox{if }  \eta \notin A, \ \eta\in \mathbb{R},\\
 x + n, & \mbox{if }  \eta=+\infty.
  \end{array}
\right.
\]
Then $x\in [a_n, b_n]$ for all $n\in \mathbb{N}$,
and $A= \bigcup_{n=1}^{+\infty} [a_n, b_n]$.
Thus $A\subseteq\bigcup \{[a,b] : x\in [a,b]\subseteq A\}$.
Obviously $A\supseteq\bigcup \{[a,b] : x\in [a,b]\subseteq A\}$.
So (\rmn2-a) is true.

Now we show (\rmn2-b).
Let $B$ be an
 interval in $S_x$. Define $S_B:=\{[a,b] : x\in [a,b]\subseteq B\}$.
Clearly $S_B\subseteq S_x$. By (\rmn2-a),
 $B=\bigcup \{[a,b]: [a,b] \in S_B\}$.
Thus
$\overbrace{x} = \bigcup \{C : C\in S_x\} = \bigcup \{[a,b]: [a,b] \in S_C, C\in S_x\}\subseteq \bigcup \{[a,b] : [a,b]\in S_x\}$.
Obviously
$\overbrace{x} \supseteq \bigcup \{[a,b] : [a,b]\in S_x\}$.
So (\rmn2-b) is true.

 (\rmn3) Suppose that $f$ is a measurable function on $[\mu, \nu]$; that is,
 for each $r\in \mathbb{R}$, the set $\{f>r\}$ is a measurable set. Then $\{\alpha\in [\mu, \nu]: f(\alpha)=+\infty\} = \bigcap_{n=1}^{+\infty}\{f>n\}$ is a measurable set.
 Also $\{\alpha\in [\mu, \nu]: f(\alpha)=-\infty\} = [\mu,v] \setminus \bigcup_{n=1}^{+\infty}\{f> -n\}$ is a measurable set.

  }
\end{re}

Let
 $f$ be a function from an interval $[\mu, \nu]$ to $\widehat{\mathbb{R}}$.
 Let $\al\in (\mu, \nu]$.
$f$ is said to be \textbf{\emph{left continuous at}} \bm{$\al$} if $f(\al) = \lim_{\gamma\to \al-} f(\gamma)$.
$f$ is said to be \textbf{\emph{left continuous on}} \bm{$(\mu, \nu]$} if
$f$ is left continuous at each $\xi\in (\mu, \nu]$.
 Let $\beta\in [\mu, \nu)$.
$f$ is said to be \textbf{\emph{right continuous at}} \bm{$\beta$} if $f(\beta) = \lim_{\gamma\to \beta+} f(\gamma)$.
$f$ is said to be \textbf{\emph{right continuous on}} \bm{$[\mu, \nu)$} if
$f$ is right continuous at each $\eta\in [\mu, \nu)$.

Let
 $f$ be a function from an interval $[\mu, \nu]$ to $\widehat{\mathbb{R}}$,
and $\al\in (\mu, \nu]$.
Clearly if $f$ is left continuous at $\al$,
then
 $f$ is left lower semicontinuous at $\al$;
if $f$ is left continuous on $(\mu, \nu]$,
then
 $f$ is left lower semicontinuous on $(\mu, \nu]$ (see Remark \ref{lfe}(\rmn3)).
So the following
Corollary \ref{rsfc} follows immediately from
Proposition \ref{rsf}.

\begin{tl} \label{rsfc}
Let
 $f$ be a function from an interval $[\mu, \nu]$ to $\widehat{\mathbb{R}}$.
If $f$ is left continuous on $(\mu, \nu]$,
then $f $ is a measurable function on $[\mu, \nu]$.
\end{tl}

Let $(X,d)$ be a metric space and $x\in X$.
Define a function $f: X\to \mathbb{R}$ as $f(y)=d(x,y)$ for all $y\in X$.
Then $|f(y_1)-f(y_2)|= |d(x,y_1)-d(x,y_2)|\leq d(y_1, y_2)$ for all $y_1,y_2\in X$.
Thus $f$ is a Lipschitz mapping,
and so is uniformly continuous.
Below we use the symbol
$d(x, \cdot)$ to denote this function $f$.

A compact subset of a Hausdorff topological space is closed.
Each extended metric space is a Hausdorff topological space.
Let $(Y,\tau)$ be a topological space and $A,B,C$ be subsets of $Y$.
(\rmn1) $C$ is compact in $(Y,\tau)$ if and only if $(C, \tau_C)$ is a
compact topological space, where $\tau_C$ denotes the topology induced on $C$ by $\tau$.
(\rmn2) If $A$ is compact in $(Y,\tau)$ and $B$ is closed in
 $(Y,\tau)$, then $A\cap B$ is compact in $(Y,\tau)$.
 (\rmn2) can be shown as follows.
Consider conditions (a) $A\cap B$ is compact in $(Y,\tau)$,
and
 (b) $A\cap B$ is compact in $(A, \tau_A)$.
Then (a)$\Leftrightarrow$(b) because,
by (\rmn1), each of (a) and (b) is equivalent to
$(A\cap B, \tau_{A\cap B})$ is compact.
Note that $A\cap B$ is closed in the compact topological space$(A, \tau_A)$ (see (\rmn1)).
Thus (b) is true and so (a) is true. So (\rmn2) is proved.

\begin{tm} \label{rcm}
 Let    $u,v \in  F^1_{USC} (\mathbb{R}^m)$. Then $H([u]_\al, [v]_\al) $ is a measurable function of $\al$ on $[0,1]$; that is, $d_p(u,v)$ is well-defined.
\end{tm}

\begin{proof}

From Proposition \ref{rsf}, to show the desired result, we only need to show the conclusion
that
 $H([u]_\al, [v]_\al) \leq \liminf_{\gamma\to \al-} H([u]_\gamma, [v]_\gamma) $ for all $\al\in (0,1]$,
which,
by Lemma \ref{fcm},
 is equivalent to the following conclusion
 \begin{description}
   \item[(a)] For each $r\in \mathbb{R}$,
if $\al>0$ and $\al\in \{\al\in [0,1]: H([u]_\al, [v]_\al ) >r\}  $,
then there exists a $\delta (\al) > 0$
such that $[\al-\delta(\al), \al]  \subseteq \{\al\in [0,1]: H([u]_\al, [v]_\al ) >r\}$.
 \end{description}

 Assume that (a) is not true; that is, there is a $r\in \mathbb{R}$ and $\al>0$
with $\al\in \{\al\in [0,1]: H([u]_\al, [v]_\al ) >r\}  $
such that
 for each $\delta  > 0$,
 $[\al-\delta, \al] \nsubseteq \{\al\in [0,1]: H([u]_\al, [v]_\al ) >r\}$.
Then
 there exists a sequence $\{\gamma_n\}$ in $[0,\al)$ such that for $n=1,2,\ldots$, $\gamma_{n+1} > \gamma_n $, $\gamma_n \to \al$,
and
\begin{equation}\label{aem}
H([u]_{\gamma_n}, [v]_{\gamma_n})  \leq r
\end{equation}

 In this proof, we still use
  $d$ to denote the Euclidean metric on $\mathbb{R}^m$.
Given $x\in [u]_\al$.
We claim that
\\
(b) for each $n\in N$, $d(x, [v]_{\gamma_n}) \leq H([u]_{\gamma_n}, [v]_{\gamma_n})  \leq r$.
\\
(c) for each $n\in N$, there is a $y_n \in [v]_{\gamma_n}$ such that
$d(x, y_n) = d(x, [v]_{\gamma_n}) \leq r $.

 Let $n\in \mathbb{N}$. Note that $[u]_\al\subseteq [u]_{\gamma_n}$. Thus
 $x\in [u]_{\gamma_n}$ and hence by \eqref{aem}, $d(x, [v]_{\gamma_n}) \leq H([u]_{\gamma_n}, [v]_{\gamma_n})  \leq r$. So (b) is true.

 Let $n\in \mathbb{N}$.
For each $k\in \mathbb{N}$, there is
a $z_k\in [v]_{\gamma_n}$ with $d(x, z_k) \leq d(x, [v]_{\gamma_n}) +1/k\leq r+1/k \leq r+1$ (by (b), the second $\leq$ holds).
Note that the sequence $\{z_k\}$ is in $\overline{B}(x, r+1) \cap [v]_{\gamma_n}$, which is a compact subset of $\mathbb{R}^m$
as $\overline{B}(x, r+1)\in K(\mathbb{R}^m)$ and $[v]_{\gamma_n} \in C(\mathbb{R}^m)$.
 Thus there is a subsequence $\{z_{k_l}\}$ of $\{z_k\}$
and a $y_n$ in $\overline{B}(x, r+1) \cap [v]_{\gamma_n}$
such that
$y_n=\lim_{l\to\infty} z_{k_l}$.

Clearly $d(x, y_n)\geq d(x, [v]_{\gamma_n})$.
On the other hand, $d(x, y_n) = \lim_{l\to\infty} d(x, z_{k_l}) \leq \lim_{l\to\infty} \big(d(x, [v]_{\gamma_n})+1/{k_l}\big) = d(x, [v]_{\gamma_n})$ (since $d(x, \cdot)$ is continuous and $y_n=\lim_{l\to\infty} z_{k_l}$, the first ``='' holds).
Thus $d(x, y_n)= d(x, [v]_{\gamma_n})$.
 This together with (b) imply that (c) is true.

By (c), the sequence $\{y_n\}$ is in $\overline{B}(x, r)$, which is a compact subset of $\mathbb{R}^m$.
Thus
there is a subsequence $\{ y_{n_i}  \}$
of
$\{y_n\}$ and a $y$ in $\overline{B}(x, r)$
such that
$y=\lim_{i\to\infty} y_{n_i}$.
Observe that
$v(y)\geq \limsup_{i\to \infty} v(y_{n_i}) \geq \limsup_{i\to \infty} \gamma_{n_i}=\lim_{i\to \infty} \gamma_{n_i} = \al$;
so $y\in [v]_\al$ (see also (\uppercase\expandafter{\romannumeral1}) below).
Thus
$d(x, [v]_\al)\leq d(x,y)\leq r$
($y\in \overline{B}(x, r)$ means $d(x,y)\leq r$).

Since $x\in [u]_\al$ is arbitrary, we have that $H^*([u]_\al, [v]_\al) \leq r$.
Similarly, we can deduce
 that
 $H^*([v]_\al, [u]_\al) \leq r$.
Thus
$H([u]_\al, [v]_\al ) \leq r$, which is a contradiction.
So (a) is true. This completes the proof.

\vspace{1mm}
Contents in the following clause (\uppercase\expandafter{\romannumeral1})
are basic and easy to see.

(\uppercase\expandafter{\romannumeral1})
(\uppercase\expandafter{\romannumeral1}-1)
Clearly, for each $\gamma\in (0,1]$, $v(y)\geq \gamma \Leftrightarrow  y\in [v]_\gamma$.
(\uppercase\expandafter{\romannumeral1}-2)
for each $\gamma\in [0,1]$, $y\in [v]_\gamma \Rightarrow  v(y)\geq \gamma$.
If $\gamma=0$, then (\uppercase\expandafter{\romannumeral1}-2) holds obviously.
If $\gamma\in (0,1]$, then (\uppercase\expandafter{\romannumeral1}-2)
follows from (\uppercase\expandafter{\romannumeral1}-1).
So (\uppercase\expandafter{\romannumeral1}-2) is true.

As $\al\in (0,1]$, by (\uppercase\expandafter{\romannumeral1}-1), $v(y)\geq \al $ implies $y\in [v]_\al$.
By (\uppercase\expandafter{\romannumeral1}-2),
for each $i\in \mathbb{N}$,
$ v(y_{n_i}) \geq \gamma_{n_i}$ as $y_{n_i} \in [v]_{\gamma_{n_i}}$.

We can also show $y\in [v]_\al$ as follows.
Given $i_0\in \mathbb{N}$.
Note that for each $i\in \mathbb{N}$, $[v]_{\gamma_{n_i}} \supseteq [v]_{\gamma_{n_{i+1}}}$ and $y_{n_i} \in [v]_{\gamma_{n_i}}$. So
 $\{y_{n_i}, i\geq i_0\}$ is a sequence in
$[v]_{\gamma_{n_{i_0}}}$.
Since
$[v]_{\gamma_{n_{i_0}}}\in C(\mathbb{R}^m)$
and $y$ is the limit of
$\{y_{n_i}, i\geq i_0\}$, we have that $y\in [v]_{\gamma_{n_{i_0}}}$.
As $i_0\in \mathbb{N}$ is arbitrary,
it follows
that $y \in \bigcap_{i=1}^{+\infty} [v]_{\gamma_{n_i}} = [v]_\al$.

\end{proof}

\begin{re}
 {\rm
Let $(X, d)$ be a metric space.
We say that
 $S\subseteq F^1_{USC}(X)$ satisfies condition $(X,d)$-\uppercase\expandafter{\romannumeral1} if
$[u]_\al \cap \overline{B}(x,r)$ is compact in $(X, d)$ for all $u\in S$, $\al\in (0,1]$, $x\in X$ and $r\in \mathbb{R}^+$.
We have conclusion
\\
(\rmn1) If $S\subseteq F^1_{USC}(X)$ satisfies condition $(X,d)$-\uppercase\expandafter{\romannumeral1}, then
for all $u,v \in S$, $H([u]_\al, [v]_\al) $ is a measurable function of $\al$ on $[0,1]$.

The proof of Theorem \ref{rcm} will become the proof of conclusion (\rmn1)
if the proof of Theorem \ref{rcm} is adjusted as follows.
First delete ``In this proof, we still use
  $d$ to denote the Euclidean metric on $\mathbb{R}^m$''
and ``as $\overline{B}(x, r+1)\in K(\mathbb{R}^m)$ and $[v]_{\gamma_n} \in C(\mathbb{R}^m)$''.
Then
 replace ``a sequence $\{\gamma_n\}$ in $[0,\al)$'' by ``a sequence $\{\gamma_n\}$ in $[\al/2,\al)$'' and
replace $\mathbb{R}^m$ by $X$.
At last,
replace $\overline{B}(x, r)$ by $\overline{B}(x, r) \cap [v]_{\al/2}$
in the paragraph beginning with ``By (c)''.
As $\{\gamma_{n_i} \}$ in the adjusted proof
is a sequence in $[\al/2, \al) \subseteq (0,1]$,
``By (\uppercase\expandafter{\romannumeral1}-2)'' can also be replaced by ``By (\uppercase\expandafter{\romannumeral1}-1)''

Clearly,
for all $u\in F^1_{USC}(\mathbb{R}^m)$, $\al\in (0,1]$, $x\in \mathbb{R}^m$ and $r\in \mathbb{R}^+$,
$[u]_\al \cap \overline{B}(x,r) \in K(\mathbb{R}^m)\cup \{\emptyset\}$
as $[u]_\al \in C(\mathbb{R}^m)$ and $ \overline{B}(x,r) \in K(\mathbb{R}^m)$.
So
 $S=F^1_{USC}(\mathbb{R}^m) $ satisfies condition $\mathbb{R}^m $-\uppercase\expandafter{\romannumeral1}.

Clearly,
for all $u\in F^1_{USCG} (X)$, $\al\in (0,1]$, $x\in X$ and $r\in \mathbb{R}^+$,
$[u]_\al \cap \overline{B}(x,r) \in K(X)\cup \{\emptyset\}$
as $[u]_\al \in K(X)$ and $ \overline{B}(x,r) \in C(X)$.
So
 $S= F^1_{USCG} (X)$ satisfies condition $(X,d)$-\uppercase\expandafter{\romannumeral1}.

So Theorem \ref{rcm} and Proposition \ref{lcpm}
are special cases of conclusion (\rmn1).

 }
\end{re}

The following Proposition \ref{cmer} is known.
From Proposition \ref{cmer}, it is easy to obtain Proposition \ref{gnc}, which are part of Lemma 5.4 in \cite{huang19c}.

\begin{pp}  \label{cmer}
For $n=1,2,\ldots$, let $U_n \in K(X)$ and $V_n \in K(X)$ .
\\
(\romannumeral1) If $U_1\supseteq U_2 \supseteq \ldots \supseteq U_n\supseteq \ldots$, then
$U = \bigcap_{n=1}^{+\infty}  U_n    \in K(X)$ and $H(U_n, U) \to 0$ as $n\to +\infty$.
\\
(\romannumeral2) If $V_1\subseteq V_2 \subseteq \ldots \subseteq V_n\subseteq \ldots$
and
$V = \overline{\bigcup_{n=1}^{+\infty}  V_n }   \in K(X)$, then $H(V_n, V) \to 0$ as $n\to +\infty$.
\end{pp}

\begin{pp}\cite{huang19c}
  \label{gnc}
Let $u\in F^1_{USCG}(X)$.
\\
(\romannumeral1) \ For each $\al\in (0,1]$, $\lim_{\beta\to \al-} H([u]_\beta, [u]_\al) =0$.
\\
(\romannumeral2) \ For each $\al\in (0,1)$, $\lim_{\gamma\to \al+} H([u]_\gamma, \overline{\{u>\al\}}) =0$.
\\
(\romannumeral3) \ If $u\in F^1_{USCB}(X)$, then $\lim_{\gamma\to 0+} H([u]_\gamma, [u]_0) =0$.

\end{pp}

Let $u,v \in  F^1_{USC}(X)$ and $h\in [0,1)$.
We can see
 $H([u]_\al, [v]_\al)$ as a function of $[0,1]$ into $\widehat{\mathbb{R}}$ with $\al$ as the independent variable,
 $H([u]_\al, [u]_{\al-h})$ a function of $[h,1]$ into $\widehat{\mathbb{R}}$ with $\al$ as the the independent variable,
and
$H([u]_\al, [u]_{\al+h}) $
a function of $[0, 1-h]$ into $\widehat{\mathbb{R}}$ with $\al$ as the independent variable.
In the sequel,
we will not emphasize whether the expressions $H([u]_\al, [v]_\al)$,
$H([u]_\al, [u]_{\al-h})$ and $H([u]_\al, [u]_{\al+h})$
are functions or numeric values, because it is easy to tell what they mean from the context.

The following Proposition \ref{lrcm} is Proposition 5.5 in \cite{huang19c}.

\begin{pp} \cite{huang19c} \label{lrcm}
 (\romannumeral1) If $u,v \in  F^1_{USCG}(X)$, then $H([u]_\al, [v]_\al) $
is
 left continuous at each $\al\in (0,1]$.
  (\romannumeral2) If $u,v \in  F^1_{USCB}(X)$, then $H([u]_\al, [v]_\al) $
is
 right continuous at $\al=0$.
\end{pp}

\begin{re} \label{tra}
  {\rm The following fact (a) is known and used in the proofs of Propositions \ref{lrcm} and \ref{lcpmcu}.
 (a) Let $(Y,\rho)$ be an extended metric space and $x,y,x_1,y_1\in Y$.
If $\rho(x,y)$ and $\rho(x_1, y_1)$ are finite, then
$|\rho(x,y) - \rho(x_1,y_1)|\leq \rho(x,x_1) + \rho(y,y_1)$.

To show (a), it suffices to verify
(a-1) $\rho(x,y) - \rho(x_1,y_1) \leq  \rho(x,x_1) + \rho(y,y_1)$
and
(a-2) $\rho(x_1,y_1) - \rho(x,y) \leq  \rho(x,x_1) + \rho(y,y_1)$.
By the triangle inequality,
$\rho(x,y) \leq \rho(x,y_1) + \rho(y,y_1) \leq \rho(x_1,y_1) + \rho(x,x_1) + \rho(y,y_1)$.
So (a-1) is true.
Similarly $\rho(x_1,y_1) \leq \rho(x_1,y) + \rho(y,y_1) \leq \rho(x,y) + \rho(x,x_1) + \rho(y,y_1)$. So (a-2) is true.
Thus (a) is true.

 (a-2) can also be obtained
by interchanging $x$ and $x_1$ and interchanging $y$ and $y_1$ in (a-1).
Of course, the inequality to show (a-2) can also be obtained from the inequality to show (a-1)
by same means.

}
\end{re}

The following
Proposition \ref{lcpm} follows immediately from Proposition \ref{lrcm}(\romannumeral1) and
Corollary \ref{rsfc}. Proposition \ref{lcpm} has been given in \cite{huang19c} (see the paragraph before Proposition 6.1 of \cite{huang19c}).

\begin{pp} \cite{huang19c} \label{lcpm}
  For each $u,v \in  F^1_{USCG}(X)$, $H([u]_\al, [v]_\al) $ is a measurable function of $\al$ on $[0,1]$; that is, $d_p(u,v)$ is well-defined.
\end{pp}

\begin{re}\label{uen}
 {\rm

In the proof of Lemma 6.5 of \cite{huang17}, we gave the conclusion (a) for each $u\in F_{USCG}^1(X)$,
the cut-function
$[u](\al) = [u]_\al$ from $[0,1]$ to $(C(X),H)$ is left continuous on $(0,1]$.
From conclusion (a) it is easy to get Proposition \ref{lrcm}(\rmn1):
for each $u,v \in  F^1_{USCG}(X)$, $H([u]_\al, [v]_\al) $
is
 left continuous at each $\al\in (0,1]$.
We consider Corollary \ref{rsfc}, which says that a left continuous function $f: [\mu, \nu] \to \widehat{\mathbb{R}}$ is measurable,
 to be a known conclusion.
So Proposition \ref{lcpm} can be seen as an easy corollary of
conclusion (a).

Conclusion (a) and the statement of Proposition \ref{lcpm} have been given in our paper arxiv1904.07489v2 (see the proof of Lemma 6.3
and Page 5 in this paper), which is an early version of \cite{huang17} and was submitted on
2019.07.06.
From these contents in arxiv1904.07489v2, it is easy and natural to give the proof
of Proposition \ref{lcpm}.
Conclusion (a) is Proposition \ref{gnc}(\rmn1).

Steps 2 and 3 of the proof of Proposition \ref{rsf} should be known,
we give them for completeness of this paper.

 }
\end{re}

\begin{pp}\label{lcpmcu}
Let $u \in  F^1_{USCG}(X)$ and $h\in [0,1]$.
\\
(\romannumeral1) \ $H([u]_\al, [u]_{\al-h}) $
is
 left continuous at each $\al\in (h,1]$.
\\
(\romannumeral2) \
$H([u]_\al, [u]_{\al-h}) $ is a measurable function of $\al$ on $[h,1]$; that is, $\left( \int_h^1  H([u]_\alpha,   [u]_{\alpha-h})^p \;d\alpha   \right)^{1/p}$ is well-defined.
\\
(\romannumeral3) \ $H([u]_\al, [u]_{\al+h}) $
is
 left continuous at each $\al\in (0,1-h]$.
\\
(\romannumeral4) \
$H([u]_\al, [u]_{\al+h}) $ is a measurable function of $\al$ on $[0, 1-h]$;
that is, $\left( \int_0^{1-h}  H([u]_\alpha,   [u]_{\alpha+h})^p \;d\alpha   \right)^{1/p}$ is well-defined.
\end{pp}

\begin{proof}

The proofs of (\romannumeral1) and (\romannumeral3)
are similar to that of Proposition \ref{lrcm}(\rmn1).

First we show (\rmn1).
Let $\al\in (h,1]$.
Note that for each $\beta \in (h,1]$,
$H([u]_\beta, [u]_{\beta-h}) $
is
finite and
$
   |H([u]_\al, [u]_{\al-h}) -H([u]_\beta, [u]_{\beta-h}) |  \leq
  H([u]_\al,   [u]_\beta) + H([u]_{\al-h}, [u]_{\beta-h}).
$
By
 Proposition \ref{gnc} (\romannumeral1), $\lim_{\beta \to \al-} (H([u]_\al,   [u]_\beta) + H([u]_{\al-h}, [u]_{\beta-h}) ) = 0$. Hence
  $\lim_{\beta \to \al-} H([u]_\beta, [u]_{\beta-h}) =H([u]_\alpha, [u]_{\alpha-h})$;
that is,
 $H([u]_\alpha, [u]_{\alpha-h}) $ is
 left continuous at $\al$.
So (\romannumeral1) is true. Thus from
Corollary \ref{rsfc},
 (\romannumeral2) is true.

Now we show (\rmn3). Let $\al\in (0,1-h]$.
Note that for each $\beta\in (0, 1-h]$,
 $H([u]_\beta, [u]_{\beta+h}) $
is
finite and
$ |H([u]_\al, [u]_{\al+h}) -H([u]_\beta, [u]_{\beta+h}) |  \leq
  H([u]_\al,   [u]_\beta) + H([u]_{\al + h}, [u]_{\beta + h}).
$
By
 Proposition \ref{gnc}(\romannumeral1), $\lim_{\beta \to \al-} (H([u]_\al,   [u]_\beta) + H([u]_{\al+h}, [u]_{\beta+h}) ) = 0$. Hence
  $\lim_{\beta \to \al-} H([u]_\beta, [u]_{\beta+h})  = H([u]_\alpha, [u]_{\alpha+h})$;
that is,
 $H([u]_\alpha, [u]_{\alpha+h})$ is
 left continuous at $\al$.
So (\romannumeral3) is true. Thus from
Corollary \ref{rsfc},
 (\romannumeral4) is true.

If $h=1$, then $H([u]_\al, [u]_{\al-h}) $ is a function with its domain being a single point $\{1\}$ and  $H([u]_\al, [u]_{\al+h}) $ is a function with its domain
being a single point $\{0\}$. So Proposition \ref{lcpmcu} holds trivially in the case when $h=1$.
\end{proof}

To give the example which shows the last statement presented in \cite{huang17} which is listed at the beginning of this section,
we need some conclusions at first.

The following two representation theorems should be known.
We suppose that
$\sup \emptyset =0$ in this paper.
\begin{tm} \label{rep}
Let $Y$ be a set.
  If $u\in F(Y)$, then for each
$\al\in (0,1]$,
  $[u]_\al = \cap_{\beta<\al} [u]_\beta.$

Conversely,
suppose that
$\{v(\al): \al \in (0,1]\}$
is
a family of sets in $Y$ satisfying $v(\al) = \cap_{\beta<\al} v(\beta)$ for all
$\al\in (0,1]$. Define
$u\in F(Y)$ by
$u(x) := \sup\{  \al\in (0,1]:  x\in v(\al)    \} $ ($\sup\emptyset = 0$) for each
$x\in Y$.
Then $u$ is the unique fuzzy set in $Y$ satisfying that $[u]_\al = v(\al)$ for all $\al\in (0,1]$.
\end{tm}

\begin{tl} \label{repc}
Let $(Y,\tau)$ be a topological space.
 (\rmn1) If $u\in F(Y)$, then for each
$\al\in (0,1]$,
  $[u]_\al = \cap_{\beta<\al} [u]_\beta$,
and $[u]_0= \overline{\cup_{\al>0} [u]_\al}$.

(\rmn2)
Suppose that
$\{v(\al): \al \in (0,1]\}$
is
a family of sets in $Y$ satisfying $v(\al) = \cap_{\beta<\al} v(\beta)$ for all
$\al\in (0,1]$ and $v(0)= \overline{\cup_{\al>0} v(\al)}$. Define
$u\in F(Y)$ by
$u(x) := \sup\{  \al\in (0,1]:  x\in v(\al)    \} $ for each
$x\in Y$.
Then $u(x) = \sup\{  \al\in [0,1]:  x\in v(\al)    \} $ for each
$x\in Y$ and $u$ is the unique fuzzy set in $Y$ satisfying that $[u]_\al = v(\al)$ for all $\al\in [0,1]$ (respectively, $\al\in (0,1]$).
\end{tl}

\begin{proof} If $u\in F(Y)$, then
$[u]_0= \overline{    \{x: u(x) > 0 \}    }=\overline{\cup_{\al>0} [u]_\al}$.
(\rmn1) follows from this fact and Theorem \ref{rep}.

Now we show (\rmn2).
Given $x\in X$.
Put
$a=\sup\{  \al\in (0,1]:  x\in v(\al)    \}$
and
$b = \sup\{  \al\in \{0\}:  x\in v(\al)    \}$.
Then
$a \geq 0$ and $b=0$.
Thus
$u(x)=a= a \vee b=\sup\{  \al\in [0,1]:  x\in v(\al)    \} $ (As $[0,1] = (0,1] \cup \{0\}$, the last $=$ holds).

By Theorem \ref{rep}, $u$ is the unique fuzzy
set in $Y$ satisfying that $[u]_\al = v(\al)$ for all $\al\in (0,1]$.
Then
$[u]_0 = \overline{\cup_{\al>0} [u]_\al}= \overline{\cup_{\al>0} v(\al)}=v(0)$.
So
$u$ is also the unique fuzzy
set in $Y$ satisfying that $[u]_\al = v(\al)$ for all $\al\in [0,1]$.
Thus (\rmn2) is proved.

\end{proof}

Let $Y$ be a set. By Theorem \ref{rep},
we can define a fuzzy set $u$ in $Y$
by putting $[u]_\al=v_\al$ for all $\al\in (0,1]$, where $\{v_\al: \al\in (0,1]\}$ satisfies condition (a)
 $v(\al) = \cap_{\beta<\al} v(\beta)$ for all
$\al\in (0,1]$.
Let $(Y,\tau)$ be a topological space.
By Corollary \ref{repc},
we can define a fuzzy set $u$ in $Y$
by putting $[u]_\al=v_\al$ for all $\al\in [0,1]$, where $\{v_\al: \al\in [0,1]\}$ satisfies condition (b)
 $v(\al) = \cap_{\beta<\al} v(\beta)$ for all
$\al\in (0,1]$ and $v(0)= \overline{\cup_{\al>0} v(\al)}$. See Example \ref{nmf}.
As conditions (a) and (b) are easy to verify, we omit these steps in this paper.
In the sequel, we
will define a fuzzy set in a set or in a metric space without saying which representation theorem is used since it is easy
to see.

\begin{pp}\label{ace}
  Let $Y$ and $Z$ be two nonempty sets and $A\in F(Y)$ and $B\in F(Z)$.
Consider the conditions (a) $A=B$;
 (b) for each $\al\in (0,1]$, $[A]_\al = [B]_\al$;
 (c) there is a dense set $S$ of $(0,1]$ such that for each $\al\in S$, $[A]_\al = [B]_\al$.
Then
(\rmn1) (a)$\Rightarrow$(b)$\Leftrightarrow$(c);
(\rmn2)
if $Y=Z$, then (a)$\Leftrightarrow$(b);
(\rmn3)
(b) does not imply (a).

Endow two topologies $\tau_1$ and $\tau_2$ on $Y$ and $Z$,
respectively. Consider
condition (b$_1$) for each $\al\in [0,1]$, $[A]_\al = [B]_\al$.
We have that
 (\rmn4) clearly (a)$\Rightarrow$(b$_1$)$\Rightarrow$(b);
 (\rmn5) (b$_1$) does not imply (a); (\rmn6) (b) does not imply (b$_1$);
 and (\rmn7) if $(Y, \tau_1)=(Z, \tau_2)$, then (a)$\Leftrightarrow$(b$_1$)$\Leftrightarrow$(b).
\end{pp}

\begin{proof}
Clearly (a)$\Rightarrow$(b)$\Rightarrow$(c) (see  (\uppercase\expandafter{\romannumeral1}) below).
Suppose that (c) is true. Given $\al\in (0,1]$.
Observe that for each $k\in \mathbb{N}$, $S\cap (\al\cdot \frac{k-1}{k}, \al) \not= \emptyset$, so we can choose $\al_k \in S\cap(\al\cdot \frac{k-1}{k}, \al)$.
Then $\al_k \to \al-$ as $k\to\infty$.
We have that $[A]_\al = \cap_{k=1}^{+\infty}[A]_{\al_k} = \cap_{k=1}^{+\infty}[B]_{\al_k} = [B]_\al$ (see also (\uppercase\expandafter{\romannumeral2}) below).
Hence (b) is true. Thus (c)$\Rightarrow$(b).
So (\rmn1) is proved.

We show (\rmn2) as follows.
Assume that (b) is true. As $Y=Z$,
by Theorem \ref{rep}, there is a unique fuzzy set
$u$ in $Y$ satisfying condition (1) $[u]_\al = [A]_\al = [B]_\al$ for all $\al\in (0,1]$. Thus $A=u=B$ as both $A$ and $B$ satisfy
condition (1).
Hence (a) is true and (b)$\Rightarrow$(a) is proved.
By (\rmn1), (a)$\Rightarrow$(b). Thus (a)$\Leftrightarrow$(b).
So (\rmn2) is proved.

We still use $\mathbb{R} \setminus \{1\}$ to denote the metric space
$(\mathbb{R} \setminus \{1\}, \rho_1|_{\mathbb{R} \setminus \{1\}})$.
Consider
$\widehat{3}_{F(\mathbb{R} \setminus \{1\})}$ and $\widehat{3}_{F(\mathbb{R})}$.
Then for each $\al\in [0,1]$,
 $[\widehat{3}_{F(\mathbb{R} \setminus \{1\})}]_\al =  \{3\} = [\widehat{3}_{F(\mathbb{R})}]_\al$. But   $\widehat{3}_{F(\mathbb{R} \setminus \{1\})} \not= \widehat{3}_{F(\mathbb{R})}$ (see also (\uppercase\expandafter{\romannumeral1}) below). So (\rmn3) and (\rmn5) are proved.
Consider
$(0,1)_{F(\mathbb{R} \setminus \{1\})}$ and $(0,1)_{F(\mathbb{R})}$.
Then for each $\al\in (0,1]$,
 $[(0,1)_{F(\mathbb{R} \setminus \{1\})}]_\al =  (0,1) = [(0,1)_{F(\mathbb{R})}]_\al$.
But $[(0,1)_{F(\mathbb{R} \setminus \{1\})}]_0 \not= [(0,1)_{F(\mathbb{R})}]_0$
because $[(0,1)_{F(\mathbb{R} \setminus \{1\})}]_0 =  [0,1)$ and $[(0,1)_{F(\mathbb{R})}]_0 = [0,1]$.
So (\rmn6) is proved.

Suppose that $(Y, \tau_1)=(Z, \tau_2)$.
If (b) is true, then $[A]_0=  \overline{\cup_{\al>0} [A]_\al}=  \overline{\cup_{\al>0} [B]_\al}=[B]_0$, and therefore
 (b$_1$) is true. Thus (b)$\Rightarrow$(b$_1$).
Hence (a)$\Leftrightarrow$(b$_1$)$\Leftrightarrow$(b), by (\rmn2) and (\rmn4). So (\rmn7) is proved.

\vspace{1mm}
The contents in the following (\uppercase\expandafter{\romannumeral1}) and (\uppercase\expandafter{\romannumeral2}) are easy to see.

(\uppercase\expandafter{\romannumeral1})
 Let $u\in F(Y)$ and $v\in F(Z)$. Then $u$ and $v$
are functions with domains $Y$ and $Z$, respectively.
If $u=v$, then they have the same domain; that is $Y=Z$.

(\uppercase\expandafter{\romannumeral2})
 Let $P$ be a set and $u\in F(P)$. If $\beta\in (0,1]$,
and the sequence ${\beta_n}$ in $[0, \beta]$ satisfying that
 $\beta_n\to \beta$, then $[u]_\beta = \cap_{n=1}^{+\infty} [u]_{\beta_n}$.

Clearly  $[u]_\beta \subseteq \cap_{n=1}^{+\infty} [u]_{\beta_n}$
as for each $n\in \mathbb{N}$, $[u]_\beta \subseteq  [u]_{\beta_n}$.
Let $x\in \cap_{n=1}^{+\infty} [u]_{\beta_n}$.
This means that
for each $n\in \mathbb{N}$, $x \in [u]_{\beta_n}$,
i.e., $u(x) \geq \beta_n$.
So $u(x) \geq \beta= \lim_{n\to\infty} \beta_n$.
Then $x\in [u]_\beta$.
Hence $ \cap_{n=1}^{+\infty} [u]_{\beta_n} \subseteq [u]_\beta $.
Thus $[u]_\beta =\cap_{n=1}^{+\infty} [u]_{\beta_n}$.

\end{proof}

\begin{re}\label{hme}
  {\rm
(a) If $(Y, \rho)$ is an extended metric space,
then
the Hausdorff distance $H$ on $C(Y)$ induced by $\rho$ using \eqref{hau} is
an extended metric on
$C(Y)$.

Clearly $H$ is a function
of $C(Y)\times C(Y)$ into $\widehat{\mathbb{R}}$
and
$H$ satisfies positivity and symmetry.
To show (a), we only need to show that
$H$
satisfies the triangle inequality.
To do this,
it suffices to show
\\
(b)
 $ H^*(U, W) \leq  H^*(U,V) + H^*(V,W)$
for each $U, V, W \in C(Y)$.
\\
(Assume (b) is true. Then for each $U, V, W \in C(Y)$,
 $ H^*(U, W) \leq  H(U,V) + H(V,W)$ and  $H^*(W,U) \leq  H(W,V) + H(V,U)$;
whence $ H(U, W) \leq  H(U,V) + H(V,W)$.)

To show (b), let $x\in U$. Then
\begin{align*}
  \rho(x, W)=\inf_{z\in W} \rho(x,z) &\leq \inf_{z\in W} \inf_{y\in V} \{\rho(x,y) + \rho(y,z) \} \\
&= \inf_{y\in V}\inf_{z\in W}  \{\rho(x,y) + \rho(y,z) \}
\\
&\leq \inf_{y\in V} \{\rho(x,y) + \rho(y,W) \}
\\
&\leq \inf_{y\in V} \{\rho(x,y) + H^*(V,W)\}
= \inf_{y\in V} \rho(x,y) + H^*(V,W)
\\
&=  \rho(x,V) + H^*(V,W)
\leq  H^*(U,V) + H^*(V,W).
\end{align*}
As $x\in U$ is arbitrary,
it follows that (b) is true.
So (a) is true.

Assume that $(Y,\rho)$ is an extended metric space but not a metric space.
Then there exist $x,y\in Y$ such that $\rho(x,y)=+\infty$.
Thus $\{x\}$ and $\{y\}$ are in $C(Y)$ and $H(\{x\},\{y\})=+\infty$.
So the Hausdorff distance $H$ on $C(Y)$ is not a metric.
$H$ on $C(Y)$ does not need to be a metric even if
$(Y,\rho)$ is a metric space. See Remark \ref{haum}.

}
\end{re}

Let $J$ be a nonempty set, and for each $j \in J$, let $(X_j, d_j)$
be a
metric space.
For each $x \in \prod_{j\in J} X_j$ and each $j\in J$, we use $x_j$
to denote $p_j(x)$, where $p_j: \prod_{j\in J} X_j \to X_j$ is the projection mapping.

Define a mapping $d_J$
of
$\prod_{j\in J} X_j\times \prod_{j\in J} X_j$ into $\widehat{\mathbb{R}}$
as
\begin{equation}\label{upm}
d_J(x,y):= \sup\{d_j(x_j, y_j):  j \in J\}
\end{equation}
for each $x, y \in \prod_{j\in J} X_j$.
Then $d_J$ is an extended metric on $\prod_{j\in J} X_j$
(Clearly $d_J$ satisfies positivity and symmetry.
For each $x,y,z\in \prod_{j\in J} X_j$,
$d_J(x,z)= \sup\{d_j(x_j, z_j):  j \in J\} \leq \sup\{d_j(x_j, y_j)+ d_j(y_j, z_j):  j \in J\} \leq \sup\{d_j(x_j, y_j):  j \in J\} + \sup\{d_j(y_j, z_j):  j \in J\} = d_J(x,y) + d_J(y,z)$.
So $d_J$
satisfies the triangle inequality.).

We use the symbol $\prod_{j\in J} (X_j, d_j)$ to denote the extended metric space $(\prod_{j\in J} X_j, d_J)$.
If not mentioned specially,
 we suppose by default
that
$\prod_{j\in J} X_j$ is endowed with the extended metric $d_J$ given by \eqref{upm}.
We also use
$\prod_{j\in J} X_j$ to denote the metric space $(\prod_{j\in J} X_j, d_J)$.

For each $j \in J$, let $u_j  \in  F(X_j)$.
Define
 $u \in F (\prod_{j \in J} X_j) $
as
\begin{equation}\label{pdf}
 [u]_\al = \prod_{j \in J} [u_j]_\al \   \mbox{for each} \    \al \in (0,1].
\end{equation}
We use
\bm{$ \prod_{j \in J} u_j$}
to denote the fuzzy set
 $u$ in $\prod_{j \in J} X_j$ determined by \eqref{pdf}.

From
 Theorem \ref{rep},
$u $ is well-defined because for each $\al\in (0,1]$,
$$
 [u]_\al  =   \prod_{j \in J} [u_j]_\al
=  \bigcap_{\beta < \al}   \prod_{j \in J}  [u_j]_\beta
 = \bigcap_{\beta < \al} [u]_\beta
$$
(Given $x=(x_j)_{j \in J} \in \bigcap_{\beta < \al}   \prod_{j \in J}  [u_j]_\beta$. Let $j\in J$.
Then for each $\beta\in [0,\al)$, $x_j\in [u_j]_\beta$.
This means $x_j\in [u_j]_\alpha$.
Since $j\in J$ is arbitrary, we have that $x\in \prod_{j \in J} [u_j]_\al$.
Thus $\prod_{j \in J} [u_j]_\al
\supseteq \bigcap_{\beta < \al}   \prod_{j \in J}  [u_j]_\beta$.
Clearly  $\prod_{j \in J} [u_j]_\al
\subseteq \bigcap_{\beta < \al}   \prod_{j \in J}  [u_j]_\beta$,
and so $\prod_{j \in J} [u_j]_\al
= \bigcap_{\beta < \al}   \prod_{j \in J}  [u_j]_\beta$.).

In this paper, if not mentioned specially,
we use $\overline{S}$ to denote
the closure of $S$ in a certain extended metric space $(X, d_X)$.
For a set $S \subseteq X_j$, $j \in J$,
we use
$\overline{S}$
to denote
  the closure
of
$S$ in $( X_j, d_j)$.
For
 a set $S \subseteq \prod_{j\in J} X_j$,
we also use
$\overline{S} $ to denote the closure
 of $S$
in $(\prod_{j\in J} X_j, d_J)$.
The readers can judge
the meaning of
$\overline{S}$
according to the contexts.

\begin{lm}\label{cdm}
Let $J$ be a nonempty set. For each $j \in J$, let $(X_j, d_j)$
be a
metric space and
   $A_j$ be a subset of $X_j$.
  Then
$\overline{\prod_{j \in J}   A_j} = \prod_{j \in J}  \overline{A_j }  $.

\end{lm}

\begin{proof}
Let $x\in  \overline{\prod_{j \in J}   A_j}$.
Then there is a sequence $\{x_n\}$ in $\prod_{j \in J}   A_j$
satisfying $d_J(x_n, x)\to 0$ as $n\to\infty$.
Thus for each $j\in J$, $d_j({x_n}_j,  {x}_j)\to 0$ as $n\to\infty$;
whence ${x}_j\in \overline{A_j}$.
So $x\in  \prod_{j \in J} \overline{A_j}$.
Thus $\overline{\prod_{j \in J}   A_j} \subseteq \prod_{j \in J}  \overline{A_j }$.

Let $x= (x_j)_{j\in J} \in \prod_{j \in J}  \overline{A_j }$.
Given $\varepsilon>0$. Then for each $j\in J$, there exists
$y_j\in A_j$
such that
$d_j (x_j, y_j) \leq \varepsilon$, as $x_j\in  \overline{A_j }$.
Put $y:=(y_j)_{j\in J}$.
Then $y\in \prod_{j \in J}   A_j$ and
$d_J(x,y) \leq \varepsilon$.
Since $\varepsilon>0$ is arbitrary, we have $x\in \overline{\prod_{j \in J}   A_j}  $.
Thus
 $\overline{\prod_{j \in J}   A_j} \supseteq \prod_{j \in J}  \overline{A_j }  $.

In summary, $\overline{\prod_{j \in J}   A_j} = \prod_{j \in J}  \overline{A_j }  $.

\end{proof}

\begin{tm} \label{pfne}
Let $J$ be a nonempty set. For each $j \in J$, let $(X_j, d_j)$
be a
metric space and
 $u_j  \in  F_{USC} (X_j)$.
Then
 $u = \prod_{j \in J} u_j \in F_{USC} (\prod_{j \in J} X_j)$.

\end{tm}

\begin{proof}
  By \eqref{pdf} and Lemma \ref{cdm}, for each $\al\in (0,1]$,
$$ \overline{[u]_\al} =\overline{\prod_{j \in J} [u_j]_\al}= \prod_{j \in J} \overline{[u_j]_\al  } =  \prod_{j \in J} [u_j]_\al = [u]_\al. $$
Thus
$u \in F_{USC} (\prod_{j \in J} X_j)$.

\end{proof}

\begin{tm} \label{pfn}
Let $J$ be a nonempty set. For each $j \in J$, let $(X_j, d_j)$
be a
metric space and
  $u_j  \in  F^1_{USC}(X_j)$.
Then
 $u = \prod_{j \in J} u_j \in F^1_{USC} (\prod_{j \in J} X_j)$.

\end{tm}

\begin{proof}
By
 Theorem \ref{pfne}, $u \in F_{USC} (\prod_{j \in J} X_j)$.
As for each $j\in J$, $[u_j]_1 \not= \emptyset$,
it follows that $[u]_1=\prod_{j \in J} [u_j]_1 \not= \emptyset$;
that is, $u\in F^1(\prod_{j \in J} X_j)$.
So $u\in F^1_{USC} (\prod_{j \in J} X_j)$.

\end{proof}

In the following theorem, we
use $H$ to denote the Hausdorff metric on $C(X_j)$ induced by $d_j$.
We also
use
$H$ to denote the Hausdorff metric on $C(\prod_{j\in J} X_j)$
induced by $d_J$.

\begin{tm} \label{hnp}
Let $J$ be a nonempty set. For each $j \in J$, let $(X_j, d_j)$
be a
metric space and let
$A_j$
and
$B_j$ be elements in $C(X_j)$.
Then
\\
(\rmn1) $\prod_{j\in J} A_j$ and $\prod_{j\in J} B_j$
are elements in $C(\prod_{j \in J}  X_j)$;
\\
(\rmn2) for each $x\in \prod_{j \in J}  X_j $,
$
d_J(x,    \prod_{j\in J} B_j)
= \sup_{j\in J}  d_j (x_j, B_j)
$;
\\
(\rmn3)  $H^*(\prod_{j\in J} A_j, \prod_{j\in J} B_j) = \sup_{j\in J} H^*(A_j, B_j )$;
\\
(\rmn4)
  $H(\prod_{j\in J} A_j, \prod_{j\in J} B_j) = \sup_{j\in J} H (A_j, B_j )$.
\end{tm}

\begin{proof} From Lemma \ref{cdm},
$\overline{\prod_{j\in J} A_j}= \prod_{j\in J} \overline{A_j}= \prod_{j\in J} A_j$. So $\prod_{j\in J} A_j$ is closed in $\prod_{j \in J}  X_j$.
Clearly $\prod_{j\in J} A_j\not= \emptyset$.
Thus $\prod_{j\in J} A_j\in C(\prod_{j \in J}  X_j)$.
Similarly $\prod_{j\in J} B_j \in C(\prod_{j \in J}  X_j)$.
So (\rmn1) is true.

Now we show (\rmn2).
Let $\varepsilon>0$. For each $j\in J$,
there is a $y_j\in B_j$
such that
$ d_{j} (x_{j}, y_j) \leq \inf_{z_j\in B_j} d_j (x_j, z_j)+ \varepsilon
 = d_j (x_j, B_j)+ \varepsilon $.
Put $y:= (y_j)_{j\in J}$. Clearly $y\in \prod_{j\in J} B_j$.
Then
$d_J(x,    \prod_{j\in J} B_j) \leq d_J(x,y) = \sup_{j\in J} d_{j} (x_{j}, y_j) \leq \sup_{j\in J} (d_j (x_j, B_j)+ \varepsilon ) = \sup_{j\in J} d_j (x_j, B_j)+ \varepsilon $.
Since $\varepsilon>0$ is arbitrary, we conclude that
$d_J(x,    \prod_{j\in J} B_j) \leq \sup_{j\in J} d_j (x_j, B_j) $.

Let $j\in J$. Then
$d_J(x,    \prod_{j\in J} B_j) = \inf_{y\in  \prod_{j\in J} B_j} d_J(x, y)
\geq
\inf_{y\in  \prod_{j\in J} B_j} d_j(x_j, y_j) = \inf_{y_j\in B_j} d_j (x_j, y_j)
= d_j(x_j, B_j).
$
As $j\in J$ is arbitrary, it follows that
$d_J(x,    \prod_{j\in J} B_j) \geq \sup_{j\in J} d_j(x_j, B_j)
$. So $d_J(x,    \prod_{j\in J} B_j) = \sup_{j\in J} d_j (x_j, B_j)$ and (\rmn2) is true.

By (\rmn2),
\begin{align*}
  H^* (\prod_{j\in J} A_j, \prod_{j\in J} B_j)
 & = \sup_{x \in \prod_{j\in J} A_j} d_J(x,    \prod_{j\in J} B_j) \\
 & =
\sup_{x \in \prod_{j\in J} A_j}    \sup_{j\in J} d_j (x_j,   B_j) =    \sup_{j\in J} \sup_{x \in \prod_{j\in J} A_j} d_j (x_j,   B_j)  \\
&
= \sup_{j\in J} \sup_{x_j \in A_j} d_j (x_j, B_j)
=   \sup_{j\in J}  H^* (A_j, B_j).
\end{align*}
So (\rmn3) is true. Using (\rmn3), we obtain
\begin{align*}
 & H  (\prod_{j\in J} A_j, \prod_{j\in J} B_j)
= H^* (\prod_{j\in J} A_j, \prod_{j\in J} B_j) \vee H^* (\prod_{j\in J} B_j, \prod_{j\in J} A_j)\\
&
=  \sup_{j\in J}  H^* (A_j, B_j) \vee  \sup_{j\in J}  H^* (B_j, A_j)
=  \sup_{j\in J}  (H^* (A_j, B_j) \vee    H^* (B_j, A_j))
= \sup_{j\in J}  H (A_j, B_j).
\end{align*}
So (\rmn4) is true.

Let $x\in \prod_{j \in J}  X_j$.
For each $j\in J$, put $A_j= \{x_j\}$ in (\rmn3)
(clearly $\{x_j\} \in C(X_j)$).
Then we obtain
$H^*(\prod_{j\in J} \{x_j\}, \prod_{j\in J} B_j) = \sup_{j\in J} H^*(\{x_j\}, B_j )$.
This means that
$d_J(x,    \prod_{j\in J} B_j)  = \sup_{j\in J}  d_j (x_j, B_j)$,
since
$d_J(x,    \prod_{j\in J} B_j) = H^*(\{x\}, \prod_{j\in J} B_j)=H^*(\prod_{j\in J} \{x_j\}, \prod_{j\in J} B_j)$ and $\sup_{j\in J}  d_j (x_j, B_j) =\sup_{j\in J} H^*(\{x_j\}, B_j )$.
So, in some sense, (\rmn2) can be seen as a special case of (\rmn3).

\end{proof}

Now, we give an example to show
that
there exists a metric space $X$ and $u,v\in F^1_{USC}(X)$ such that $d_p(u,v)$ is not well-defined; that is, $H([u]_\al, [v]_\al) $ is a non-measurable function
of $\al$ on
$[0,1]$.

\begin{eap}\label{nmf}
  {\rm
Consider the metric space $([0, 100] \setminus \{10\}, \rho_1|_{[0, 100] \setminus \{10\}})$, which is subspace of $\mathbb{R}$.
We also use $[0, 100] \setminus \{10\}$
to denote this metric space.
For each $z\in (0,1]$,
define $u^z \in F_{USC}^1 ([0,100] \setminus \{10\})$ by putting
\[
[u^z]_\al
=\left\{
  \begin{array}{ll}
     \{3\},  &  \al\in [z,1],  \\
  \{ 3\} \cup (10, 11-\varepsilon], & \al=  z \cdot\varepsilon,  \   0 \leq \varepsilon < 1.
  \end{array}
 \right.
\]
For each $z\in (0,1]$, define $v^z \in F_{USC}^1 ([0,100] \setminus \{10\})$ by putting
\[
[v^z]_\al
=\left\{
  \begin{array}{ll}
     \{73\},  &  \al\in (z,1],  \\
 \mbox{} [71,81], & \al \in [0,z].
  \end{array}
 \right.
\]
(see also clauses (\rmn1) and (\rmn2) at the end of this example.)
Then for each $z\in (0,1]$,
\begin{footnotesize}
 \begin{equation}\label{uvhc}
 \ H([u^z]_\al, [v^z]_\al)
=
\left\{
  \begin{array}{ll}
    H(\{3\},  \{73\}), & \al\in (z,1], \\
     H(\{3\}, [71,81]), & \al=z, \\
 H(\{ 3\} \cup (10, 11-\varepsilon], [71,81]), &    \al = z \cdot\varepsilon,  \   0 \leq \varepsilon<1,
  \end{array}
\right.
=
\left\{
  \begin{array}{ll}
    70, & \al\in (z,1], \\
   78, & \al=z, \\
    70+\varepsilon, &    \al = z \cdot\varepsilon,  \ 0 \leq \varepsilon < 1,
  \end{array}
\right.
\end{equation}
\end{footnotesize}
where $H$ is the Hausdorff metric on $C([0,100] \setminus \{10\})$ induced by $\rho_1|_{[0,100] \setminus \{10\}}$.

Consider the metric space $([0,9], \rho_1|_{[0,9]})$,
which is a subspace of $\mathbb{R}$.
Define
$w\in F([0,9])$ as
$w(t) =1$ for all $t\in [0,9]$.
As $[w]_\al= [0,9]$ for all $\al\in [0,1]$, it follows that
$w\in F^1_{USCB} ([0,9]) = F^1_{USC} ([0,9])$.

Let $A$ be a non-measurable set in $(0,1]$.
Put
 $u := \prod_{z\in [0,1]} u_z$
and
$v := \prod_{z\in [0,1]} v_z  $,
where
\begin{gather*}
u_z=\left\{
             \begin{array}{ll}
               u^z, & z\in A, \\
w, & z\in [0,1] \setminus A,
             \end{array}
           \right.
\
v_z=\left\{
             \begin{array}{ll}
               v^z, & z\in A, \\
w, & z\in [0,1] \setminus A.
             \end{array}
           \right.
\end{gather*}
Then by Theorem \ref{pfn},
$u$ and $v$ are fuzzy sets in $F_{USC}^1 ( \prod_{z\in [0,1]} X_z )$,
where
 \[ X_z=\left\{
     \begin{array}{ll}
[0,100] \setminus \{10\} , & z\in A,\\
\mbox{}[0,9], & z\in [0,1] \setminus A.
     \end{array}
   \right.
\]
Here we mention that $\prod_{z\in [0,1]} X_z$ is endowed with the metric $\lambda$ defined by \eqref{upm};
that is,
$
\lambda(x,y)= \sup\{ |x_z - y_z|:  z \in [0,1]\}
$
for each $x=(x_z)_{z \in [0,1]}$ and each $y=(y_z)_{z \in [0,1]}$ in $\prod_{z\in [0,1]} X_z$.

In the rest of this example, we use
$H$ to denote the Hausdorff metric on $C(\prod_{z\in [0,1]} X_z)$ induced by $\lambda$, use $H$ to denote the Hausdorff metric on $C([0,9])$ induced by $\rho_1|_{[0,9]}$, and use $H$ to denote the Hausdorff metric on $C([0,100] \setminus \{10\})$ induced by $\rho_1|_{[0,100] \setminus \{10\}}$.

For each $\al\in [0,1]$,
\begin{align*}
&H([u]_\al, [v]_\al ) = H(\prod_{z\in [0,1]}[u_z]_\al, \prod_{z\in [0,1]} [v_z]_\al ) = \sup_{z\in [0,1] } H( [u_z]_\al, [v_z]_\al   ) \mbox{  \hspace{1mm}    (by \eqref{pdf} and by Theorem \ref{hnp}(\rmn4)) } \\
&=\sup_{z\in A } H( [u^z]_\al, [v^z]_\al   ) \vee \sup_{z\in[0,1] \setminus A } H( [w]_\al, [w]_\al   )\\
& =\sup_{z\in A } H( [u^z]_\al, [v^z]_\al   )   \mbox{  \hspace{3cm} (since $H( [w]_\al, [w]_\al   )=0$) }     \\
& \left\{
    \begin{array}{ll}
      = 78, & \al \in A, \\
    \leq 71, & \al\in [0,1] \setminus A.
    \end{array}
  \right.   \mbox{  \hspace{1cm} (by \eqref{uvhc}. See also below clause (\rmn3))}
\end{align*}
So
$\{\al\in [0,1]: H([u]_\al, [v]_\al ) > 73\} =A$,
and thus
$H([u]_\al, [v]_\al) $ is a non-measurable function
of $\al$ on
$[0,1]$.

Below illustrations are easy to see.

(\rmn1) Let $z\in (0,1]$. Note that
$[u^z]_\al = \cap_{\beta<\al} [u^z]_{\beta}$ for all $\al\in (0,1]$
and $[u^z]_0 = \overline{\cup_{\al>0} [u^z]_{\al}}$, where $\overline{\mbox{} \,  \cdot \, \mbox{}}$ means the closure in $[0,100] \setminus \{10\}$.
Thus by Corollary \ref{repc}, $u^z \in  F([0,100] \setminus \{10\})$.
Clearly $[u^z]_1\not= \emptyset$
and $[u^z]_\al \in C([0,100] \setminus \{10\})$ for all $\al\in [0,1]$.
This means that $u^z\in F_{USC}^1 ([0,100] \setminus \{10\})$.

(\rmn2) Replace $u^z$ by $v^z$ in (\rmn2), then we obtain
a proof for the conclusion that $v^z\in F_{USC}^1 ([0,100] \setminus \{10\})$
for each $z\in (0,1]$.

(\rmn3)
By \eqref{uvhc},
for each $z\in (0,1]$ and each $\al\in [0,1]$, if $z=\al$, then
$H([u^z]_\al, [v^z]_\al)=78$;
if $z\not=\al$, then $H([u^z]_\al, [v^z]_\al)<71$.
So if $ \al \in A$, then
$\sup_{z\in A } H( [u^z]_\al, [v^z]_\al   )
     = H( [u^\alpha]_\al, [v^\alpha]_\al   ) = 78$;
if $\al\in [0,1] \setminus A$, then $z\not=\al$ for all $z\in A$,
and hence $\sup_{z\in A } H( [u^z]_\al, [v^z]_\al   ) \leq 71$.

}
\end{eap}

In the sequel, we give some improvements of Propositions \ref{gmn}
and \ref{lcpm} and Theorem \ref{rcm},
which are
 statements first presented in \cite{huang17}.
We first prove
Theorem \ref{cme}, which is an improvement of
Propositions \ref{gmn}
and \ref{lcpm}.
Then we show
  Theorem
\ref{crm} and use it to improve Theorem \ref{cme} and Theorem \ref{rcm}.

Let $v\in F^1_{USC} (X)$
and let $0\leq \alpha < \beta \leq 1$. The ``variation''
$w_v(\alpha,\beta)$ is defined as
$w_v(\alpha, \beta):= \sup \{H([v]_\xi, [v]_\eta)  : \xi, \eta \in (\al, \beta]   \}$.

\begin{tm} \label{cme}
   Let $u\in F^1_{USC} (X)$ and let $v\in F^1_{USCG} (X)$. Then
$H([u]_\al, [v]_\al)$ is a measurable function of $\al$ on $[0,1]$; that is, $d_p(u,v)$ is well-defined.
\end{tm}

\begin{proof}
The proof is divided into three steps.
We put easy proofs of some simple conclusions in clauses (\uppercase\expandafter{\romannumeral1}),  (\uppercase\expandafter{\romannumeral2}) and (\uppercase\expandafter{\romannumeral3}) at the end of this proof.

\textbf{Step 1} \
Prove statement (\romannumeral1)
 $H^*([u]_\al, [v]_\al)$ is a measurable function of $\al$ on $[0,1]$.

For each $\xi\in \mathbb{R}$ and each $n\in \mathbb{N}\setminus\{1\}$,
define
\begin{gather*}
  S_\xi := \{\al\in [0,1]:  H^*([u]_\al, [v]_\al) \geq \xi\},
 \\
  S_{\xi, n} := S_\xi \cap (\frac{1}{n}, 1].
\end{gather*}

 To show statement (\rmn1),
it suffices to
verify
statement (\rmn1-a) for each $\xi\in \mathbb{R}$ and each $n\in \mathbb{N}\setminus\{1\}$,
$S_{\xi, n}$ is a measurable set (see also  (\uppercase\expandafter{\romannumeral1})).
Fix
 $\xi\in \mathbb{R}$ and $n\in \mathbb{N}\setminus\{1\}$.

Since $v \in F^1_{USCG} (X)$, by Lemma 6.5 in \cite{huang17},
for each $k\in \mathbb{N}$,
there exist $\al_1^{(k)}, \cdots, \al_{l_k}^{(k)}$
such that
$\frac{1}{n} = \al_1^{(k)}  < \cdots < \al_{l_k}^{(k)} = 1$
and
$w_v(\al_i^{(k)}, \al_{i+1}^{(k)}) \leq \frac{1}{k}$ for all $i=1, \ldots, l_k-1$.

For each $k\in \mathbb{N}$ and
 each $i=1, \ldots, l_k-1$,
define
$T_{k,i} :=  \{ x: \ \mbox{there exists} \ s \in S_\xi  \ \mbox{such that}  \ \al_i^{(k)} < x \leq s \leq \al_{i+1}^{(k)} \}   $; in other words,
$T_{k,i} = \cup \{ (\al_i^{(k)}, s] : s \in S_\xi \cap (\al_i^{(k)}, \al_{i+1}^{(k)}]  \}$. Here we mention that the result of an empty union is $\emptyset$.
Put
$T_k := \cup_{i=1}^{l_k-1} T_{k,i}$.
Obviously for each $k\in \mathbb{N}$ and
 each $i=1, \ldots, l_k-1$,
$T_{k,i} \subseteq (\al_i^{(k)}, \al_{i+1}^{(k)}] $ and therefore
$T_k = \cup_{i=1}^{l_k-1} T_{k,i} \subseteq \cup_{i=1}^{l_k-1}(\al_i^{(k)}, \al_{i+1}^{(k)}] = (1/n,1]$.

We claim that
\\
(\romannumeral1-1) for each $k\in \mathbb{N}$, $T_k$ is a measurable set,
\\
(\romannumeral1-2) for each $k\in \mathbb{N}$, $T_k \supseteq S_{\xi, n}$, and
\\
(\romannumeral1-3) for each $k\in \mathbb{N}$, $T_k \subseteq S_{\xi-\frac{1}{k}, n}    $.

First we show (\rmn1-1). Let $k\in \mathbb{N}$. Clearly
for each $i\in \{1, \ldots, l_k-1\}$, $T_{k,i}$ is an interval (see also
(\uppercase\expandafter{\romannumeral2})).
Thus
$T_k= \cup_{i=1}^{l_k-1} T_{k,i}$ is a measurable set.
Hence (\romannumeral1-1) is true.

(\romannumeral1-2) follows immediately from the definition of $T_k$. See also (\uppercase\expandafter{\romannumeral3}).

Now we show (\rmn1-3). Let $k\in \mathbb{N}$.
Let $i\in \{1, \ldots, l_k-1\}$ and $x \in T_{k,i}$. There exists an
 $s \in S_\xi $ such that $\al_i^{(k)} < x \leq s \leq \al_{i+1}^{(k)}$.
Note that
\\
(a-1) By Remark \ref{hme}(b), $H^*([u]_{s}, [v]_x) + H^*([v]_x, [v]_{s})
\geq H^*([u]_{s}, [v]_{s})$. Then, as
$H^*([v]_x, [v]_{s}) $ is finite,
$H^*([u]_{s}, [v]_x)= H^*([u]_{s}, [v]_x) + H^*([v]_x, [v]_{s})
- H^*([v]_x, [v]_{s})
\geq H^*([u]_{s}, [v]_{s}) - H^*([v]_x, [v]_{s})$;
\\
(b-1) $H^*([v]_x, [v]_{s}) \leq H([v]_x, [v]_{s})\leq w_v(\al_i^{(k)}, \al_{i+1}^{(k)}) \leq \frac{1}{k}$.
\\
Thus
\begin{align*}
H^*([u]_x, [v]_x)
\geq
H^*([u]_{s}, [v]_x)
\geq H^*([u]_{s}, [v]_{s}) - H^*([v]_x, [v]_{s})
 \geq \xi - 1/k.
\end{align*}
Hence $x\in S_{\xi-\frac{1}{k}}$.
Since $i\in \{1, \ldots, l_k-1\}$ and $x \in T_{k,i}$ are arbitrary,
we conclude that
$T_k = \cup_{i=1}^{l_k-1} T_{k,i}\subseteq S_{\xi-\frac{1}{k}}$.
Then, as
$T_k \subseteq (\frac{1}{n}, 1]$, $T_k \subseteq  S_{\xi-\frac{1}{k}}\cap (\frac{1}{n}, 1]=  S_{\xi-\frac{1}{k}, n}$. Thus (\romannumeral1-3) is proved.

From (\romannumeral1-2) and (\romannumeral1-3), we have
\begin{equation}\label{seu}
  S_{\xi, n}  \subseteq \bigcap_{k=1}^{+\infty} T_k \subseteq \bigcap_{k=1}^{+\infty} S_{\xi-\frac{1}{k}, n} = S_{\xi,n}.
\end{equation}
By
(\romannumeral1-1),
$\bigcap_{k=1}^{+\infty} T_k   $ is measurable. Thus by \eqref{seu}, $S_{\xi,n} = \bigcap_{k=1}^{+\infty} T_k$ is measurable.
So statement (\rmn1-a) is true and then
statement (\rmn1) is proved.

\textbf{Step 2} \
Prove statement (\rmn2) $H^*([v]_\al, [u]_\al)$ is a measurable function of $\al$ on $[0,1]$.

For each $\xi\in \mathbb{R}$ and each $n\in \mathbb{N}\setminus \{1\}$,
define
\begin{gather*}
  S^\xi := \{\al\in [0,1]:  H^*([v]_\al, [u]_\al) \geq \xi\},
 \\
  S^{\xi, n} := S^\xi \cap (\frac{1}{n}, 1].
\end{gather*}

To show statement (\rmn2),
it suffices to
show statement (\rmn2-a) for each $\xi\in \mathbb{R}$ and each $n\in \mathbb{N}\setminus \{1\}$,
$S^{\xi, n}$ is a measurable set (see also  (\uppercase\expandafter{\romannumeral1})). Fix
 $\xi\in \mathbb{R}$ and $n\in \mathbb{N}\setminus\{1\}$.

For each $k\in \mathbb{N}$ and
 each $i=1, \ldots, l_k-1$,
define
$
T^{k,i} :=  \{ x: \ \mbox{there exists} \ s \in S^\xi \ \mbox{such that} \ \al_i^{(k)} < s \leq x \leq \al_{i+1}^{(k)}\}
$; in other words,
$T^{k,i} = \cup \{ [s, \al_i^{(k+1)}] : s \in S^\xi \cap (\al_i^{(k)}, \al_{i+1}^{(k)}]  \}$. Put
$T^k := \cup_{i=1}^{l_k-1} T^{k,i}$.
Obviously for each $k\in \mathbb{N}$ and
 each $i=1, \ldots, l_k-1$,
$T^{k,i} \subseteq (\al_i^{(k)}, \al_{i+1}^{(k)}] $ and therefore
$T^k = \cup_{i=1}^{l_k-1} T^{k,i} \subseteq \cup_{i=1}^{l_k-1}(\al_i^{(k)}, \al_{i+1}^{(k)}] = (1/n,1]$.

We claim
that
\\
(\rmn2-1) for each $k\in \mathbb{N}$, $T^k$ is a measurable set,
\\
(\rmn2-2) for each $k\in \mathbb{N}$, $T^k \supseteq S^{\xi, n}$, and
\\
(\rmn2-3) for each $k\in \mathbb{N}$, $T^k \subseteq S^{\xi-\frac{1}{k}, n}$.

First we show (\rmn2-1). Let $k\in \mathbb{N}$. Clearly
for each $i\in \{1, \ldots, l_k-1\}$, $T^{k,i}$ is an interval (see also  (\uppercase\expandafter{\romannumeral2})).
Thus
$T^k= \cup_{i=1}^{l_k-1} T^{k,i}$ is a measurable set.
Hence (\romannumeral2-1) is true.

(\romannumeral2-2) follows from the definition of $T^k$.
See also (\uppercase\expandafter{\romannumeral3}).

Now we show (\rmn2-3). Let $k\in \mathbb{N}$.
Let $i\in \{1, \ldots, l_k-1\}$ and
 $x \in T^{k,i}$. Then there exists an
 $s \in S^\xi $ such that $\al_i^{(k)} < s \leq x \leq \al_{i+1}^{(k)}$.
Note that
\\
(a-2) By Remark \ref{hme}(b),
 $H^*([v]_s, [v]_x) + H^*([v]_x, [u]_s) \geq H^*([v]_s, [u]_s)$. Then, as
$H^*([v]_s, [v]_{x}) $ is finite,
$H^*([v]_x, [u]_s)  =H^*([v]_s, [v]_x) + H^*([v]_x, [u]_s) - H^*([v]_s, [v]_x)
\geq H^*([v]_s, [u]_s) - H^*([v]_s, [v]_x)
$;
\\
(b-2) $H^*([v]_s, [v]_{x}) \leq H([v]_s, [v]_{x})\leq w_v(\al_i^{(k)}, \al_{i+1}^{(k)}) \leq \frac{1}{k}$.
\\
 Thus
\begin{align*}
H^*([v]_x, [u]_x) \geq H^*([v]_x, [u]_s) \geq H^*([v]_s, [u]_s) - H^*([v]_s, [v]_x) \geq \xi - 1/k.
\end{align*}
Hence
$x\in S^{\xi-\frac{1}{k}}$.
Since $i\in \{1, \ldots, l_k-1\}$ and $x \in T^{k,i}$ are arbitrary,
we conclude that
$T^k = \cup_{i=1}^{l_k-1} T^{k,i}\subseteq S^{\xi-\frac{1}{k}}$.
Then, as
$T^k \subseteq (\frac{1}{n}, 1]$, $T^k \subseteq  S^{\xi-\frac{1}{k}}\cap (\frac{1}{n}, 1]=  S^{\xi-\frac{1}{k}, n}$. Thus (\romannumeral2-3) is proved.

From (\romannumeral2-2) and (\romannumeral2-3),
\begin{equation}\label{pseu}
  S^{\xi, n} \subseteq \bigcap_{k=1}^{+\infty} T^k \subseteq \bigcap_{k=1}^{+\infty} S^{\xi-\frac{1}{k}, n} = S^{\xi,n}.
\end{equation}
By (\romannumeral2-1), $\bigcap_{k=1}^{+\infty} T^k$ is measurable. Thus, by \eqref{pseu}, $S^{\xi,n} = \bigcap_{k=1}^{+\infty} T^k$ is measurable.
So statement (\rmn2-a) is true and then
statement (\rmn2) is proved.

\textbf{Step 3} Prove statement (\rmn3)
 $H([u]_\al, [v]_\al)$ is a measurable function of $\al$ on $[0,1]$.

Note that
 $H([u]_\al, [v]_\al) = \max \{H^*([u]_\al, [v]_\al), H^*([v]_\al, [u]_\al)\}$.
Thus
statement (\rmn3) follows immediately
from statements (\rmn1) and (\rmn2), which say
that both $H^*([u]_\al, [v]_\al)$ and $H^*([v]_\al, [u]_\al)$ are measurable functions of $\al$ on $[0,1]$.
So the proof is completed.

\vspace{1mm}

Contents in the following clauses (\uppercase\expandafter{\romannumeral1}), (\uppercase\expandafter{\romannumeral2}) and (\uppercase\expandafter{\romannumeral3}) are easy to see.

(\uppercase\expandafter{\romannumeral1})
Consider statement (\rmn1-b) for each $\xi\in \mathbb{R}$,
$S_\xi$ is a measurable set.
Then statement (\rmn1), statement (\rmn1-a) and statement(\rmn1-b)
are equivalent to each other
(Obviously (\rmn1)$\Leftrightarrow$(\rmn1-b)$\Rightarrow$(\rmn1-a).
For each $\xi\in \mathbb{R}$, $\cup_{n=2}^{+\infty} S_{\xi, n}= S_\xi \cap (0,1] $ and therefore
$S_\xi = \cup_{n=2}^{+\infty} S_{\xi, n}$
or
$S_\xi = \cup_{n=2}^{+\infty} S_{\xi, n} \cup \{0\}$.
So (\rmn1-a)$\Rightarrow$(\rmn1-b). Thus (\rmn1)$\Leftrightarrow$(\rmn1-b)$\Leftrightarrow$(\rmn1-a).).
So
to show statement (\rmn1),
it suffices to
verify
statement (\rmn1-a).

Contents in the previous paragraph from ``Consider'' to the end are still correct if we replace
\rmn1, $S_\xi$ and $S_{\xi, n}$ everywhere by \rmn2, $S^\xi$ and $S^{\xi, n}$, respectively.

(\uppercase\expandafter{\romannumeral2})  Let $k\in \mathbb{N}$ and $i\in \{1, \ldots, l_k-1\}$. Then $T_{k,i}$ and $T^{k,i}$ are intervals.

Denote $S(i,k) := S_\xi \cap (\al_i^{(k)}, \al_{i+1}^{(k)}]$.
Let $x,y\in T_{k,i}$ with $x<y$. Then there exist $s_x$ and $s_y$
in $S(i,k) $ such that $x\in (\al_i^{(k)}, s_x]$ and $y\in (\al_i^{(k)}, s_y]$.
There is no loss of generality in assuming that $s_x\leq s_y$.
Then
$\{x,y\}\subsetneqq(\al_i^{(k)}, s_y]$.
Thus
$ [x,y]\subsetneqq (\al_i^{(k)}, s_y] \subseteq T_{k,i}$.
So
$T_{k,i}$ is an interval.

Clearly $T_{k,i}=\emptyset \Leftrightarrow S(i,k) = \emptyset$. Assume
$S(i,k) \not= \emptyset$. Put $s_1 =\sup S(i,k)$.
Then $T_{k,i} = (\al_i^{(k)}, s^1] \Leftrightarrow s_1 \in S(i,k)$,
and
$T_{k,i} = (\al_i^{(k)}, s_1)\Leftrightarrow s_1\notin S(i,k)$.

Denote $S'(i,k) := S^\xi \cap (\al_i^{(k)}, \al_{i+1}^{(k)}]$.
Let $x,y\in T^{k,i}$ with $x<y$. Then there exist $s^x$ and $s^y$
in $S'(i,k) $ such that $x\in [s^x, \al_i^{(k+1)}]$ and $y\in [s^y, \al_i^{(k+1)}]$.
There is no loss of generality in assuming that $s^x\leq s^y$.
Then
$\{x,y\}\subsetneqq [s^x, \al_i^{(k+1)}]$.
Thus
$ [x,y]\subseteq [s^x, \al_i^{(k+1)}]\subseteq T^{k,i}$.
So
$T^{k,i}$ is an interval.

Clearly $T^{k,i}=\emptyset \Leftrightarrow S'(i,k) = \emptyset$.
Assume
$S'(i,k) \not= \emptyset$. Put $s_2 =\inf S'(i,k)$.
 Then $s_2 \in S'(i,k) \Leftrightarrow T^{k,i} = [s_2, \al_i^{(k+1)}]$, and
$s_2 \notin S'(i,k)\Leftrightarrow T^{k,i} = (s_2, \al_i^{(k+1)}]$.
Clearly $s_2 = \al_i^{(k+1)} \Leftrightarrow S'(i,k)= \{\al_i^{(k+1)} \} \Leftrightarrow T^{k,i}= \{\al_i^{(k+1)}\} \Leftrightarrow$ $T^{k,i}$ is a singleton set (in this case,  $s_2\in S'(i,k)$).

(\uppercase\expandafter{\romannumeral3}) To show (\rmn1-2),
let $k\in \mathbb{N}$. For each $i\in \{1, \ldots, l_k-1\}$, $T_{k,i} \supseteq S_\xi  \cap (\al_i^{(k)}, \al_{i+1}^{(k)}]$.
So
$T_k= \cup_{i=1}^{l_k-1} T_{k,i}\supseteq \cup_{i=1}^{l_k-1} (S_\xi  \cap (\al_i^{(k)}, \al_{i+1}^{(k)}]) = S_\xi \cap (\cup_{i=1}^{l_k-1}  (\al_i^{(k)}, \al_{i+1}^{(k)}]) = S_\xi \cap (1/n,1]= S_{\xi, n}$.
Thus (\romannumeral1-2) is true.

To show (\rmn2-2),
let $k\in \mathbb{N}$. For each $i\in \{1, \ldots, l_k-1\}$, $T^{k,i} \supseteq S^\xi  \cap (\al_i^{(k)}, \al_{i+1}^{(k)}]$.
So
$T^k= \cup_{i=1}^{l_k-1} T^{k,i}\supseteq \cup_{i=1}^{l_k-1} (S^\xi  \cap (\al_i^{(k)}, \al_{i+1}^{(k)}]) = S^\xi \cap (\cup_{i=1}^{l_k-1}  (\al_i^{(k)}, \al_{i+1}^{(k)}]) = S^\xi \cap (1/n,1]= S^{\xi, n}$.
Thus (\romannumeral2-2) is true.

\end{proof}

\begin{re}
  {\rm
(\rmn1)
Theorem \ref{cme} is an improvement of
Proposition \ref{lcpm}.

(\rmn2)
Let $u\in F^1_{USC}(X)$ and $x\in X$. Then $\widehat{x}\in F^1_{USCG}(X)$.
Thus by Theorem \ref{cme}, $d_p(u, \widehat{x})$
is well-defined. It is easy to see that
$H([u]_\al, \{x_0\}) $ is a measurable function of $\al$ on $[0,1]$ if and only if $d_p(u, \widehat{x_0})$ is well-defined.
So
Proposition \ref{gmn}
is an easy consequence of Theorem \ref{cme},
and, in some sense,
Theorem \ref{cme} is an improvement of Proposition \ref{gmn}.

(\rmn3) Obviously, if $\xi\leq 0$, then $S_\xi=S^\xi=[0,1]$ and $S_{\xi, n} = S^{\xi, n} = (\frac{1}{n}, 1]$ for all $n\in \mathbb{N}\setminus \{1\}$.
}
\end{re}

Let $(Z, \rho)$ be a metric space and $S$ be a subset of $Z$.
We will use
 $\overline{S}^{(Z, \rho)}$ to denote
the closure of $S$
in
 $(Z, \rho)$ when we want to emphasis the space $(Z, \rho)$. If there is no
confusion, we will write $\overline{S}^{(Z, \rho)}$ as $\overline{S}^Z$.

Let $(X, d_X)$ be a subspace of $(Y, d_Y)$.
For each $u\in F(X)$, define
$\bm{u^Y \in   F_{USC}(Y)}$ by putting
\begin{equation}\label{uyn}
 [u^Y]_\al = \cap_{\beta<\al} \overline{[u]_\beta }^Y    \ \mbox{for all} \ \al\in (0,1].
\end{equation}

Note that
$
  [u^Y]_\al= \cap_{\beta<\al} [u^Y]_\beta \mbox{ for all } \al\in (0,1]
$ (for each $ \al\in (0,1]$,
$\cap_{\beta<\al} [u^Y]_\beta = \cap_{\beta<\al} \cap_{\gamma<\beta}  \overline{[u]_\gamma }^Y
=\cap_{\gamma<\al}  \overline{[u]_\gamma }^Y =    [u^Y]_\al
$).
Then by Theorem \ref{rep},
 $u^Y$ is a well-defined fuzzy set in $Y$.
By \eqref{uyn}, for each $\al\in (0,1]$,
 $[u^Y]_\al \in C(Y) \cup \{\emptyset\}$.
So $u^Y \in F_{USC}(Y)$.

\begin{pp} \label{egcn}
Let $(X, d_X)$ be a subspace of $(Y, d_Y)$ and $u\in F(X)$.
\\
(\romannumeral1)
$[u^Y]_\al \supseteq  \overline{[u]_\al}^Y \supseteq [u]_\al$ for all $\al\in (0,1]$.
\\
(\romannumeral2)
$\overline{[u]_\beta}^Y \supseteq [u^Y]_\al$ for all $\al,\beta$ with
$0\leq \beta < \al \leq 1$.
\\
(\romannumeral3)  $[u^Y]_0 = \overline{[u]_0}^Y$.
\\
(\rmn4)  If $u\in F^1(X)$, then $u^Y \in F^1_{USC}(Y)$. However, the converse is false.
\\
(\rmn5)   If $u\in F_{USCG}(X)$, then $[u^Y]_\al=[u]_\al$
for all $\al\in (0,1]$, and so $u^Y \in F_{USCG}(Y)$.
\\
(\rmn6)  (\rmn6-1) If $u\in F_{USC}^1(X)$, then $u^Y \in F^1_{USC}(Y)$.
 (\rmn6-2) If $u\in F_{USCG}^1(X)$, then $u^Y \in F^1_{USCG}(Y)$.
  (\rmn6-3) $u^Y \in F^1_{USCG}(Y)$ does not imply $u\in F_{USC}(X)$;
and so the converse of (\rmn6-2) is false.
\\
(\rmn7) $u\in F_{USC}(X)$ if and only if $u^X= u$.
\end{pp}

\begin{proof}
Let $\al\in (0,1]$. Obviously $\overline{[u]_\al}^Y \supseteq [u]_\al$.
Clearly
for each $\beta\in [0,\al)$, $\overline{[u]_\beta }^Y  \supseteq  \overline{[u]_\alpha }^Y$.
Then, by \eqref{uyn},
 $[u^Y]_\al = \cap_{\beta<\al} \overline{[u]_\beta }^Y \supseteq  \overline{[u]_\alpha }^Y$. So (\romannumeral1) is proved.

Let $\al\in (0,1]$ and $\beta\in [0, \al)$. Then
$[u^Y]_\al = \cap_{\gamma<\al} \overline{[u]_\gamma}^Y \subseteq\overline{[u]_\beta}^Y$. Hence (\romannumeral2) is proved.

By (\rmn2), for each $\al\in (0,1]$, $[u^Y]_\alpha \subseteq \overline{[u]_0}^Y$.
Hence
 $[u^Y]_0 =\overline{\cup_{\al>0} [u^Y]_\alpha}^Y      \subseteq  \overline{[u]_0}^Y$ ($\subseteq$ holds because $ \cup_{\al>0} [u^Y]_\alpha      \subseteq  \overline{[u]_0}^Y$ and that $\overline{[u]_0}^Y$ is closed in $(Y, d_Y)$).
Also
 $[u^Y]_0 =\overline{\cup_{\al>0} [u^Y]_\alpha}^Y       \supseteq   \overline{\cup_{\al>0} [ u]_\alpha }^Y (\mbox{by (\rmn1)}) = \overline{\overline{\cup_{\al>0} [ u]_\alpha}^X }^Y (\mbox{by (a) below}) =\overline{[u]_0}^Y$.
Thus
 $[u^Y]_0 = \overline{[u]_0}^Y$ and (\rmn3) is proved.

Now we show (\rmn4). If $u\in F^1(X)$, then $[u^Y]_1 \not= \emptyset$ as
$[u^Y]_1 \supseteq [u]_1$, by (\rmn1). So $u^Y \in F^1(Y)$.
Also
$u^Y \in F_{USC}(Y)$. Thus
$u^Y \in F^1_{USC}(Y) = F^1(Y)\cap F_{USC}(Y)$.

The following example shows that $u^Y \in F^1_{USC}(Y)$ does not imply $u\in F^1(X)$.

We also use $\mathbb{R} \setminus \{1\}$ to denote
the metric space $(\mathbb{R} \setminus \{1\}, \rho_1|_{\mathbb{R} \setminus \{1\}})$. $\mathbb{R} \setminus \{1\}$ is a subspace of $\mathbb{R}$.

Define
$v\in F(\mathbb{R} \setminus \{1\})$ by
putting
$
[v]_\al=\left\{
       \begin{array}{ll}
        [\al,1), & \al\in [0,1), \\
   \emptyset, & \al = 1.
       \end{array}
     \right.
$
Then $v\not\in F^1(\mathbb{R} \setminus \{1\})$.
Clearly
$
\overline{[v]_\al}^\mathbb{R}=\left\{
       \begin{array}{ll}
        [\al,1], & \al\in [0,1), \\
   \emptyset, & \al = 1.
       \end{array}
     \right.
$
So
$
[v^\mathbb{R}]_\al=  [\al,1]$ for all $ \al\in [0,1]$.
Thus
 $v^\mathbb{R} \in F^1_{USCG}(\mathbb{R})\subsetneqq F^1_{USC}(\mathbb{R})$. So (\rmn4) is proved.
See (b) for further discussions.

We show (\rmn5) as follows.
For each $\al\in (0,1]$, $[u]_\al \in K(X)\cup \{\emptyset\}$,
and hence $[u]_\al \in K(Y)\cup \{\emptyset\} \subseteq C(Y)\cup \{\emptyset\}$;
so $\overline{[u]_\al}^Y = [u]_\al $.
Thus for each $\al\in (0,1]$,
$[u^Y]_\al  = \cap_{\beta<\al} \overline{[u]_\beta}^Y = \cap_{\beta\in (0,\al)} \overline{[u]_\beta}^Y =\cap_{\beta\in (0,\al)} [u]_\beta = [u]_\al$;
so
$[u^Y]_\al  \in K(Y)\cup \{\emptyset\}$.
Hence
$u^Y \in F_{USCG}(Y)$.
Thus (\rmn5) is proved.

 If $u\in F_{USC}^1(X)$, then $u\in F^1(X)$ and so, by (\rmn4), $u^Y \in F^1_{USC}(Y)$. Thus (\rmn6-1) is true.
If $u\in F_{USCG}^1(X)= F^1(X)\cap F_{USCG}(X)$, then by (\rmn4)(\rmn5), $u^Y \in F^1_{USC}(Y)\cap F_{USCG}(Y)= F^1_{USCG}(Y)$.
Thus (\rmn6-2) is true.

Define
$w\in F(\mathbb{R} \setminus \{1\})$ by
putting
$
[w]_\al=\left\{
       \begin{array}{ll}
[0,1), & \al=0, \\ \mbox{}
        [\al,1) \setminus \{1/2\}, & \al\in (0,1), \\
   \emptyset, & \al = 1.
       \end{array}
     \right.
$
Note that $[w]_{\al} \not\in C(\mathbb{R} \setminus \{1\}) \cup \{\emptyset\}$ if and only if $\al\in (0, 1/2]$.
So $w \notin F_{USC}(\mathbb{R} \setminus \{1\})$.
Observe that for each $\al\in [0,1]$,
$
\overline{[w]_\al}^\mathbb{R}=
\overline{[v]_\al}^\mathbb{R}$. So for each $\al\in (0,1]$,
$[w^\mathbb{R}]_\al= [v^\mathbb{R}]_\al$ (and therefore $[w^\mathbb{R}]_0= [v^\mathbb{R}]_0$ ).
Thus
$w^\mathbb{R}= v^\mathbb{R} \in F^1_{USCG}(\mathbb{R})$.
Hence $u^Y \in F^1_{USCG}(Y)$ does not imply $u\in F_{USC}(X)$.
So (\rmn6-3) is true and
  (\rmn6) is proved.
Note that $v^\mathbb{R} \in F^1_{USCG}(\mathbb{R})$
but $v\notin F_{USCG}(\mathbb{R} \setminus \{1\})$ (see (b) below).
From this fact, we can also know that the converse of (\rmn6-2) is false.

Finally we show
(\rmn7). If $u^X= u$, then $u\in F_{USC}(X)$ as $u^X \in F_{USC}(X)$.
If
 $u\in F_{USC}(X)$, then for each $\al\in (0,1]$,
$[u^X]_\al = \cap_{\beta<\al} \overline{[u]_\beta}^X
=\cap_{\beta<\al} [u]_\beta = [u]_\al$.
Thus, by Theorem \ref{rep},
 $u^X= u$ as both $u^X$ and $u$ are fuzzy sets in $X$.

(a) Let $S\subseteq X$. Then $\overline{\overline{S}^X}^Y = \overline{S}^Y$
(Clearly $\overline{\overline{S}^X}^Y \supseteq \overline{S}^Y$ as $ \overline{S}^X  \supseteq  S$.
On the other hand, $\overline{\overline{S}^X}^Y \subseteq \overline{S}^Y$
because $ \overline{S}^X  \subseteq \overline{S}^Y$ and that $\overline{S}^Y$ is closed in $(Y,d_Y)$.).

(b) In fact $v\in F_{USC} (\mathbb{R} \setminus \{1\})\setminus F_{USCG} (\mathbb{R} \setminus \{1\})$ as $[v]_\al$ is closed in $\mathbb{R} \setminus \{1\}$ for all $\al\in [0,1]$ and $[v]_\al$ is not compact in $\mathbb{R} \setminus \{1\}$ for all $\al\in [0,1)$.
We can see that $
v(x)=\left\{
       \begin{array}{ll}
        x, & x\in [0,1), \\
   0, & x\in (\mathbb{R} \setminus \{1\}) \setminus  [0,1),
       \end{array}
     \right.
$ and
$
v^{\mathbb{R}}(x)=\left\{
       \begin{array}{ll}
        x, & x\in [0,1], \\
   0, & x\in \mathbb{R} \setminus  [0,1].
       \end{array}
     \right.
$

\end{proof}

\begin{pp}
  Let $(X, d_X)$ be a subspace of $(Y, d_Y)$ and let
$(Y, d_Y)$ be a subspace of $(Z, d_Z)$.
If $u\in F(X)$,
then
$u^Z= (u^Y)^Z$.

\end{pp}

\begin{proof}

As $u^Z$ and $(u^Y)^Z$ are fuzzy sets in $Z$,
by Proposition \ref{ace}(\rmn2), to verify
$u^Z= (u^Y)^Z$ it suffices to show that for each $\al\in (0,1]$, $[u^Z]_\al= [(u^Y)^Z]_\al$.

Let $\al\in (0,1]$.
 Given $\beta<\al$. Take $\gamma$ satisfying
that $\beta<\gamma < \al$.
Then
$\overline{[u]_\beta }^Z = \overline{\overline{[u]_\beta }^Y}^Z \supseteq  \overline{ [u^Y]_{\gamma} }^Z \supseteq [(u^Y)^Z]_\al $
(the two ``$\supseteq$'' follow from Proposition \ref{egcn}(\rmn2)).
Thus
$[u^Z]_\al = \cap_{\beta<\al} \overline{[u]_\beta }^Z \supseteq  [(u^Y)^Z]_\al $.
On the other hand, $[u^Z]_\al = \cap_{\beta<\al} \overline{[u]_\beta }^Z \subseteq \cap_{\beta<\al} \overline{[u^Y]_\beta }^Z =[(u^Y)^Z]_\al $ (by Proposition \ref{egcn}(\rmn1), $\subseteq$ holds).
So $[u^Z]_\al = [(u^Y)^Z]_\al $. As $\al\in (0,1]$ is arbitrary, we have that
$u^Z= (u^Y)^Z$.

\end{proof}

Let $(X, d_X)$ be a subspace of $(Y, d_Y)$.
For each $u\in F(X)$, define
$$\Upsilon(u)^Y := \{\al\in (0,1]: [u^Y]_\al \supsetneqq \overline{[u]_\al}^Y   \}, \  \sigma(u)^Y :=  [0,1] \setminus \Upsilon(u)^Y.$$
Clearly $\sigma(u)^Y  = ((0,1]\setminus\Upsilon(u)^Y) \cup \{0\}$.

We will use
the following Theorem
\ref{crm} to improve Theorem \ref{cme} and Theorem \ref{rcm}.

Let $(Z, \rho)$ be a metric space. The Hausdorff distance $H$ on $C(Z)$
  can be extended to the Hausdorff distance $H'$ on $C(Z)\cup \{\emptyset\}$ as follows:
$H'(A, B)=H(A,B)$, if $A,B \in C(Z)$;
 $H'(\emptyset, \emptyset)=0$; and
$H'(\emptyset, S)=H'(S,\emptyset)= +\infty$, if $S\in C(Z)$.
Obviously $H'$ is an extended metric on $C(Z)\cup \{\emptyset\}$.
We will use the same symbol to represent the Hausdorff distance on $C(Z)\cup \{\emptyset\}$ mentioned above and the Hausdorff distance on $C(Z)$.

Let $(Z, \rho)$ be a metric space.
We will use
 $H_{(Z, \rho)}$ to denote
the Hausdorff distance on $C(Z)\cup \{\emptyset\}$ induced by
$\rho$ on $Z$ when we want to emphasis the space $(Z, \rho)$. If there is no
confusion, we will write  $H_{(Z, \rho)}$ as $H_{Z}$.
Let $(Y, \rho|_Y)$ be a subspace of $(Z, \rho)$ and $A,B$
two subsets of $Y$.
Then clearly $H_Y(A,B) = H_Z(A,B)$.

\begin{tm} \label{crm}
Let $(X, d_X)$ be a subspace of $(Y, d_Y)$ and $u, v \in F(X)$.
\\
(\rmn1) \
Let $\al\in [0,1]$. Then $[u^Y]_\al = \overline{[u]_\al}^Y$
if and only if
$\al\in \sigma(u)^Y $.
\\
(\romannumeral2) \ Let $u \in F_{USC} (X)$.
Then the
cardinality
 of $\Upsilon(u)^Y$ is less than the cardinality of $Y \setminus X$.
\\
(\romannumeral3) \ Let $u, v \in F_{USC} (X)$. Then
 for each $ \al \in [0,1] \setminus (\Upsilon(u)^Y \cup \Upsilon(v)^Y )$,
$$
H_Y([u^Y]_\al, [v^Y]_\al) = H_X([u]_\al, [v]_\al ).
$$
\end{tm}

\begin{proof}
First we show (\rmn1).
Consider the conditions
(\rmn1-1) $[u^Y]_\al \not= \overline{[u]_\al}^Y$,
(\rmn1-2) $[u^Y]_\al \supsetneqq \overline{[u]_\al}^Y$, and
(\rmn1-3) $\al\in \Upsilon(u)^Y $.
By
Proposition \ref{egcn}(\rmn1)(\rmn3),
(\rmn1-1)$\Leftrightarrow$(\rmn1-2) and
(\rmn1-2)$\Leftrightarrow$(\rmn1-3).
  Thus
(\rmn1-1)$\Leftrightarrow$(\rmn1-3).
This means that $[u^Y]_\al = \overline{[u]_\al}^Y \Leftrightarrow
\al\in  [0,1] \setminus \Upsilon(u)^Y = \sigma(u)^Y $.
So (\rmn1) is proved.

To show that (\romannumeral2) is true,
it suffices to construct an injection
$j: \Upsilon(u)^Y \to Y\setminus X$.

Let $\gamma \in       \Upsilon(u)^Y$.
Then there is an
$x_\gamma \in Y$
such that
 $x_\gamma\in [u^Y]_\gamma \setminus \overline{[u]_\gamma}^Y$.
Define $j(\gamma) =  x_\gamma$
for
each
$\gamma \in \Upsilon(u)^Y$.
Since $x_\gamma\notin  \overline{[u]_\gamma}^Y$, we have
$x_\gamma\notin [u]_\gamma = \cap_{\beta<\gamma} [u]_\beta$. So there is a $\beta < \gamma$
such
that
$x_\gamma\notin [u]_\beta = \overline{[u]_\beta}^X$.
Also $x_\gamma\in [u^Y]_\gamma \subseteq \overline{[u]_\beta}^Y$ (by Proposition \ref{egcn}(\rmn2), $\subseteq$ holds).
So $x_\gamma\in  \overline{[u]_\beta}^Y \setminus \overline{[u]_\beta}^X$
and
thus
$x_\gamma\in    Y \setminus X$ (see also (a) below).
Hence $j$ is an function from $\Upsilon(u)^Y$ to $Y\setminus X$.

Let $\xi, \eta \in \Upsilon(u)^Y$ with $\xi < \eta$. Since $x_\xi \notin \overline{[u]_\xi}^Y$ and $\overline{[u]_\xi}^Y \supseteq [u^Y]_\eta$ (by Proposition \ref{egcn}(\rmn2)),
we have that
  $x_\xi \notin [u^Y]_\eta$. But $x_\eta \in [u^Y]_\eta$,
and therefore $x_\xi \not= x_\eta$.
Thus $j$ is an injection.
So (\romannumeral2) is proved.

Now we show (\rmn3).
Clearly
 $[0,1] \setminus (\Upsilon(u)^Y \cup \Upsilon(v)^Y ) =   \sigma(u)^Y \cap \sigma(v)^Y $.
Thus by (\rmn1),
for
each $ \al \in [0,1] \setminus (\Upsilon(u)^Y \cup \Upsilon(v)^Y )$,
$[u^Y]_\al = \overline{[u]_\al}^Y$ and $[v^Y]_\al = \overline{[v]_\al}^Y$.
So
$$
H_Y([u^Y]_\al, [v^Y]_\al) = H_Y(\overline{[u]_\al}^Y,  \overline{[v]_\al}^Y    )   = H_X([u]_\al, [v]_\al ), $$
and
(\romannumeral3) is proved.

(a) Let $A$ be a subset of $X$. If $a\in \overline{A}^Y \setminus \overline{A}^X$, then $a\in Y \setminus X$.

Clearly $a\in Y$.
As $a\in \overline{A}^Y$, there exist a sequence $\{a_n\}$ in $A$
such that $d_Y(a_n, a) \to 0$. Assume $a\in X$. Then $d_X(a_n, a) \to 0$. So
$a\in \overline{A}^X$, which is a contradiction.
  Hence $a\notin X$.
Thus $a\in Y \setminus X$.

\end{proof}

Let $S$ be a subset of $\mathbb{R}$, and $P(x)$ a statement about real numbers $x$. If
there exists a set $E$ of measure zero such that $P(x)$ holds for all
$x \in S\setminus E$,
then we say that $P(x)$ holds almost everywhere on $S$,
or
 that $P(x)$ holds almost everywhere on $x\in S$ if we want to
emphasis $x$.
We also write ``almost everywhere'' as ``a.e.'' for simplicity.

The following fact is well-known.
\\
Let $S$ be a measurable set of $\mathbb{R}^m$.
Let $f_1$ and $f_2$ be
two functions of $S$ to $\widehat{\mathbb{R}}$ with $x$ as the
independent variable. Assume that $f_1(x)=f_2(x)$
 almost everywhere on $x\in S$. Then $f_1$ is a measurable function on $S$ if and only if $f_2$ is a measurable function on $S$.

\begin{tl} \label{mpn} Let $(X, d_X)$ be a subspace of $(Y, d_Y)$ and $Y \setminus X$ an at most countable set.
 Then
for each $u, v \in F^1_{USC} (X)$,
$H_Y([u^Y]_\al, [v^Y]_\al) $ is a measurable function of $\al$ on $[0,1]$
if and only if
 $H_X([u]_\al, [v]_\al) $ is a measurable function of $\al$ on $[0,1]$.

\end{tl}

\begin{proof}
By Theorem \ref{crm}(\romannumeral3),
$H_Y([u^Y]_\al, [v^Y]_\al)= H_X([u]_\al, [v]_\al)$
for each $ \al \in [0,1] \setminus (\Upsilon(u)^Y \cup \Upsilon(v)^Y )$.
Also, by Theorem \ref{crm}(\romannumeral2),
$\Upsilon(u)^Y \cup \Upsilon(v)^Y $ is at most countable, and therefore
$\Upsilon(u)^Y \cup \Upsilon(v)^Y $ is a set of measure zero.
So
$H_Y([u^Y]_\al, [v^Y]_\al)= H_X([u]_\al, [v]_\al)$
almost everywhere
on $\al \in [0,1]$.
Thus $H_Y([u^Y]_\al, [v^Y]_\al) $ is a measurable function of $\al$ on $[0,1]$
if and only if
 $H_X([u]_\al, [v]_\al) $ is a measurable function of $\al$ on $[0,1]$.

\end{proof}

Let $S\subseteq \mathbb{R}^m$.
We also use $\mathbb{R}^m \setminus  S$
to denote the metric space $(\mathbb{R}^m \setminus S, \rho_m|_{\mathbb{R}^m \setminus S})$, which is
 a subspace of
$\mathbb{R}^m$.

\begin{tm} \label{reg}
Let
 $S$ be an at most countable subset of $\mathbb{R}^m$.
Let    $u,v \in  F^1_{USC} (\mathbb{R}^m \setminus  S)$. Then $H_{\mathbb{R}^m \setminus S}([u]_\al, [v]_\al) $ is a measurable function of $\al$ on $[0,1]$;
 that is, $d_p(u,v)= \left(\int_0^1 H_{\mathbb{R}^m \setminus S} ([u]_\al, [v]_\al)^p  \,   d\al   \right)^{1/p}$ is well-defined.

\end{tm}

\begin{proof}
Put $Y=\mathbb{R}^m$. By Proposition \ref{egcn}(\rmn6-1) $u^Y, v^Y \in F^1_{USC} (Y) $. Then from Theorem \ref{rcm}, $H_Y([u^Y]_\al, [v^Y]_\al) $ is a measurable function of $\al$ on $[0,1]$. So by Corollary \ref{mpn}, $H_{\mathbb{R}^m \setminus S}([u]_\al, [v]_\al) $ is a measurable function of $\al$ on $[0,1]$.

\end{proof}

\begin{tm} \label{cmeg}
  Let $(X, d_X)$ be a subspace of $(Y, d_Y)$ and $Y \setminus X$ an at most countable set.
 Let $u, v \in F^1_{USC} (X)$.
If
$u^Y \in  F^1_{USCG} (Y)$, then $H_X([u]_\al, [v]_\al) $ is a measurable function of $\al$ on $[0,1]$; that is, $d_p(u,v)= \left(\int_0^1 H_X([u]_\al, [v]_\al)^p  \,   d\al   \right)^{1/p}$ is well-defined.
\end{tm}

\begin{proof}
Note that
$u^Y \in F^1_{USCG} (Y)$ and, by Proposition \ref{egcn}(\rmn6-1), $v^Y \in F^1_{USC} (Y)$.
Then by Theorem \ref{cme},
$H_Y([u^Y]_\al, [v^Y]_\al) $ is a measurable function of $\al$ on $[0,1]$. Thus
by Corollary \ref{mpn}, $H_X([u]_\al, [v]_\al) $ is a measurable function of $\al$ on $[0,1]$.

\end{proof}

\begin{re}{\rm
 Let $(X, d_X)$ be a subspace of $(Y, d_Y)$.
\\
(\rmn1) Proposition \ref{egcn}(\rmn6) shows that
  $u\in F^1_{USCG}(X)$ implies
$u^Y \in  F^1_{USCG}(Y)$, but that the converse is false.
So Theorem \ref{cmeg} is an improvement
of
  Theorem \ref{cme}.
  By Proposition \ref{egcn}(\rmn7), for each $u\in F^1_{USC} (X)$,
$u=u^X$. So
Theorem \ref{cme} is the special case of Theorem \ref{cmeg} when $Y=X$.
\\
(\rmn2)
Theorem \ref{rcm} is the special case of Theorem \ref{reg}
when $S=\emptyset$. Theorem \ref{reg} is an improvement of
Theorem \ref{rcm}.

}
\end{re}

Let $(X, d_X)$ be a subspace of $(Y, d_Y)$.
For each $u\in F(X)$, define
$\bm{u_Y \in   F(Y)}$ as
\[
 u_Y (t)= \left\{
         \begin{array}{ll}
           u(t), & t\in X, \\
0,& t\in Y\setminus X.
         \end{array}
       \right.
\]
Then for each $\al\in (0,1]$, $[u_Y]_\al= [u]_\al$.
So $u\in F^1(X)$ if and only if $u_Y\in F^1(Y)$.
Clearly for each $u,v\in F_{USC}(X)$ with $u_Y,v_Y\in F_{USC}(Y)$  and each $\al\in [0,1]$,
$H_X([u]_\al,[v]_\al) = H_Y([u_Y]_\al,[v_Y]_\al)$.

\begin{re} \label{bgm}
 {\rm
Let $(X, d_X)$ be a subspace of $(Y, d_Y)$ and $u\in F(X)$.
For each $A$ in $X$, if $A\in C(Y)\cup \{\emptyset\}$ then $A\in C(X)\cup \{\emptyset\}$. So if $u_Y\in F_{USC}(Y)$ then $u\in F_{USC}(X)$,
and
if $u_Y\in F^1_{USC}(Y)$ then $u\in F^1_{USC}(X)$.
For each $A$ in $X$, $A\in K(X)\cup \{\emptyset\}$ if and only if $A\in K(Y)\cup \{\emptyset\}$. So $u\in F_{USCG}(X)$ if and only if $u_Y\in F_{USCG}(Y)$,
and
$u\in F^1_{USCG}(X)$ if and only if $u_Y\in F^1_{USCG}(Y)$.
 }
\end{re}

\begin{pp} \label{abcu}
Let $(X, d_X)$ be a subspace of $(Y, d_Y)$.
Then the following conditions are equivalent:
  (\rmn1) $X$ is closed in $(Y,d_Y)$;
(\rmn2) $u\in F^1_{USC}(X)$ is equivalent to  $u_Y\in F^1_{USC}(Y)$;
(\rmn3) $u\in F_{USC}(X)$ is equivalent to  $u_Y\in F_{USC}(Y)$.
\end{pp}

\begin{proof}
  It is known that (\rmn1) is equivalent to property (a)
for each $A$ in $X$, $A\in C(X)\cup \{\emptyset\}$ if and only if $A\in C(Y)\cup \{\emptyset\}$ ((\rmn1)$\Rightarrow$(a) is well-known. If (a) is true, then $X\in C(Y)$ since $X\in C(X)$. Thus (a)$\Rightarrow$(\rmn1).).
 So (\rmn1)$\Rightarrow$(\rmn3).
(\rmn3)$\Rightarrow$(\rmn2) follows
from the facts that
$u\in F^1(X) \Leftrightarrow u_Y\in F^1(Y)$,
$F^1_{USC}(X)= F^1(X) \cap F_{USC}(X)$ and $F^1_{USC}(Y)= F^1(Y) \cap F_{USC}(Y)$.
Assume that (\rmn2) is true.
Then $X_{F(Y)} = (X_{F(X)})_Y \in F^1_{USC}(Y)$ since $X_{F(X)} \in F^1_{USC}(X)$.
Thus
$X= [X_{F(Y)}]_1 \in C(Y)$; that is, (\rmn1) is true.
So (\rmn2)$\Rightarrow$(\rmn1).
This completes the proof.

\end{proof}

Let $S\subseteq \mathbb{R}^m$.
We also use $S$
to denote the metric space $(S, \rho_m|_{S})$, which is
 a subspace of
$\mathbb{R}^m$.

\begin{tm} \label{regn}
Let
 $A$ be a nonempty subset of $\mathbb{R}^m$ with $\overline{A}^{\mathbb{R}^m}\setminus A$ being at most countable.
Let    $u,v \in  F^1_{USC} (A)$.
Then $H_{A}([u]_\al, [v]_\al) $ is a measurable function of $\al$ on $[0,1]$;
 that is, $d_p(u,v)= \left(\int_0^1 H_{A} ([u]_\al, [v]_\al)^p  \,   d\al   \right)^{1/p}$ is well-defined.

\end{tm}

\begin{proof} Set $Y= \overline{A}^{\mathbb{R}^m}$.
We claim that
\\
(a) $H_{A}([u]_\al, [v]_\al) $ is a measurable function of $\al$ on $[0,1]$
if and only if
$H_Y([u^Y]_\al,[v^Y]_\al) $  is a measurable function of $\al$ on $[0,1]$;
\\
(b)
for each $\al\in [0,1]$,
$H_Y([u^Y]_\al,[v^Y]_\al) = H_{\mathbb{R}^m} ([(u^Y)_{\mathbb{R}^m}]_\al, \, [(v^Y)_{\mathbb{R}^m}]_\al)$;
\\
(c) $H_{\mathbb{R}^m} ([(u^Y)_{\mathbb{R}^m}]_\al, \, [(v^Y)_{\mathbb{R}^m}]_\al)$
 is a measurable function of $\al$ on $[0,1]$.

By Corollary \ref{mpn}, (a) is true.
By Proposition \ref{egcn}(\rmn6-1), $u^Y, v^Y \in  F^1_{USC} (Y)$.
As $Y\in C(\mathbb{R}^m)$, by Proposition \ref{abcu},
$(u^Y)_{\mathbb{R}^m} \in F^1_{USC} (\mathbb{R}^m)$ and
$(v^Y)_{\mathbb{R}^m} \in F^1_{USC} (\mathbb{R}^m)$.
So (b) is true, and
by Theorem \ref{rcm},
(c) is true.

From (a), (b) and (c), it follows that $H_{A}([u]_\al, [v]_\al) $ is a measurable function of $\al$ on $[0,1]$.

\end{proof}

\begin{re}\label{rfgu}
 {\rm
Let
 $S$ be an at most countable subset of $\mathbb{R}^m$ and $u,v \in  F^1_{USC} (\mathbb{R}^m \setminus  S)$.
Put
$A=\mathbb{R}^m \setminus S$. Then $\overline{A}^{\mathbb{R}^m}= \mathbb{R}^m$
and  $\overline{A}^{\mathbb{R}^m}\setminus A = \mathbb{R}^m \setminus A = S$.
Thus by Theorem \ref{regn},
 $H_{A}([u]_\al, [v]_\al) $ is a measurable function of $\al$ on $[0,1]$.
So
Theorem \ref{regn} is an improvement of
Theorem \ref{reg}.

Note that for each $A \subseteq \mathbb{R}^m$,  $\mathbb{R}^m \setminus A$
is at most countable if and only if $ \overline{A}^{\mathbb{R}^m}=\mathbb{R}^m$ and $\overline{A}^{\mathbb{R}^m} \setminus A$
is at most countable.
So
Theorem \ref{reg} is the special case of Theorem \ref{regn}
when $\overline{A}^{\mathbb{R}^m} = \mathbb{R}^m$.
 }
\end{re}

\begin{re} \label{cmegun}
{\rm
 (a) Let $(X, d_X)$ be a subspace of $(Y, d_Y)$ and $Y \setminus X$ an at most countable set. Let $(Y, d_Y)$ be a subspace of $(Z, d_Z)$.
 Let $u, v \in F^1_{USC} (X)$.
If
$(u^Y)_Z \in  F^1_{USCG} (Z)$, then $H_X([u]_\al, [v]_\al) $ is a measurable function of $\al$ on $[0,1]$; that is, $d_p(u,v)= \left(\int_0^1 H_X([u]_\al, [v]_\al)^p  \,   d\al   \right)^{1/p}$ is well-defined.

By Remark \ref{bgm}, $(u^Y)_Z \in  F^1_{USCG} (Z)$ if and only if
$u^Y \in  F^1_{USCG} (Y)$. So conclusion (a) follows immediately from
 Theorem \ref{cmeg}. Clearly conclusion (a) is equivalent to Theorem \ref{cmeg}.
 There exist conclusions such as conclusion (a), which are easy consequences
of Theorem \ref{cmeg} and other results. We will not list them all.
}
\end{re}

The results in this section
except Theorem \ref{regn} and Remark \ref{rfgu}
were recorded in \cite{huang17c}.
In essence, contents including
Theorem \ref{crm}, Corollary \ref{mpn} and Theorem \ref{reg}
 have already been proved in chinaXiv:202108.00116v1, which is a previous version of \cite{huang17c}.

\section{Properties of $d_p^*$ metric and $H_{\rm end}$ metric} \label{per}

In this section, we give the relationship between the $d_p^*$ metric and the $H_{\rm end}$ metric.
We illustrate
the relations between
the property that $U$ is uniformly $p$-mean bounded
and other properties of $U$.

\begin{tm} \label{dpeu} Let $(X,d)$ be a metric space and let $u, v \in F^1_{USC} (X)$.
\\
(\romannumeral1)
\begin{equation}\label{dpe}
  d_p^* (u,v) \geq   \left(  \frac{H_{\rm end} (u,v)^{p+1} }{p+1}   \right)^{1/p}.
\end{equation}
\\
(\romannumeral2) \
 If a sequence $\{u_n\}$ in $F^1_{USC} (X)$ satisfies that $d_p^*(u_n, u) \to 0$ as $n\to\infty$, then
  $H_{\rm end} (u_n, u) \to 0$ as $n\to\infty$.
\end{tm}

\begin{proof}
To show (\rmn1), we only need to show that for each $r>0$,
if $H_{\rm end} (u,v) > r$ then $d_p^* (u,v) \geq \left( \frac{r^{p+1}}{p+1}  \right)^{1/p}$ (see also (\uppercase\expandafter{\romannumeral1}) at the end of this proof).

Let $r>0$.
Assume that $H_{\rm end} (u,v) > r$. Then without loss of generality we suppose that $H^*({\rm end}\,u,    {\rm end}\,v) > r$. Thus there is an $(x,\beta) \in {\rm end}\, u$ such that
$\overline{d}((x,\beta), {\rm end}\, v ) >r$.
Let $\alpha\in [0, \beta]$. We claim that (a-\rmn1) $1\geq \beta>r$,
(a-\rmn2) $d(x, [v]_{\alpha}) > r-(\beta-\al)$, and
(a-\rmn3) $H^*([u]_\alpha, [v]_{\alpha}) > r-(\beta-\al)$.

$\beta\leq 1$ since $(x,\beta) \in {\rm end}\, u$.
 $\beta= \overline{d}((x,\beta), X\times \{0\})  \geq  \overline{d}((x,\beta), {\rm end}\, v ) >r$. So (a-\rmn1) is true.
$d(x, [v]_{\alpha}) + (\beta-\al) = \overline{d}((x,\beta), [v]_{\alpha}\times \{\al\}) \geq \overline{d}((x,\beta), {\rm end}\, v ) >r$, so (a-\rmn2) is true.
By (a-\rmn1), $\beta>0$ and then $x\in  [u]_\beta \subseteq [u]_\al$.
(a-\rmn3) follows from (a-\rmn2) and the fact that $x\in [u]_\al$.

Given a measurable function $f$ on $[0,1]$
satisfying that
 $f(\al) \geq H([u]_\al, [v]_\al  )$ for all $\al\in [0,1]$.
Then, using (a-\rmn1) and (a-\rmn3), we have that
\begin{align*}
& \left( \int_0^1  f(\al)^p \,d\al \right)^{1/p}
 \geq \left( \int_{\beta-r}^{\beta}  f(\alpha)^p     \,d\alpha \right)^{1/p} \\
&\geq     \left(    \int_{\beta-r}^{\beta}  (r-(\beta-\al))^p   \,d \alpha
 \right)^{1/p}   \hspace{1cm}  \mbox{(this $\geq$ can be repalced by $>$)}
 \\
&=    \left(    \int_{\beta-r}^{\beta}  (\alpha-(\beta-r))^p   \,d \alpha \right)^{1/p}  =  \left( \frac{ (\alpha-(\beta-r))^{p+1} }{p+1}\Big|_{\beta-r}^{\beta} \right)^{1/p} =
  \left( \frac{r^{p+1}}{p+1}  \right)^{1/p}. \  \mbox{(see also (\uppercase\expandafter{\romannumeral2})) }
\end{align*}
So from the definition of $d_p^* (u,v)$, we have that
$d_p^* (u,v) \geq \left( \frac{r^{p+1}}{p+1}  \right)^{1/p}$.
Thus (\rmn1) is true.
(\romannumeral2) follows immediately from (\romannumeral1)
(see also (\uppercase\expandafter{\romannumeral3})).

\vspace{1mm}

The contents in the following
(\uppercase\expandafter{\romannumeral1})--(\uppercase\expandafter{\romannumeral3}) are basic and easy.
\\
(\uppercase\expandafter{\romannumeral1})  \  \eqref{dpe} is equivalent to statement (b) for each $r>0$,
if $H_{\rm end} (u,v) > r$ then $d_p^* (u,v) \geq \left( \frac{r^{p+1}}{p+1}  \right)^{1/p}$. Obviously \eqref{dpe}$\Rightarrow$(b).
Assume that (b) is true.
If $H_{\rm end} (u,v) = 0$, then \eqref{dpe} holds obviously.
Suppose $H_{\rm end} (u,v) > 0$.
Then
we can
choose a sequence $\{r_n\}$ in $\mathbb{R}^+$ satisfying that $r_n\to H_{\rm end} (u,v) -$ as $n\to\infty$. Then $d_p^* (u,v) \geq \lim_{n\to\infty}\left( \frac{r_n^{p+1}}{p+1}  \right)^{1/p} = \left( \frac{H_{\rm end} (u,v)^{p+1}}{p+1}  \right)^{1/p}$. Thus \eqref{dpe} is true.
Hence (b)$\Rightarrow$\eqref{dpe}. So \eqref{dpe}$\Leftrightarrow$(b).
\\
(\uppercase\expandafter{\romannumeral2})  \
$ \left(    \int_{\beta-r}^{\beta}  (\alpha-(\beta-r))^p   \,d \alpha \right)^{1/p}  =  \left( \int^{r}_0  t^p     \,d t \right)^{1/p}
= \left( \frac{r^{p+1}}{p+1}  \right)^{1/p}$.
\\
(\uppercase\expandafter{\romannumeral3})  \
We show (\rmn2) by (\rmn1) as follows.
 Given $n\in \mathbb{N}$, put $c_n:=  \left(  \frac{H_{\rm end} (u_n,u)^{p+1} }{p+1}   \right)^{1/p}$;
then
  $d_p^* (u_n,u) \geq  c_n \geq 0$, by (\rmn1) and the fact that $H_{\rm end} (u_n,u) \geq 0$.  Since $d_p^* (u_n,u) \to 0$ as $n\to\infty$, we have
that $ c_n \to 0$ as $n\to\infty$.
Note that for each $n\in \mathbb{N}$,
$H_{\rm end} (u_n,u) = ((p+1) {c_n}^p )^{\frac{1}{p+1}} $.
Thus $H_{\rm end} (u_n,u) \to 0$ as $n\to \infty$.
So (\rmn2) is proved. (If $H_{\rm end} (u_n,u) \to 0$ then obviously $c_n\to 0$.)

\end{proof}

\begin{re} \label{pse}
  {\rm
(\rmn1) Clearly, if
$d_p(u,v)$ is well-defined (at this time $d^*_p(u,v) = d_p(u,v)$),
 then
$d^*_p(u,v)$ can be replaced by $d_p(u,v)$ in Theorem \ref{dpeu}.

(\rmn2) Theorem \ref{dpeu}(\rmn2) is Theorem 6.2 in \cite{huang19c}.
 Theorem \ref{dpeu}(\rmn1) is an improvement of Theorem \ref{dpeu}(\rmn2) (i.e. Theorem 6.2 in \cite{huang19c}).

}
\end{re}

The following example shows that
  the ``='' can be obtained in \eqref{dpe}.
\begin{eap} \label{ecn}
{\rm
Define $u$ and $v$ in $F^1_{USCB}(\mathbb{R})$ as
\[
u(x)
=
\left\{
  \begin{array}{ll}
1, & x=0,\\
0.5-x, & x\in (0, 0.5],\\
0, & \mbox{otherwise},
  \end{array}
\right.
\
v(x)
=
\left\{
  \begin{array}{ll}
1, & x=0,\\
0.5, & x\in (0, 0.5],\\
0, & \mbox{otherwise}.
  \end{array}
\right.
\]
Then ${\rm end}\, u \subsetneqq {\rm end}\, v$, $H_{\rm end} (u,v) =H^*({\rm end}\, v, {\rm end}\, u)= \overline{\rho_1}((0.5,0.5), {\rm end}\, u) = 0.5$,
and
\[
H([u]_\al, [v]_\al)
=
\left\{
  \begin{array}{ll}
   0, & \al\in (0.5,1],\\
\al, & \al \in [0, 0.5].
  \end{array}
\right.
\]
Thus
$d_p(u,v) = (\int_0^{0.5}  \al^p  \,d\al)^{1/p} = \left(\frac{0.5^{p+1}}{p+1} \right)^{1/p} =  \left(  \frac{H_{\rm end} (u,v)^{p+1} }{p+1}   \right)^{1/p}$.

}
\end{eap}

The metric $\widehat{d}$ on $X \times [0,1]$ is defined
as
$ \widehat{d}((x,\al), (y, \beta)) = \max\{ d(x,y), |\al-\beta| \}$
 for each $(x,\al)$ and each $(y, \beta)$ in $X\times [0,1]$.

We can see
that
 for each $(x,\al)$ and each $(y, \beta)$ in $X\times [0,1]$,
\begin{align}\label{mern}
&\widehat{d}((x,\al), (y, \beta)) \leq \overline{d}((x,\al), (y, \beta))
= d(x,y) + |\alpha-\beta| \nonumber
\\
&\leq \min\{ 2\widehat{d}((x,\al), (y, \beta)) , \  d(x,y) +1  \}
\mbox{\ \  \ (note that $|\alpha-\beta|\leq |1-0|\leq  1$)} \nonumber
\\
&\leq \min\{ 2\widehat{d}((x,\al), (y, \beta)) , \  \widehat{d}((x,\al), (y, \beta)) +1  \}.
\end{align}
By \eqref{mern}, $\widehat{d}$ induces the same topology
 on $X\times [0,1]$ as $\overline{d}$.
 $C(X\times [0,1])$ is the set of all non-empty closed subsets of $(X\times [0,1],\overline{d})$.
Then $C(X\times [0,1])$ is also the set of all non-empty closed subsets of $(X\times [0,1], \widehat{d})$.

We use $\bm{H'}$ to denote the Hausdorff extended metric on $C(X\times [0,1])$ induced by $\widehat{d}$ on $X \times [0,1]$.
We use
$\bm{H'^{*}}$ to denote the Hausdorff pre-distance related to $H'$, which is induced by $\widehat{d}$ on $X \times [0,1]$ (see Section \ref{bas}).

The endograph metric $H_{\rm end}' $ and the sendograph extended metric $H_{\rm send}' $
can be defined on $F^1_{USC}(X)$ by using
the Hausdorff extended
metric $H'$ on $C(X \times [0,1])$ as follows.
For each $u,v \in F^1_{USC}(X)$,
\begin{gather*}
\bm{  H_{\rm end}'(u,v)    }: =  H'({\rm end}\, u,  {\rm end}\, v ),
\\
\bm{  H_{\rm send}'(u,v)    }: =  H'({\rm send}\, u,  {\rm send}\, v ).
  \end{gather*}
 Clearly for each $u,v \in F^1_{USC} (X)$,
$
  d_\infty(u,v) \geq H_{\rm send}'(u,v) \geq H_{\rm end}'(u,v).
$

\begin{re}
  {\rm
We can see that
each of
the $H'$ on $C(X \times [0,1])$ and the $H_{\rm send}'$ on $F^1_{USC}(X)$ is an extended metric but may not a metric.

We can see that both $H'$ on $C(\mathbb{R}^m \times [0,1])$ and $H_{\rm send}'$ on $F^1_{USCG}(\mathbb{R}^m)$ are not metrics, they are extended metrics.

For simplicity,
in this paper, we call the Hausdorff extended
metric $H'$ on $C(X \times [0,1])$ the Hausdorff
metric $H'$,
and call
$H_{\rm send}'$ on $F^1_{USC}(X)$ the $H_{\rm send}'$ metric or the sendograph metric $H_{\rm send}'$.
}
\end{re}

In some references, the sendograph metric and the endograph metric refer to $H_{\rm send}'$ and $H_{\rm end}'$, respectively.
In the following, we give some conclusions on
the relationship between $H_{\rm send}$ and
$H_{\rm send}'$, and the relationship between $H_{\rm end}$ and $H_{\rm end}'$.

We say a function $f: \mathbb{R}^+ \to \mathbb{R}^+$ is increasing
if
$f(x) \leq f(y)$ when $x < y$.

\begin{pp}\label{fdhc}
  Let $f$ be an increasing and continuous function from $\mathbb{R}^+$ to $\mathbb{R}^+$.
Let $Z$ be a set and let $d_1$ and $d_2$ be two metrics on
$Z$. If $d_1(\xi,\eta) \leq f(d_2(\xi,\eta))$ for all $\xi, \eta \in Z$, then
for each $S \subseteq Z$ and each $W \subseteq Z$,
\\
(\romannumeral1) \ for each $\xi \in Z$, $d_1(\xi,S)   \leq    f(d_2(\xi, S))$,
\\
(\romannumeral2) \ $H^*_1(W,S)  \leq     f(H^*_2(W,S) )$, and
\\
(\romannumeral3) \ $H_1(W,S)  \leq     f(H_2(W,S) )$,
\\
where
$H_1$ and $H_2$ are the Hausdorff extended metrics induced by $d_1$ and $d_2$, respectively,
$H^*_1$ and $H^*_2$ are the Hausdorff pre-distance related to $H_1$ and $H_2$, respectively.
\end{pp}

\begin{proof} For each $\xi \in Z$,
$d_1(\xi,S) = \inf_{\eta \in S} d_1(\xi,\eta)    \leq \inf_{\eta\in S} f(d_2(\xi,\eta)) = f(\inf_{\eta\in S} d_2(\xi,\eta)) = f(d_2(\xi, S)).$
So (\romannumeral1) is true.

By (\romannumeral1),
$H^*_1(W,S) = \sup_{\xi\in W} d_1(\xi,S)    \leq \sup_{\xi\in W} f(d_2(\xi,S)) = f(\sup_{\xi\in W} d_2(\xi,S)) = f(H^*_2(W,S) )$.
So (\romannumeral2) is true.

By (\romannumeral2),
\begin{align*}
&H_1(W,S) = H^*_1(W,S) \vee   H^*_1(S,W) \nonumber \\
& \leq  f(H^*_2(W,S)) \vee  f( H^*_2(S,W))    \nonumber  \\
&= f(H^*_2(W,S) \vee  H^*_2(S,W)) = f(H_2(W,S) ).
\end{align*}
So (\romannumeral3) is true.

\end{proof}

\begin{pp}\label{senr}
 Let $(X,d)$ be a metric space
and $u,v \in  F^1_{USC}(X)$.
\\
(\rmn1) $H_{\rm send}' (u,v) \leq H_{\rm send} (u,v) \leq \min\{ 2 H_{\rm send}' (u,v)  , \  H_{\rm send}' (u,v) +1  \}$,
\\
(\rmn2) $H_{\rm end}' (u,v) \leq H_{\rm end} (u,v) \leq  \min\{ 2 H_{\rm end}' (u,v)  , \ 1  \}$.

\end{pp}

\begin{proof}

Define two functions $f_1$ and $f_2$
from $\mathbb{R}^+$ to $\mathbb{R}^+$
given
by
 $f_1(x)=x$ and $f_2(x) = \min\{2x, x+1\}$.
Then $f_1$ and $f_2$
are
increasing and continuous functions.

By \eqref{mern}, for each $(x,\al)$ and each $(y, \beta)$ in $X\times [0,1]$,
$\widehat{d}((x,\al), (y, \beta)) \leq f_1(\overline{d}((x,\al), (y, \beta))) $
and
$\overline{d}((x,\al), (y, \beta)) \leq f_2( \widehat{d}((x,\al), (y, \beta))) $.
Then,
by Proposition \ref{fdhc}(\romannumeral3),
$H_{\rm send}' (u,v) \leq H_{\rm send} (u,v) \leq \min\{ 2 H_{\rm send}' (u,v)  , \  H_{\rm send}' (u,v) +1  \}$;
by Proposition \ref{fdhc}(\romannumeral3) and the fact that $H_{\rm end} (u',v') \leq 1$ for all $u',v' \in  F^1_{USC}(X)$, we obtain that
$
 H_{\rm end}' (u,v) \leq H_{\rm end} (u,v) \leq \min\{ 2 H_{\rm end}' (u,v)  , \  H_{\rm end}' (u,v) +1, \ 1  \}
= \min\{ 2 H_{\rm end}' (u,v)  , \ 1  \}$.
So (\rmn1) and (\rmn2) are proved.

\end{proof}

\begin{re} \label{bpcn}
  {\rm
Let $u$ and $u_n$, $n=1,2,\ldots$, in $F^1_{USC}(X)$.
(\romannumeral1) By Proposition \ref{senr}(\rmn1), $H_{\rm send}'(u_n,u) \to 0$ if and only if $H_{\rm send}(u_n,u) \to 0$.
(\romannumeral2) By Proposition \ref{senr}(\rmn2), $H_{\rm end}'(u_n,u) \to 0$ if and only if $H_{\rm end}(u_n,u) \to 0$.

}
\end{re}

The following
Proposition \ref{dpeumr} and Example \ref{eym} illustrate the relationship between $d_p^*$
and
$H'_{\rm end}$ on $F^1_{USC} (X)$. Proposition \ref{dpeumr} and Example \ref{eym} are
 parallel conclusions of
Theorem \ref{dpeu}
and Example \ref{ecn}, respectively.
The idea of proving Proposition \ref{dpeumr} is similar to that of proving Theorem \ref{dpeu}.
The idea of constructing
 Example \ref{eym} is similar to that of constructing Example \ref{ecn}.

\begin{pp} \label{dpeumr} Let $(X,d)$ be a metric space and let $u, v \in F^1_{USC} (X)$.
\\
(\romannumeral1)
\begin{equation}\label{dperym}
  d_p^* (u,v) \geq    H_{\rm end}' (u,v)  ^{1+1/p}.
\end{equation}
\\
(\romannumeral2)
  If a sequence $\{u_n\}$ in $F^1_{USC} (X)$ satisfies that
$d_p^*(u_n, u) \to 0$ as $n\to\infty$, then
  $H_{\rm end}' (u_n, u) \to 0$ as $n\to\infty$.
\end{pp}

\begin{proof}
To show (\romannumeral1), we only need to show that for each $r>0$,
if $H_{\rm end}' (u,v) > r$ then $d_p^* (u,v) \geq r^{1+1/p}$.
 (see also (\uppercase\expandafter{\romannumeral1}) at the end of this proof).

Let $r>0$.
Assume that $H_{\rm end}' (u,v) > r$.
Then without loss of generality we can suppose that $H'^*({\rm end}\,u,    {\rm end}\,v) > r$.
Thus there is an $(x,\beta) \in {\rm end}\, u$ with
$\widehat{d}((x,\beta), {\rm end}\, v ) >r$.
Let $\alpha\in [\beta-r, \beta]$.
We claim that (a-\rmn1) $1\geq\beta > r$,
(a-\rmn2) $d(x, [v]_{\alpha}) > r$,
 and (a-\rmn3)
$H^*([u]_\alpha, [v]_{\alpha}) > r $.

$\beta\leq 1$ since $(x,\beta) \in {\rm end}\, u$.
 $\beta= \widehat{d}((x,\beta), X\times \{0\})  \geq  \widehat{d}((x,\beta), {\rm end}\, v ) >r$, so (a-\rmn1) is true.
  Assume that $d(x, [v]_{\alpha}) \leq r$. Then
$ \widehat{d}((x,\beta), {\rm end}\, v )\leq \widehat{d}((x,\beta), [v]_{\alpha}\times \{\al\}) = \max\{d(x, [v]_{\alpha} ), |\beta-\al|\} \leq r$ (see also (\uppercase\expandafter{\romannumeral2})). This is a contradiction. So (a-\rmn2) is true.
By (a-\rmn1), $\beta>0$ and then $x\in  [u]_\beta \subseteq [u]_\al$. (a-\rmn3) follows from (a-\rmn2) and the fact that $x\in [u]_\al$.

Given
 a measurable function $f$ on $[0,1]$
satisfying that
 $f(\al) \geq H([u]_\al, [v]_\al  )$ for all $\al\in [0,1]$.
Then, using (a-\rmn1) and (a-\rmn3), we have that
\begin{align*}
& \left( \int_0^1  f(\al)^p \,d\al \right)^{1/p}
 \geq \left( \int_{\beta-r}^{\beta}  f(\alpha)^p     \,d\alpha \right)^{1/p} \\
&\geq     \left(    \int_{\beta-r}^{\beta}  r^p   \,d \alpha \right)^{1/p}  \mbox{\ \ \ (this $\geq$ can be repalced by $>$)}\\
&= (r^p\cdot (\beta-(\beta-r)))^{1/p} = r^{1+1/p}.
\end{align*}
So from the definition of $d_p^* (u,v)$, we have that
$d_p^* (u,v) \geq r^{1+1/p}$. Thus (\rmn1) is true.
(\romannumeral2) follows immediately from (\romannumeral1) (see also (\uppercase\expandafter{\romannumeral3})).
So the proof is completed.

We can also show (\rmn2) as follows.
By Theorem \ref{dpeu}(\rmn2), $H_{\rm end} (u_n, u) \to 0$ as $n\to\infty$.
Then by Remark \ref{bpcn}(\rmn2), $H_{\rm end}' (u_n, u) \to 0$ as $n\to\infty$.
So (\romannumeral2) is true.

\vspace{1mm}

The contents in the following
(\uppercase\expandafter{\romannumeral1})--(\uppercase\expandafter{\romannumeral3}) are basic and easy.

(\uppercase\expandafter{\romannumeral1})  \  \eqref{dperym} is equivalent to statement (b) for each $r>0$,
if $H_{\rm end}' (u,v) > r$ then $d_p^* (u,v) \geq r^{1+1/p}$. Obviously \eqref{dperym}$\Rightarrow$(b).
Assume that (b) is true.
If $H_{\rm end}' (u,v) = 0$, then \eqref{dperym} holds obviously.
Suppose $H_{\rm end}' (u,v) > 0$.
Then
we can
choose a sequence $\{r_n\}$ in $\mathbb{R}^+$ satisfying that $r_n\to H_{\rm end}' (u,v) -$ as $n\to\infty$. Then $d_p^* (u,v) \geq \lim_{n\to\infty} r_n^{1+1/p} = H_{\rm end}' (u,v)^{1+1/p}  $. Thus \eqref{dperym} is true.
Hence (b)$\Rightarrow$\eqref{dperym}. So \eqref{dperym}$\Leftrightarrow$(b).

(\uppercase\expandafter{\romannumeral2})  \
Let $x\in X$, $\al,\beta \in [0,1]$ and $v\in F^1_{USC}(X)$.
Then $\widehat{d}((x,\beta), [v]_{\alpha}\times \{\al\})=\max\{d(x, [v]_{\alpha} ), |\beta-\al|\}$.

Denote $\xi:=\widehat{d}((x,\beta), [v]_{\alpha}\times \{\al\})$
 and $\eta:=\max\{d(x, [v]_{\alpha} ), |\beta-\al|\}$.
 Then
 $\xi= \inf\{\widehat{d}((x,\beta), (y,\al)): (y,\al)\in [v]_{\alpha} \times \{\al\} \}= \inf\{\max\{d(x, y), |\beta-\al|\}: y\in [v]_{\alpha}\}$. Clearly $\xi \geq \eta$.
 Choose a sequence $\{y_n\}$ in $[v]_\al$ such that $d(x, [v]_{\alpha}) = \lim_{n\to\infty} d(x,y_n)$. Then  $\eta=\max\{\lim_{n\to\infty} d(x,y_n), |\beta-\al|\} =\lim_{n\to\infty} \max\{d(x,y_n), |\beta-\al|\}  =\lim_{n\to\infty} \widehat{d}((x,\beta), (y_n, \al)   ) \geq \xi$.
 So $\xi=\eta$.

(\uppercase\expandafter{\romannumeral3})  \
We show (\rmn2) by (\rmn1) as follows.
 Given $n\in \mathbb{N}$, put $c_n:=  H_{\rm end}' (u_n, u)  ^{1+1/p}$;
then
  $d_p^* (u_n,u) \geq  c_n \geq 0$, by (\rmn1) and the fact that $H_{\rm end}' (u_n,u) \geq 0$.  Since $d_p^* (u_n,u) \to 0$ as $n\to\infty$, we have
that $ c_n \to 0$ as $n\to\infty$.
Note that for each $n\in \mathbb{N}$,
$H_{\rm end}' (u_n,u) =  {c_n}^{\frac{1}{1+\frac{1}{p}}} $.
Thus $H_{\rm end}' (u_n,u) \to 0$ as $n\to \infty$.
So (\rmn2) is proved. (If $H_{\rm end}' (u_n,u) \to 0$ then obviously $c_n\to 0$.)

(a-\rmn2) can also be shown as follows.
Choose an $\varepsilon>0$ such that $\widehat{d}((x,\beta), {\rm end}\, v ) \geq r+\varepsilon$.
Given $y\in [v]_\al$. Then $ (y, \al) \in {\rm end}\, v$ and hence $r+\varepsilon \leq \widehat{d}((x,\beta), {\rm end}\, v )\leq \widehat{d}((x,\beta), (y, \al)) =\max\{d(x, y), |\beta-\al|\}$.
As $|\beta-\al|\leq r$, we have that
$r+\varepsilon \leq d(x,y)$.
Since $y\in [v]_\al$ is arbitrary, it follows that
$r+\varepsilon \leq d(x,[v]_\al)$, thus (a-\rmn2) is true.

(a-\rmn2) can also be shown as follows.
 Assume that $d(x, [v]_{\alpha}) \leq r$. Given $\varepsilon>0$. Then
 there is a
 $y\in [v]_{\alpha}$ such that $d(x,y) < r+\varepsilon$.
 As $ (y, \al) \in {\rm end}\, v$, we have that
$ \widehat{d}((x,\beta), {\rm end}\, v )\leq \widehat{d}((x,\beta), (y, \al))=\max\{d(x, y), |\beta-\al|\} \leq r+\varepsilon$. Since $\varepsilon>0$
is arbitrary, it follows that $\widehat{d}((x,\beta), {\rm end}\, v) \leq r$. This is a contradiction. So (a-\rmn2) is true.
\end{proof}

\begin{re} \label{psem}
  {\rm
Clearly, if
$d_p(u,v)$ is well-defined (at this time $d^*_p(u,v) = d_p(u,v)$),
 then
$d^*_p(u,v)$ can be replaced by $d_p(u,v)$ in Proposition \ref{dpeumr}.

}
\end{re}
The following example shows that
  the ``='' can be obtained in \eqref{dperym}.
\begin{eap} \label{eym}
{\rm
Define $u$ and $v$ in $F^1_{USCB}(\mathbb{R})$ as
\[
u(x)
=
\left\{
  \begin{array}{ll}
1, & x=0,\\
0.5, & x\in (0, 0.5],\\
0, & \mbox{otherwise},
  \end{array}
\right.
\
v(x)
=
\left\{
  \begin{array}{ll}
1, & x=0.5,\\
0.5, & x\in [0, 0.5),\\
0, & \mbox{otherwise}.
  \end{array}
\right.
\]
Then
$H'^* ({\rm end}\,u, {\rm end}\,v) = \overline{\rho_1} ((0,1), (0.5,1)) =0.5$
and
$H'^* ({\rm end}\,v, {\rm end}\,u)= \overline{\rho_1} ((0.5,1),(0,1)) =0.5$.
So
$H_{\rm end}' (u,v) = 0.5$.
We can see that
\[
H([u]_\al, [v]_\al)
=
\left\{
  \begin{array}{ll}
   0.5, & \al\in (0.5,1],\\
0, & \al \in [0, 0.5].
  \end{array}
\right.
\]
Thus
$d_p(u,v) = (\int_{0.5}^1  0.5^p  \,d\al)^{1/p} = 0.5^{1+1/p} =  H_{\rm end}' (u,v)^{1+1/p}$.

}
\end{eap}

\begin{re} \label{mytn}
  {\rm
(\rmn1)
Let $f$ be a function from $[\mu, \nu]$ to $\widehat{\mathbb{R}}$
and $[\xi, \eta]$
a subinterval of $[\mu, \nu]$. The restriction of $f$ to a subset $S$ of $[\mu, \nu]$ is denoted by $f|_{S}$.
\\
(\rmn1-1) If $f$
is left continuous on $(\mu, \nu]$,
then clearly $f|_{[\xi, \eta]}$
is left continuous on  $(\xi, \eta]$,
and thus $f|_{[\xi, \eta]}$ is measurable on
 $[\xi, \eta]$, by Corollary \ref{rsfc}.
\\
(\rmn1-2) If $f$ is measurable on $[\mu, \nu]$, then for each measurable subset $S$ of $[\mu, \nu]$,
$f|_{S}$ is measurable on $S$; clearly, $f|_{[\xi, \eta]}$
 is measurable on
 $[\xi, \eta]$ as $[\xi, \eta]$ is measurable.
\\
(\rmn1-3) Suppose $x$ is the independent variable of $f$ and $h\in \mathbb{R}$.
If $f(x)$ is measurable on $[\mu, \nu]$, then $f(x-h)$
 is measurable on $[\mu+h, \nu+h]$.

If $f$
is left continuous on $(\mu, \nu]$, we can also show the measurability of
$f|_{[\xi, \eta]}$ as follows.
Note that $f$ is measurable on $[\mu, \nu]$, by Corollary \ref{rsfc}. Thus $f|_{[\xi, \eta]}$
 is measurable on
 $[\xi, \eta]$. In the sequel, we also use $f$ to denote $f|_S$ if there is no confusion.

(\rmn2)
In the proof of Proposition \ref{emu},
we consider
the function
 $H( [u]_{\alpha},  [u]_{\alpha - \frac{1}{N} h}   ) $ of variable $\al$.
 We have conclusion
\\
(\rmn2-a) For each
$k=0,\ldots, N-1$,
$H( [u]_{\alpha},  [u]_{\alpha - \frac{1}{N} h}   ) $
is left continuous on
$({h- \frac{k}{N} h},  {1- \frac{k}{N} h} ] $.
Thus for each
$k=0,\ldots, N-1$, $H( [u]_{\alpha},  [u]_{\alpha - \frac{1}{N} h}   ) $
is
measurable on $[{h- \frac{k}{N} h},  {1- \frac{k}{N} h} ] $.
This means
the expression
  $ \sum_{k=0}^{N-1} \left(      \int_{h- \frac{k}{N} h}^{1- \frac{k}{N} h}  H( [u]_{\alpha},  [u]_{\alpha - \frac{1}{N} h}   )^p   \;d\alpha       \right)^{1/p}$
 appearing in the proof of Proposition \ref{emu} is well-defined.

Method 1 to prove conclusion (\rmn2-a) is
by combining (\rmn1-1) and Proposition \ref{lcpmcu}(\rmn1), which is a conclusions on left continuity of Hausdorff distance functions given in Section \ref{meau}.
Method 1 is
 as follows.
By
 Proposition \ref{lcpmcu}(\rmn1),
 $H( [u]_{\alpha},  [u]_{\alpha - \frac{1}{N} h}   ) $ is left continuous on $(h/N,1]$.
Observe that
$[{h- \frac{k}{N} h},  {1- \frac{k}{N} h} ] \subseteq [h/N,1]$
for each $k=0,\ldots, N-1$;
so, by (\rmn1-1),
 we obtain conclusion (\rmn2-a).

Method 2 to prove conclusion (\rmn2-a) is similar to the proof of
Proposition \ref{lcpmcu}(\rmn1)(\rmn2), which
 uses Proposition \ref{gnc}(\romannumeral1),
 the triangle inequality of the Hausdorff metric (see Remark \ref{tra})
 and Corollary \ref{rsfc}.

In our opinion,
methods 1 and 2 to prove conclusion (\rmn2-a) are essentially the same.

(\rmn3)
In the proof of Theorem \ref{pbe},
we consider the  expressions that contain integrations in the forms (\rmn3-1) $\left(      \int_{\xi}^{\eta}   H( [u]_{\alpha},  \{x_0\}   )^p   \;d\alpha       \right)^{1/p}$, (\rmn3-2) $\left(      \int_{\xi}^{\eta}   H( [u]_{\alpha-h},  \{x_0\}   )^p   \;d\alpha       \right)^{1/p}$, and (\rmn3-3)
$\left( \int_\xi^\eta   H( [u]_{\alpha},  [u]_{\alpha-h}   )^p   \;d\alpha   \right)^{1/p}$. Here $[\xi, \eta]$ denotes the corresponding integral intervals of the integrations.

By Proposition \ref{gmn} and (\rmn1-2), the function
$ H( [u]_{\alpha},  \{x_0\}   )$ of $\al$ is measurable on
 $[\xi, \eta]$, as these $[\xi, \eta]$ are subintervals of $[0,1]$.
So (\rmn3-1) is well-defined.
By Proposition \ref{gmn} and (\rmn1-3), the function
$ H( [u]_{\alpha-h},  \{x_0\}   )$ of $\al$ is measurable on
 $[h, 1+h]$. Then by (\rmn1-2), the function
$ H( [u]_{\alpha-h},  \{x_0\}   )$ of $\al$ is measurable on $[\xi, \eta]$,
as these $[\xi, \eta]$ are subintervals of $[h, 1+h]$.
So (\rmn3-2) is well-defined.
By
 Proposition \ref{lcpmcu}(\rmn1) and (\rmn1-1), we have that
 $H( [u]_{\alpha},  [u]_{\alpha - h}   ) $ is left continuous on $[\xi, \eta]$
and thus measurable on $[\xi, \eta]$, as these $[\xi, \eta]$ are subintervals of $[h, 1]$.
So (\rmn3-3) is well-defined.

(\rmn4)
The discussions of this paper
 involve the
  Hausdorff distance functions $f: [\mu, \nu] \to \mathbb{R}^+ \cup \{0\}$
and expressions $(\int_\mu^\nu  f(\alpha)^p \, d\al)^{1/p}$.
Some of these Hausdorff distance functions $f$ are listed in the above clauses (\romannumeral2) and (\romannumeral3).
We have conclusion (c): $f$ is
measurable on $[\mu, \nu]$,
which means that $(\int_\mu^\nu  f(\alpha)^p \, d\al)^{1/p}$ is well-defined.

One way to prove conclusion (c)
is
by
combining several of (\rmn1-1), (\rmn1-2) and (\rmn1-3) and
   conclusions on measurability (respectively, left continuity) of the Hausdorff distance functions given in Section \ref{meau}. See method 1 in clause (\rmn2).

Another way to prove conclusion (c)
is by proceeding similarly to conclusions on measurability (respectively, left continuity) of the Hausdorff distance functions given in Section \ref{meau}.
See method 2 in clause (\rmn2).

As it is easy to verify conclusion (c) by the above mentioned ways,
in the sequel, we will not explain
the well-definedness of these expressions
 $(\int_\mu^\nu  f(\alpha)^p \, d\al)^{1/p}$ one by one.

  (\rmn1-1), (\rmn1-2) and (\rmn1-3) are already known (we consider Corollary \ref{rsfc}
 to be a known conclusion, although we can not find this conclusion in the references that we can obtain.).
Thus
 conclusion (c) can be seen as corollaries of
conclusions on measurability (respectively, left continuity) of the Hausdorff distance functions given in Section \ref{meau}.

}
\end{re}

 The following two concepts are essentially proposed by
 Diamond and Kloeden \cite{da3} and Ma \cite{ma}, respectively.

For each $u\in F^1_{USCG}   (  X )$
and each $h\in [0,1)$,
 $\left( \int_h^1  H([u]_\alpha,   [u]_{\alpha-h})^p \;d\alpha   \right)^{1/p}$ is well-defined (see Proposition \ref{lcpmcu}).

\begin{df}
  \cite{da3}
  Let $u\in F^1_{USCG}   (  X )^p$. If for given $\varepsilon>0$, there is a $ \delta=\delta(u,\varepsilon) \in (0,1]$
  such that for all $0\leq h <\delta$
  $$ \left( \int_h^1  H([u]_\alpha,   [u]_{\alpha-h})^p \;d\alpha   \right)^{1/p}  <\varepsilon ,   $$
where  $1\leq p <+\infty$, then we say that $u$ is $p$-mean left continuous.

Suppose that $U$ is a set in $F^1_{USCG}   (  X )^p$.
If the above inequality holds uniformly for all $u\in U$,
then
we say that $U$ is  $p$-mean equi-left-continuous.
\end{df}

\begin{df}
   \cite{ma}
   A set $U$ in $ F^1_{USCG}   (  X )^p$
   is said to be uniformly $p$-mean bounded if there exists an $M\in [0,+\infty)$ and an $x_0\in X$
   such that
   $d_p(u, \widehat{x_0})\leq M$
   for all $u\in U$.
\end{df}

\begin{re} \label{ube}
  {\rm

Let $U$ be a set in $ F^1_{USCG}  (  X )^p$. Then the following conditions
are equivalent to each other:
(\rmn1) $U$ is uniformly $p$-mean bounded,
 (\rmn2) for each $w\in  F^1_{USCG}  (  X )^p$,
 there exists an $M(w)\in [0, +\infty)$
   such that
   $d_p(u, w)\leq M(w)$ for all $u\in U$, and
(\rmn3)
   $\sup \{d_p(u, v): u, v\in U\}< +\infty$.

If $U=\emptyset$, then obviously (\rmn1)$\Leftrightarrow$(\rmn2)$\Leftrightarrow$(\rmn3).
 Assume $U\not=\emptyset$.
(\rmn3)$\Rightarrow$(\rmn2). Suppose (\rmn3) is true.
Choose $u_0\in U$. Given $w\in  F^1_{USCG}  (  X )^p$.
Set $L:=\sup \{d_p(u, v): u, v\in U\}$
and
$M(w):= L+ d_p(u_0, w)$.
Then $M(w)\in [0,+\infty)$, and
for each $u\in U$,
 $d_p(u, w)\leq  d_p(u, u_0) + d_p(u_0, w)\leq L+ d_p(u_0, w) = M(w)$. So (\rmn2) is true.
As for each $x\in X$, $\widehat{x}\in  F^1_{USCG}  (  X )^p$,
it follows that
(\rmn2)$\Rightarrow$(\rmn1).
(\rmn1)$\Rightarrow$(\rmn3).
Suppose (\rmn1) is true. Then
for each $u,v\in U$,
 $d_p(u, v)\leq d_p(u, \widehat{x_0})+ d_p(v, \widehat{x_0})\leq 2M$.
Note that
$2M \in [0,\infty)$. So (\rmn3) is true.
Thus (\rmn1)$\Leftrightarrow$(\rmn2)$\Leftrightarrow$(\rmn3).

We say that a set $U$ is \emph{bounded} in a metric space $(Y, \rho)$ if
 there exists an $M\in [0,+\infty)$ and an $x_0\in Y$
   such that
   $\rho(x, x_0)\leq M$
   for all $x\in U$.

Let $U$ be a set in $(Y,\rho)$. Then the following conditions
are equivalent to each other:
(a) $U$ is bounded in $(Y,\rho)$,
 (b) for each $x\in Y$,
 there exists an  $M(x)\in [0, +\infty)$
   such that
   $\rho(u, x)\leq M(x)$ for all $u\in U$, and
(c)    $\sup \{\rho(u, v): u, v\in U\}< +\infty$.
The proof of (a)$\Leftrightarrow$(b)$\Leftrightarrow$(c)
is similar to that of the above (\rmn1)$\Leftrightarrow$(\rmn2)$\Leftrightarrow$(\rmn3).

Let $U$ be a set in $ F^1_{USCG}   (  X )^p$.
  Clearly, $U$ is uniformly $p$-mean bounded is equivalent to $U$ is
   bounded in $(F^1_{USCG}   (X)^p, d_p)$ (each of these two statements
is equivalent to $\sup \{d_p(u, v): u, v\in U\}< +\infty$.).

}
\end{re}

Suppose that
$U$ is a subset of $F^1_{USC} (X)$ and $\al\in [0,1]$.
For
writing convenience,
we denote
\begin{itemize}
  \item  $U(\al):= \bigcup_{u\in U} [u]_\al$, and

\item  $U_\al : =  \{[u]_\al: u \in U\}$.
\end{itemize}

Let $U$ be a
set in $F^1_{USCG}   ( X )^p$.
The following Lemma \ref{pmu} and Theorem \ref{pbe} illustrate
the relations between
the property that $U$ is uniformly $p$-mean bounded
and other properties of $U$.

\begin{lm} \label{pmu}
  Let $U$ be a
subset of $F^1_{USCG}   ( X )^p$.
If $U$ is uniformly $p$-mean bounded, then for each $h\in (0,1]$, $U(h)$ is bounded in $(X,d)$.
\end{lm}

\begin{proof}
  Assume that there is an $h_0\in (0,1]$ such that $U(h_0)$ is not bounded in $(X,d)$. Choose $x_0\in X$.
Then
$\sup \{d(x, x_0) :   x\in U(h_0)\} = +\infty$ since otherwise
$\sup \{d(x, x_0) :   x\in U(h_0)\} < +\infty$, i.e. $U(h_0)$ is bounded, which is a contradiction.
Note that $\sup \{d(x, x_0) :   x\in U(h_0)\}  =
 \sup_{u\in U} \sup\{d(x, x_0)  :   x\in [u]_{h_0} \} =
 \sup_{u\in U}  H([u]_{h_0}, \{x_0\})$ (see (a) in Remark \ref{bsm}).
Thus
$\sup \{H([u]_{h_0}, \{x_0\}) :   u\in U \} = +\infty$.

By Proposition \ref{gmn}, $H( [u]_\alpha,  \{x_0\}   )$ is a decreasing function of $\al$ on $[0,1]$.
Thus for each $u\in U$,
\begin{footnotesize}
  $$ \hspace{-3mm} \left(    \int_0^1  H( [u]_\alpha,  \{x_0\}   )^p   \;d\alpha       \right)^{1/p} \geq
  \left(    \int_0^{h_0}  H( [u]_\alpha,  \{x_0\}   )^p   \;d\alpha       \right)^{1/p} \geq \left(    \int_0^{h_0}  H( [u]_{h_0},  \{x_0\}   )^p   \;d\alpha       \right)^{1/p} =
 {h_0}^{1/p} \cdot H([u]_{h_0}, \{x_0\}). $$
\end{footnotesize}
As $\sup \{H([u]_{h_0}, \{x_0\}) :   u\in U \} = +\infty$, it follows that
$\sup \left\{\left(      \int_0^1  H( [u]_\alpha,  \{x_0\}   )^p   \;d\alpha       \right)^{1/p}: u\in U \right\} = +\infty$,
which
contradicts the assumption that $U$ is uniformly $p$-mean bounded.

So for each $h\in (0,1]$, $U(h)$ is bounded in $(X,d)$.

\end{proof}

 \begin{tm} \label{pbe}
 Let $U$ be a
subset of $F^1_{USCG}   ( X )^p$.
If $U$ is $p$-mean equi-left-continuous, then the following
three properties are equivalent:
\\
(\romannumeral1)
    There exists an $h\in (0,1)$ such that $U(h)$ is bounded in $(X,d)$;
    \\
    (\romannumeral2) For each $h\in (0,1]$, $U(h)$ is bounded in $(X,d)$;
    \\
 (\romannumeral3) $U$ is uniformly $p$-mean bounded.
 \end{tm}

\begin{proof}
  (\romannumeral1)$\Rightarrow$(\romannumeral3).
   Assume that (\romannumeral1) is true, i.e. there is an $h_1\in (0,1)$ such that
      $U(h_1)$ is bounded in $(X,d)$.
      Then there exists an $L>0$ such that
      $\sup\{ d(x,y) : x,y\in U(h_1)\}\leq L$. Put $M= L \cdot (1-h_1)^{1/p}$.
      Then for all $h\in [h_1, 1]$ and $u\in U$,
 \begin{equation}\label{bou}
\left( \int_h^1  H( [u]_\alpha,  \{x_0\}   )^p   \;d\alpha       \right)^{1/p}    \leq  L\cdot (1-h_1)^{1/p} = M.
  \end{equation}

 Since $U$ is $p$-mean equi-left-continuous, there is an $h_2 > 0$ such that
      for all $h\in [0, h_2]$ and $u\in U$,
      \begin{equation}\label{elc}
        \left(      \int_h^{1}  H( [u]_{\alpha},  [u]_{\alpha-h}   )^p   \;d\alpha       \right)^{1/p} < 1
      \end{equation}

      Choose an $h \leq \min\{1-h_1, h_2\}$ satisfying $1/h \in \mathbb{N}$.
Set $N := 1/h $.
Then by \eqref{elc} for $k=1,\ldots, N-1$ and $u\in U$,
\begin{align} \label{lce}
 &\left|
 \left(      \int_{kh}^{(k+1)h}   H( [u]_{\alpha},  \{x_0\}   )^p   \;d\alpha       \right)^{1/p}
 -\left(      \int_{(k-1)h}^{kh}  H( [u]_\alpha,  \{x_0\}   )^p   \;d\alpha       \right)^{1/p}
 \right| \nonumber
 \\
 & = \left|
 \left(      \int_{kh}^{(k+1)h}    H( [u]_{\alpha},  \{x_0\}   )^p   \;d\alpha       \right)^{1/p}
 -
 \left(      \int_{kh}^{(k+1)h}   H( [u]_{\alpha-h},  \{x_0\}   )^p   \;d\alpha       \right)^{1/p}
\right| \nonumber
\\
&\leq
  \left(      \int_{kh}^{(k+1)h}   H( [u]_{\alpha},  [u]_{\alpha-h}   )^p   \;d\alpha       \right)^{1/p}
\leq\left( \int_h^1   H( [u]_{\alpha},  [u]_{\alpha-h}   )^p   \;d\alpha   \right)^{1/p} <1,
\end{align}
 and thus by \eqref{bou} and \eqref{lce}, for all $u\in U$,
  \begin{align*}
  & \left(      \int_0^1  H( [u]_\alpha,  \{x_0\}   )^p   \;d\alpha       \right)^{1/p}   \\
 & \leq
 \sum_{k=0}^{N-1} \left(      \int_{kh}^{(k+1)h}   H( [u]_\alpha,  \{x_0\}   )^p   \;d\alpha       \right)^{1/p}
 \\
   & <  M + \cdots+ (M+(N-1))
   \\
  & = N\cdot M + N(N-1)/2,
   \end{align*}
and so
(\romannumeral3) is true.

 (\romannumeral3)$\Rightarrow$(\romannumeral2)
 follows from Lemma \ref{pmu}.

(\romannumeral2)$\Rightarrow$(\romannumeral1) is obvious.

\end{proof}

\begin{re}
{\rm
From Corollary \ref{pcnrem} and Theorem \ref{rcs},
we know
that for
 a $p$-mean equi-left-continuous set $U$ in $F^1_{USCG}   (\mathbb{R}^m )^p$,
 the properties
 (\romannumeral1) $U(\al)$
is bounded in $\mathbb{R}^m$ for each $\al \in (0,1]$,
and
 (\romannumeral2) $U$ is uniformly $p$-mean bounded,
are equivalent.
}

\end{re}

\begin{pp} \label{emu}
     Let $U$ be a
subset of $F^1_{USCG}   ( X )^p$.
If $U$ is $p$-mean equi-left-continuous, then for each $h\in [0,1]$, there exists a $C_h \in \mathbb{R}$ such that
  for all $u\in U$,
 \begin{equation*}\label{elscgr}
       \left(      \int_{h}^{1}  H( [u]_{\alpha},  [u]_{\alpha-h}   )^p   \;d\alpha       \right)^{1/p}
     \leq C_h.
           \end{equation*}

\end{pp}

\begin{proof}
Since $U$ is $p$-mean equi-left-continuous,
then
there is an $h_0>0$ such that
  for all $u\in U$ and $h\in [0, h_0]$,
 \begin{equation}\label{elsc}
       \left(      \int_{h}^{1}  H( [u]_{\alpha},  [u]_{\alpha-h}   )^p   \;d\alpha       \right)^{1/p}
     \leq 1.
           \end{equation}

Let $h\in [0,1]$.
If $h \in [0, h_0]$, then put $C_h=1$ and the desired result follows from \eqref{elsc}.

  If $h\in (h_0, 1]$, then there is an $N=   N(h) \in \mathbb{N} $ such that
  $ h / N \leq h_0$. Put $C_h=N(h)$.
  Thus for all $u\in U$,
  \begin{align*}
       & \left(      \int_{h}^{1}  H( [u]_{\alpha},  [u]_{\alpha-h}   )^p   \;d\alpha       \right)^{1/p}
       \\
       & \leq    \sum_{k=0}^{N-1} \left(      \int_{h}^{1}  H( [u]_{\alpha- \frac{k}{N} h},  [u]_{\alpha - \frac{k+1}{N} h}   )^p   \;d\alpha       \right)^{1/p}
       \\
         & \leq    \sum_{k=0}^{N-1} \left(      \int_{h- \frac{k}{N} h}^{1- \frac{k}{N} h}  H( [u]_{\alpha},  [u]_{\alpha - \frac{1}{N} h}   )^p   \;d\alpha       \right)^{1/p}
           \\
         & \leq    \sum_{k=0}^{N-1} \left(      \int_{\frac{1}{N} h}^{1}  H( [u]_{\alpha},  [u]_{\alpha - \frac{1}{N} h}   )^p   \;d\alpha       \right)^{1/p}
       \\
       & \leq  N=C_h.
  \end{align*}

Here we mention the fact that
for each $k=0,\ldots, N-1$,
 $H( [u]_{\alpha- \frac{k}{N} h},  [u]_{\alpha - \frac{k+1}{N} h})   $ is
left continuous on $(h,1]$ and thus is measurable on $[h,1]$;
so the expression
$\sum_{k=0}^{N-1} \left(      \int_{h}^{1}  H( [u]_{\alpha- \frac{k}{N} h},  [u]_{\alpha - \frac{k+1}{N} h}   )^p   \;d\alpha       \right)^{1/p}$
 is well-defined.
 The proof of this fact is similar to the proof of
Proposition \ref{lcpmcu}(\rmn1)(\rmn2), which
 use Proposition \ref{gnc}(\romannumeral1),
 the triangle inequality of the Hausdorff metric (see Remark \ref{tra})
 and Corollary \ref{rsfc}.
 Of course, there is another way of proving this fact, which is essentially the same as the one described above, see also Remark \ref{mytn}.

\end{proof}

We can see that
(\romannumeral1)$\Rightarrow$(\romannumeral3) in the proof of
Theorem \ref{pbe} can also be proved by using
Proposition \ref{emu}.

\section{Characterizations of compactness in $ (F^1_{USCG} (X)^p, d_p)$} \label{cng}

In this section, we give the characterizations of total boundedness, relative compactness and compactness in
$(F^1_{USCG} (X)^p, d_p)$.

Some fundamental conclusions and concepts in classic analysis and topology are listed below, which are useful in this paper.

\begin{itemize}
 \item A subset $Y$ of a topological space $Z$ is said to be \emph{compact} if for every set $I$
and every family of open sets, $O_i$, $i\in I$, such that $Y\subset \bigcup_{i\in I} O_i$ there exists
a finite family $O_{i_1}$, $O_{i_2}$ \ldots, $O_{i_n}$ such that
$Y\subseteq O_{i_1}\cup O_{i_2}\cup\ldots \cup O_{i_n}$.
In
the case of a metric topology, the criterion for compactness becomes that any sequence in $Y$ has a convergent subsequence in $Y$.

 \item
A \emph{relatively compact} subset $Y$ of a topological space $Z$ is a subset with compact closure. In the case of a metric topology, the criterion for relative compactness becomes that any sequence in $Y$ has a subsequence convergent in $Z$.

 \item Let $(X, d)$ be a metric space. A set $U$ in $X$ is \emph{totally bounded} if and only if, for each $\varepsilon>0$, it contains a finite $\varepsilon$ approximation, where an $\varepsilon$ approximation to $U$ is a subset $S$ of $U$ such that $U \subseteq \bigcup_{x\in S}B(x,\varepsilon)$.
An $\varepsilon$ approximation to $U$ is also called an $\varepsilon$-net of $U$.

\item Let $(X, d)$ be a metric space. A set $U$ is compact in $(X,d)$ implies that $U$ is relatively compact
in $(X,d)$, which in turn
implies that $U$ is totally bounded in $(X, d)$.

\end{itemize}

We list
the following conclusions in \cite{huang19c} which are on the property of $H_{\rm end}$ metric,
the characterizations of total boundedness, relative compactness and compactness for
$(F^1_{USCG} (X), H_{\rm end})$, and the completion of $(F^1_{USCG} (X), H_{\rm end})$, respectively. These conclusions will be useful in this paper.

\begin{tm} \cite{huang19c} \label{aec}
Let $u$ be a fuzzy set in $F^1_{USCG} (X)$ and let $u_n$, $n=1,2,\ldots$, be fuzzy sets in $F^1_{USC} (X)$.
Then $H_{\rm end}(u_n, u) \to 0$ if and only if
 $H([u_n]_\al, [u]_\al) \rightarrow 0$ holds a.e. on $\al\in (0,1)$, which is denoted by $H([u_n]_\alpha, [u]_\alpha) \str{   \rm {a.e.}   }{\longrightarrow}   0 \ (  [0,1]    )$.

\end{tm}

\begin{tm} \cite{huang19c} \label{tbfegn}
  Let $U$ be a subset of $F^1_{USCG} (X)$. Then $U$ is totally bounded in $(F^1_{USCG} (X), H_{\rm end})$
if and only if
$U(\al)$
is totally bounded in $(X,d)$ for each $\al \in (0,1]$.
\end{tm}

\begin{tm} \cite{huang19c} \label{rcfegn}
  Let $U$ be a subset of $F^1_{USCG} (X)$. Then $U$ is relatively compact in $(F^1_{USCG} (X), H_{\rm end})$
if and only if
$U(\al)$
is relatively compact in $(X, d)$ for each $\al \in (0,1]$.
\end{tm}

\begin{tm} \cite{huang19c} \label{cfeg}
  Let $U$ be a subset of $F^1_{USCG} (X)$. Then the following are equivalent:
\begin{enumerate}
\renewcommand{\labelenumi}{(\roman{enumi})}

\item
 $U$ is compact in $(F^1_{USCG} (X), H_{\rm end})$;

\item  $U(\al)$
is relatively compact in $(X, d)$ for each $\al \in (0,1]$ and $U$ is closed in $(F^1_{USCG} (X), H_{\rm end})$;

\item  $U(\al)$
is compact in $(X, d)$ for each $\al \in (0,1]$ and $U$ is closed in $(F^1_{USCG} (X), H_{\rm end})$.
\end{enumerate}

\end{tm}

\begin{tm} \cite{huang19c} \label{fgecom}
 $(F^1_{USCG} (\widetilde{X}), H_{\rm end})$ is a completion of $(F^1_{USCG} (X), H_{\rm end})$.

\end{tm}

The relationship between the $d_p$ metric and the $H_{\rm end}$ metric
given in Theorem \ref{dpeu} (see the illustrations in Remark \ref{pse})
will be used repeatedly in the sequel.

\begin{lm}
 \label{uplc}
If  $u\in F^1_{USCG}   (  X )^p$,
 then
  $u$ is $p$-mean left continuous.

\end{lm}

\begin{proof}
  The desired result can be proved in a similar fashion to Lemma 4.3 in \cite{huang97} by replacing $\{\mathbf{0}\}$ with $\{x_0\}$, where $\mathbf{0}$ denotes the point $(\str{m}{\overbrace{0,\ldots,0}})$ in $\mathbb{R}^m$ and $x_0$ is a point in $X$.

\end{proof}

\begin{tm}
\label{pcn}
Let $U$ be a subset of $F^1_{USCG}   ( X )^p$. Then
 $U$ is a relatively compact set in $(F^1_{USCG}   ( X )^p,   d_p)$
 if and only if
 \\
 (\romannumeral1) \ $U$ is a relatively compact set in $(F^1_{USCG}  (X),  H_{\rm end})$, and
 \\
(\romannumeral2) \ $U$ is $p$-mean equi-left-continuous.
\end{tm}

\begin{proof}
We prove the ``only if'' part of the theorem. \ Assume that $U$ is a relatively compact set in $(F^1_{USCG}   ( X )^p,   d_p)$.
By Theorem \ref{dpeu},
(\romannumeral1) is true.

Now we prove (\rmn2). Given $\varepsilon>0$. Since $U$ is a relatively compact set in $(F^1_{USCG} (X)^p,   d_p)$, there exists a finite $\varepsilon/3$ net $S$ of $U$. From Lemma \ref{uplc}, each $v\in S$ is
$p$-mean left-continuous. This means that
$S$ is $p$-mean equi-left-continuous, since $S$ is finite. Hence we
have
\\
(a) there exists $\delta\in (0,1]$ such that
$
  \left(\int_h^{1}    H([v]_\alpha, [v]_{\alpha-h} )^p \;d\alpha  \right)^{1/p}   <   \varepsilon/3
$
for all $h\in [0,\delta)$ and all $v\in S$.
\\
Given $u\in U$, there is a $v_0\in S$ such that $d_p(u, v_0)< \varepsilon/3$, and thus, by (a), we have that for all $h\in [0,\delta)$,
$
        \left(\int_h^1   H([u]_\alpha,  [u]_{\alpha-h} )^p \;d\alpha   \right)^{1/p}
  \leq      \left(\int_h^1   H([u]_\alpha,  [v_0]_{\alpha} )^p \;d\alpha   \right)^{1/p}
   +
    \left(\int_h^1   H([v_0]_\alpha,  [v_0]_{\alpha-h} )^p \;d\alpha   \right)^{1/p}
   +
    \left(\int_h^1   H([v_0]_{\alpha-h},  [u]_{\alpha-h} )^p \;d\alpha   \right)^{1/p}
      <  \varepsilon/3
  +
   \varepsilon/3
      +
      \varepsilon/3
 =\varepsilon.
$
Since $\varepsilon>0$ and $u\in U$ are arbitrary, it follows that $U$ is $p$-mean equi-left-continuous.
The necessity of (\romannumeral2) is essentially proved in
  the proof of
  Theorem 4.1 in \cite{huang97}. This proof of the necessity of (\romannumeral2) is similar to that of the necessity of the condition (\rmn2)
in
  Theorem 4.1 of \cite{huang97} (see Page 26 of \cite{huang97}).
 Note that $U$ is a totally bounded set in $(F^1_{USCG}   ( X )^p,   d_p)$.
  Thus the necessity of (\romannumeral2) follows from the necessity of (\rmn2) in Theorem
\ref{tcn}.

We prove the ``if'' part of the theorem.   The proof is similar to ``sufficient'' part of the proof of Theorem 4.1 in \cite{huang97}.
  A sketch of the proof is given as follows.

Assume that (\rmn1) and (\rmn2) are true.
To show that
 $U$ is a relatively compact set in $(F^1_{USCG}   ( X )^p,   d_p)$, let $\{u_n\}$ be a sequence in $U$. We only need
to find a subsequence $\{v_n\}$ of $\{u_n\}$ and a $v \in F^1_{USCG}   (  X )^p$
such that $\{v_n\}$ converges to $v$ according to the $d_p$ metric.
We split the proof into three steps.

\textbf{Step 1.} \ Find a subsequence $\{v_n\}$ of $\{u_n\}$
and a $v \in F^1_{USCG}   (X)$
such that $H_{\rm end} (v_n, v) \to 0$; that is, by Theorem \ref{aec},
\begin{equation}\label{vnec}
  H([v_n]_\alpha, [v]_\alpha) \str{   \rm {a.e.}   }{\longrightarrow}   0 \ (  [0,1]    ).
\end{equation}

From (\romannumeral1),
 this step can be done immediately.

\textbf{Step 2. }
Prove
that
\begin{equation}\label{vnc}
   \left(      \int_0^1  H( [v_n]_\alpha,   [v]_\alpha    )^p   \;d\alpha       \right)^{1/p}      \to 0 \mbox{ as } n\to +\infty.
\end{equation}

  Proceeding according to the
\textbf{Step 2} in the proof of Theorem 4.1 in \cite{huang97},
  we can obtain the desired result.

Here we mention one thing.
 In this proof of step 2, we need to prove the conclusion: for each $h \in (0,1]$,
 \begin{equation} \label{hdpc}
   \left(      \int_h^1  H( [v_n]_\alpha,   [v]_\alpha    )^p   \;d\alpha       \right)^{1/p}    \to 0 \mbox{ as } n\to +\infty.
\end{equation}
In the proof of the corresponding conclusion in \cite{huang97}, which is at Page 28 Line 6 from the bottom in \cite{huang97},
 there is a small mistake (or misprints).
 In the following, we give a slightly adjusted proof for the above conclusion
(Obviously, the proof of the corresponding conclusion in \cite{huang97}
can be adjusted similarly).

 Note that
 $[v_n]_h$ and $[v]_h$    are contained in $\overline{U(h/2)}$,
 which is compact in $X$ according to Theorem \ref{rcfegn}.
 Then
 there is an $M(h) \geq 0$ such that
 $$ \max \{  d(x,y):  x,y \in \overline{U(h/2)}   \} \leq  M(h). $$
Hence
   $$H( [v_n]_\alpha,   [v]_\alpha   )    \leq    M(h)$$
  for $\al\in [h, 1]$
  and $n=1,2,\ldots$.
 Combined with \eqref{vnec} and by using the Lebesgue's dominated convergence theorem,
we thus obtain \eqref{hdpc}.

\textbf{Step 3.} Show that $v\in F^1_{USCG}   (  X )^p$.

Since $v \in F^1_{USCG}   (X)$, it suffices to show
that
$ \left(      \int_0^1  H( [v]_\alpha,  \{x_0\}   )^p   \;d\alpha       \right)^{1/p} < +\infty$ for some $x_0 \in X$,
which
can be proved in a similar fashion
to
the conclusion ``$ \left(      \int_0^1  H( [v]_\alpha,  \{\mathbf{0}\}   )^p   \;d\alpha       \right)^{1/p} < +\infty$'' in
 the
\textbf{Step 3} in the proof of Theorem 4.1 in \cite{huang97}
by replacing $\{\mathbf{0}\}$ with $\{x_0\}$, where $x_0 \in X$.

\end{proof}

\begin{tm}
\label{pcnr} \ Let $U$ be a subset of $F^1_{USCG}   ( X )^p$.
Then
 $U$ is a relatively compact set in $(F^1_{USCG}   ( X )^p,   d_p)$
 if and only if
 \\
 (\romannumeral1) \ $U(\al)$
is relatively compact in $(X, d)$ for each $\al \in (0,1]$, and
 \\
(\romannumeral2) \ $U$ is $p$-mean equi-left-continuous.
\end{tm}

\begin{proof} The desired result follows immediately from Theorems \ref{rcfegn} and \ref{pcn}.

\end{proof}

\begin{tm} \label{tcn}
 Let $U$ be a subset of $F^1_{USCG}   ( X )^p$.
 Then
  $U$ is a totally bounded set in $(F^1_{USCG}   (  X)^p,   d_p)$
 if and only if
 \\
 (\romannumeral1) \  $U$ is a totally bounded set in $(F^1_{USCG}   (  X),   H_{\rm end})$, and
 \\
(\romannumeral2) \ $U$ is $p$-mean equi-left-continuous.
\end{tm}

\begin{proof}
We prove the ``only if'' part of the theorem.
Suppose that $U$ is totally bounded in $(F^1_{USCG}   (  X)^p,   d_p)$.
Then
by \eqref{dpe},
 $U$ is a totally bounded set in $(F^1_{USCG}   (  X),   H_{\rm end})$;
that is, (\romannumeral1) is true.

We obtain the proof of
the necessity of (\rmn2) by
replacing
``since $U$ is a relatively compact set in $(F^1_{USCG} (X)^p,   d_p)$'' by
``since $U$ is totally bounded in $(F^1_{USCG} (X)^p,   d_p)$''
in the proof of the necessity of (\romannumeral2) in Theorem \ref{pcn}.

The necessity of (\romannumeral2) is essentially proved in
the proof of
  Theorem 4.1 in \cite{huang97} (see Page 26 in \cite{huang97}).

Now we prove the ``if'' part of the theorem. Suppose that $U$ satisfies (\romannumeral1)
and
 (\romannumeral2).
Clearly
 (\romannumeral1) is equivalent to $U$ is a totally bounded set in $(F^1_{USCG}   (  \widetilde{X}),   H_{\rm end})$, which,
by Theorem \ref{fgecom}, is equivalent to $U$ is a relatively compact set in $(F^1_{USCG}   (  \widetilde{X}),   H_{\rm end})$.
Then, by Theorem \ref{pcn},
  $U$ is a relatively compact set in $(F^1_{USCG}   (  \widetilde{X})^p,   d_p)$.
    Thus $U$ is a totally bounded set in $(F^1_{USCG}   (  \widetilde{X})^p,   d_p)$. This means that
$U$ is a totally bounded set in $(F^1_{USCG}   (  X)^p,   d_p)$.

\end{proof}

\begin{tm} \label{tcns}
 Let $U$ be a subset of $F^1_{USCG}   ( X )^p$.
 Then
  $U$ is a totally bounded set in $(F^1_{USCG}   (  X)^p,   d_p)$
 if and only if
 \\
 (\romannumeral1) \  $U(\al)$
is totally bounded in $(X,d)$ for each $\al \in (0,1]$, and
 \\
(\romannumeral2) \ $U$ is $p$-mean equi-left-continuous.
\end{tm}

\begin{proof}
The desired result follows immediately from
Theorems \ref{tbfegn} and \ref{tcn}.

\end{proof}

\begin{tm} \label{gscn}
\ Let
 $U$ be a subset of $F^1_{USCG}   ( X )^p$.
Then $U$ is    compact  in $ (F^1_{USCG}   (  X )^p,   d_p)$
 if and only if
 \\
 (\romannumeral1) \  $U(\al)$
is relatively compact in $(X, d)$ for each $\al \in (0,1]$,
 \\
(\romannumeral2) \ $U$ is $p$-mean equi-left-continuous, and
\\
 (\romannumeral3) \  $U$ is a closed set in $ (F^1_{USCG}   (  X)^p,   d_p)$.
\end{tm}

\begin{proof} \ The desired result follows immediately from Theorem \ref{pcnr}.

 \end{proof}

\begin{tm} \label{gscnre}
\ Let
 $U$ be a subset of $F^1_{USCG}   ( X )^p$.
Then $U$ is    compact  in $ (F^1_{USCG}   (  X )^p,   d_p)$
 if and only if
 \\
 (\romannumeral1) \  $U(\al)$
is compact in $(X, d)$ for each $\al \in (0,1]$,
 \\
(\romannumeral2) \ $U$ is $p$-mean equi-left-continuous, and
\\
 (\romannumeral3) \  $U$ is a closed set in $ (F^1_{USCG}   (  X)^p,   d_p)$.
\end{tm}

\begin{proof} \
By Theorem \ref{gscn},
to show the desired result,
we only need to show that if $U$ is    compact  in $ (F^1_{USCG}   (  X )^p,   d_p)$,
then (\romannumeral1) is true.

Assume that
 $U$ is    compact  in $ (F^1_{USCG}   (  X )^p,   d_p)$, then by \eqref{dpe}, $U$ is    compact  in $ (F^1_{USCG}   (  X ),   H_{\rm end})$, hence
from Theorem \ref{cfeg},  (\romannumeral1) is true.

 \end{proof}

\begin{tm} \label{gscpnre}
\ Let
 $U$ be a subset of $F^1_{USCG}   ( X )^p$.
Then $U$ is a compact set in $ (F^1_{USCG}   (  X )^p,   d_p)$
 if and only if
 \\
 (\romannumeral1) \ $U$ is a compact set in $(F^1_{USCG}   (  X),   H_{\rm end})$, and
 \\
(\romannumeral2) \ $U$ is $p$-mean equi-left-continuous.
\end{tm}

\begin{proof} We prove the ``only if'' part of the theorem.
Assume that $U$ is compact in $ (F^1_{USCG}   (  X )^p,   d_p)$.
Hence by Theorem \ref{dpeu}, $U$ is compact in $(F^1_{USCG}   (  X),   H_{\rm end})$, i.e. (\romannumeral1) is true.
By Theorem \ref{gscnre}, (\romannumeral2) is true.

We prove the ``if'' part of the theorem. Assume that (\romannumeral1) and (\romannumeral2) are true.
Then
by Theorem \ref{pcn}, $U$ is relatively compact in $ (F^1_{USCG}   (  X )^p,   d_p)$.
To show that $U$ is compact in $ (F^1_{USCG}   (  X )^p,   d_p)$,
we only need to show that $U$ is closed in $ (F^1_{USCG}   (  X)^p,   d_p)$.

To do this,
let $\{u_n\}$ be a sequence in $U$
and $d_p(u_n, u) \to 0$.
Then by Theorem \ref{dpeu}, $H_{\rm end}(u_n, u) \to 0$.
Since
from (\romannumeral1),
$U$ is closed in $(F^1_{USCG} (X), H_{\rm end})$,
we have that $u\in U$.
So $U$ is closed in $ (F^1_{USCG}   (  X)^p,   d_p)$.

\end{proof}

From Theorems \ref{pcn}, \ref{tcn} and \ref{gscpnre},
   we obtain the following conclusion:
\begin{itemize}
  \item
Let $U$ be a subset in $F^1_{USCG}   (  X )^p$. Then
$U$ is total bounded (respectively, relatively
compact, compact) in $ (F^1_{USCG}   (  X )^p,   d_p)$
if and only if $U$ is total bounded (respectively, relatively
compact, compact) in $(F^1_{USCG} (X), H_{\rm end})$
and $U$ is $p$-mean equi-left-continuous.
\end{itemize}

\section{Characterizations of compactness in $ (F^1_{USCG} (\mathbb{R}^m)^p, d_p)$ }\label{cnr}

In this section, we discuss the characterizations of total boundedness, relative compactness and compactness in $ (F^1_{USCG} (\mathbb{R}^m)^p, d_p)$.
We point out that
the conclusions on the characterizations of total boundedness, relative compactness and compactness in $ (F^1_{USCG} (\mathbb{R}^m)^p, d_p)$ given in our previous work
\cite{huang9} can be seen as
corollaries of the corresponding results for $ (F^1_{USCG} (X)^p, d_p)$ given in Section \ref{cng} of this paper.
Furthermore, by using results in Sections \ref{per} and \ref{cng},
we give new
characterizations of total boundedness, relative compactness and compactness in $ (F^1_{USCG} (\mathbb{R}^m)^p, d_p)$.

Note
that for
a subset $V$ of $\mathbb{R}^m$,
 the conditions
(\romannumeral1)
$V$
is relatively compact in $\mathbb{R}^m$,
(\romannumeral2)
$V$
is totally bounded in $\mathbb{R}^m$, and
(\romannumeral3)
$V$
is bounded in $\mathbb{R}^m$, are equivalent to each other.
Thus
Theorems \ref{pcnr}, \ref{tcns}, \ref{gscn} and \ref{gscnre}
imply the following four conclusions on the characterizations of compactness in $ (F^1_{USCG} (\mathbb{R}^m)^p, d_p)$, respectively.

\begin{tl} \label{pcnrem} \ Let $U$ be a subset of $F^1_{USCG}   ( \mathbb{R}^m )^p$.
Then
 $U$ is a relatively compact set in $(F^1_{USCG}   ( \mathbb{R}^m )^p,   d_p)$
 if and only if
 \\
 (\romannumeral1) \ $U(\al)$
is bounded in $\mathbb{R}^m$ for each $\al \in (0,1]$, and
 \\
(\romannumeral2) \ $U$ is $p$-mean equi-left-continuous.
\end{tl}

\begin{tl} \label{tcnsrm}
 Let $U$ be a subset of $F^1_{USCG}   ( \mathbb{R}^m )^p$.
 Then
  $U$ is a totally bounded set in $(F^1_{USCG}   (  \mathbb{R}^m)^p,   d_p)$
 if and only if
 \\
 (\romannumeral1) \  $U(\al)$
is bounded in $\mathbb{R}^m$ for each $\al \in (0,1]$, and
 \\
(\romannumeral2) \ $U$ is $p$-mean equi-left-continuous.
\end{tl}

\begin{tl} \label{gscnrm}
\ Let
 $U$ be a subset of $F^1_{USCG}   ( \mathbb{R}^m )^p$.
Then $U$ is    compact  in $ (F^1_{USCG}   (  \mathbb{R}^m )^p,   d_p)$
 if and only if
 \\
 (\romannumeral1) \  $U(\al)$
is bounded in $\mathbb{R}^m$ for each $\al \in (0,1]$,
 \\
(\romannumeral2) \ $U$ is $p$-mean equi-left-continuous, and
\\
 (\romannumeral3) \  $U$ is a closed set in $ (F^1_{USCG}   (  \mathbb{R}^m)^p,   d_p)$.
\end{tl}

\begin{tl} \label{gscnrme}
\ Let
 $U$ be a subset of $F^1_{USCG}   ( \mathbb{R}^m )^p$.
Then $U$ is    compact  in $ (F^1_{USCG}   (  \mathbb{R}^m )^p,   d_p)$
 if and only if
 \\
 (\romannumeral1) \  $U(\al)$
is compact in $\mathbb{R}^m$ for each $\al \in (0,1]$,
 \\
(\romannumeral2) \ $U$ is $p$-mean equi-left-continuous, and
\\
 (\romannumeral3) \  $U$ is a closed set in $ (F^1_{USCG}   (  \mathbb{R}^m)^p,   d_p)$.
\end{tl}

In \cite{huang97},
we have obtained the following three conclusions on the characterizations of compactness in $ (F^1_{USCG} (\mathbb{R}^m)^p, d_p)$.

\begin{tm} \label{rcs}  {\rm (Theorem 4.1 in \cite{huang97})} \
Let $U$ be a subset of $F^1_{USCG}   (\mathbb{R}^m )^p$.
Then
 $U$ is a relatively compact set in $(F^1_{USCG}   (  \mathbb{R}^m )^p,   d_p)$
 if and only if
 \\
 (\romannumeral1) \  $U$ is uniformly $p$-mean bounded, and
 \\
(\romannumeral2) \ $U$ is $p$-mean equi-left-continuous.
\end{tm}

\begin{tm} \label{tcnr}  {\rm (Theorem 4.2 in \cite{huang97})} \
Let $U$ be a subset of $F^1_{USCG}   (\mathbb{R}^m )^p$.
Then
  $U$ is a totally bounded set in $(F^1_{USCG}   (\mathbb{R}^m )^p,   d_p)$
 if and only if
 \\
 (\romannumeral1) \  $U$ is uniformly $p$-mean bounded, and
 \\
(\romannumeral2) \ $U$ is $p$-mean equi-left-continuous.
\end{tm}

\begin{tm} \label{gscnrc}  {\rm (Theorem 4.3 in \cite{huang97})} \
Let
  $U$ be a subset of $F^1_{USCG}   (\mathbb{R}^m )^p$.
Then $U$ is    compact  in $ (F^1_{USCG}   (  \mathbb{R}^m )^p,   d_p)$
 if and only if
 \\
 (\romannumeral1) \  $U$ is uniformly $p$-mean bounded,
 \\
(\romannumeral2) \ $U$ is $p$-mean equi-left-continuous, and
\\
 (\romannumeral3) \  $U$ is a closed set in $ (F^1_{USCG}   (  \mathbb{R}^m )^p,   d_p)$.
\end{tm}

\begin{re}
  {\rm

From Theorem \ref{pbe},
we can see:
\\
Corollary \ref{pcnrem} implies Theorem 4.1 in \cite{huang97} (which is Theorem \ref{rcs} in this paper), and the converse is true;
\\
Corollary \ref{tcnsrm} implies Theorem 4.2 in \cite{huang97} (which is Theorem \ref{tcnr} in this paper), and the converse is true;
\\
Corollary \ref{gscnrm} implies Theorem 4.3 in \cite{huang97} (which is Theorem \ref{gscnrc} in this paper), and the converse is true.
\\
So
the characterizations of compactness for $(F^1_{USCG}   (\mathbb{R}^m )^p,   d_p)$ in \cite{huang97} (Theorems 4.1, 4.2 and 4.3 in \cite{huang97}) can be seen as corollaries of the characterizations of compactness for $(F^1_{USCG}   (X )^p,   d_p)$ in this paper (Theorems \ref{pcnr}, \ref{tcns} and \ref{gscn}).
The results of the characterizations of compactness for $(F^1_{USCG}   (X )^p,   d_p)$ in this paper
generalize the corresponding results
 for $(F^1_{USCG}   (\mathbb{R}^m )^p,   d_p)$ in \cite{huang97}.

}
\end{re}

In the following, we give new characterizations of compactness for $(F^1_{USCG}   (\mathbb{R}^m)^p,   d_p)$.

Using Theorem \ref{pbe},
we can obtain the following characterizations of compactness for $(F^1_{USCG}   (\mathbb{R}^m)^p,   d_p)$ from
Corollaries \ref{pcnrem}, \ref{tcnsrm} and \ref{gscnrm}.

\begin{tm} \label{rcsm}
Let $U$ be a subset of $F^1_{USCG}   (\mathbb{R}^m )^p$.
Then
 $U$ is a relatively compact set in $(F^1_{USCG}   (  \mathbb{R}^m )^p,   d_p)$
 if and only if
 \\
 (\romannumeral1) \   There exists an $h\in (0,1)$ such that $U(h)$ is bounded in $\mathbb{R}^m $, and
 \\
(\romannumeral2) \ $U$ is $p$-mean equi-left-continuous.
\end{tm}

\begin{tm} \label{tcnrsm}
Let $U$ be a subset of $F^1_{USCG}   (\mathbb{R}^m )^p$.
Then
  $U$ is a totally bounded set in $(F^1_{USCG}   (\mathbb{R}^m )^p,   d_p)$
 if and only if
 \\
 (\romannumeral1) \  There exists an $h\in (0,1)$ such that $U(h)$ is bounded in $\mathbb{R}^m $, and
 \\
(\romannumeral2) \ $U$ is $p$-mean equi-left-continuous.
\end{tm}

\begin{tm} \label{gscnrcsum}
Let
  $U$ be a subset of $F^1_{USCG}   (\mathbb{R}^m )^p$.
Then $U$ is    compact  in $ (F^1_{USCG}   (  \mathbb{R}^m )^p,   d_p)$
 if and only if
 \\
 (\romannumeral1) \  There exists an $h\in (0,1)$ such that $U(h)$ is bounded in $\mathbb{R}^m $,
 \\
(\romannumeral2) \ $U$ is $p$-mean equi-left-continuous, and
\\
 (\romannumeral3) \  $U$ is a closed set in $ (F^1_{USCG}   (  \mathbb{R}^m )^p,   d_p)$.
\end{tm}

\section{Completion of $(F^1_{USCG}   (X)^p, d_p)$}

In this section, we show that
 $ (F^1_{USCG}   (\widetilde{X}   )^p, d_p)$ is a completion of  $ (F^1_{USCB}   (  X ), d_p)$,
and thus a completion of $ (F^1_{USCG}   (  X )^p, d_p)$.

We use $(\widetilde{X}, \widetilde{d})$ to denote the completion of $(X, d)$.
We see $(X, d)$ as a subspace of $(\widetilde{X}, \widetilde{d})$.
Let
$S \subseteq \widetilde{X}$.
The symbol $\widetilde{\overline{S}}$ is used to denote
the closure of $S$
in
 $(\widetilde{X}, \widetilde{d})$.

As defined previously, we have
$K(\widetilde{X})$,
 $C(\widetilde{X})$, $F^1_{USC}(\widetilde{X})$,
$F^1_{USCG} (\widetilde{X})$, etc. according to $(\widetilde{X}, \widetilde{d})$.
For example,
  \begin{gather*}
   F^1_{USC}(\widetilde{X}) :=\{ u\in F(\widetilde{X}) : [u]_\al \in  C(\widetilde{X})  \  \mbox{for all} \   \al \in [0,1]   \}, \\
   F^1_{USCG} (\widetilde{X}) := \{ u\in  F(\widetilde{X}): [u]_\al \in K(\widetilde{X}) \ \mbox{for all} \   \al\in (0,1] \}.
 \end{gather*}

If there is no confusion,
 we also
use $H$ to denote the Hausdorff metric
on
 $C(\widetilde{X})$ induced by $\widetilde{d}$.
  We also
use $H$ to denote the Hausdorff metric
 on $C(\widetilde{X}\times [0,1])$ induced by $\overline{\widetilde{d}}$.
 We
 also use $H_{\rm end}$ to denote the endograph metric on $F^1_{USC}(\widetilde{X})$
given by using $H$ on
$C(\widetilde{X} \times [0,1])$.
We
 also use $d_p$ to denote the $d_p$ metric on $F^1_{USCG}(\widetilde{X})$.

Define $f: F^1_{USCG} (X) \to F^1_{USCG} (\widetilde{X})$ as follows:
for $u \in F^1_{USCG} (X)$,
\[f(u)(t) = \left\{
        \begin{array}{ll}
          u(t), & t\in X, \\
0, & t\in \widetilde{X} \setminus X.
        \end{array}
      \right.
\]
Then $[f(u)]_\al = [u]_\al$ for all $\al\in (0,1]$, and so $f(u) \in F^1_{USCG} (\widetilde{X})$.
We can see that for
$\rho = d_\infty, d_p, H_{\rm end}, H_{\rm send}$, $\rho(u,v) = \rho(f(u), f(v))$.
So for
$\rho = d_\infty, d_p, H_{\rm send}$, $ (F^1_{USCG} (X), \rho)$
can be embedded as an extended metric subspace of $ (F^1_{USCG} (\widetilde{X}), \rho)$.
$ (F^1_{USCG} (X), H_{\rm end})$
can be embedded as a metric subspace of $ (F^1_{USCG} (\widetilde{X}), H_{\rm end})$.

In this paper,
we see
$ (F^1_{USCG} (X), H_{\rm end})$
as a metric subspace of
$ (F^1_{USCG} (\widetilde{X}), H_{\rm end})$.
We see
$(F^1_{USCG} (X)^p, d_p) $ as a metric subspace of $(F^1_{USCG} (\widetilde{X})^p, d_p)  $.
We see
$ (F^1_{USCG} (X), d_p)$
as an extended metric subspace of
$ (F^1_{USCG} (\widetilde{X}), d_p)$.

\begin{tm} \label{scp}
 $(X,d)$ is complete
 if and only if
 $(F^1_{USCG}   (  X )^p, d_p)$ is complete.
\end{tm}

\begin{proof}
We prove the ``only if'' part of the theorem.
  Suppose that $(X,d)$ is complete.
  To show that $(F^1_{USCG}   (  X )^p, d_p)$
 is complete, we only
 need to show that each Cauchy sequence in $(F^1_{USCG}   (  X )^p, d_p)$ is relatively compact.

  Let
$\{u_n:  n\in \mathbb{N}\}$ be a Cauchy sequence in      $(F^1_{USCG}   (  X )^p, d_p)$.
 Then $\{u_n:  n\in \mathbb{N}\}$ is totally bounded in $(F^1_{USCG}   (  X )^p, d_p)$.
By Theorems \ref{pcn} and \ref{tcn},
to show
that
$\{u_n:  n\in \mathbb{N}\}$ is relatively compact in $(F^1_{USCG}   (  X )^p, d_p)$, we only need to show that
 $\{u_n:  n\in \mathbb{N}\}$ is relatively compact
in
 $(F^1_{USCG}   (  X )^p, H_{\rm end})$.

      By Theorem \ref{tcn},
  $\{u_n:  n\in \mathbb{N}\}$ is totally bounded in      $(F^1_{USCG}   (  X ), H_{\rm end})$.
 Since,
by Theorem 6.1 in \cite{huang19c},
  $(F^1_{USCG}   (  X ), H_{\rm end})$ is complete, and thus $\{u_n:  n\in \mathbb{N}\}$ is relatively compact in $(F^1_{USCG}   (  X ), H_{\rm end})$.

We prove the ``if'' part of the theorem. \ Suppose that $(F^1_{USCG}   (  X )^p, d_p)$ is complete.
Let $\widehat{X} = \{ \widehat{x} : x\in X  \}$.
Then $\widehat{X}\subseteq  F^1_{USCB} (X)$.
Define $f: X\to \widehat{X}$ by $f(x)= \widehat{x}$.
Note that $d(x,y) = d_p (\widehat{x},\,  \widehat{y})$. Hence $f$ is a isometry
from $X$ to $\widehat{X}$.
If
 $\{\widehat{x_n}\}$ converges to $u\in F^1_{USCG} (X)^p$,
then there exists an $x\in X$
such
that $[u]_\al = \{x\}$ for all $\al\in [0,1]$; that is $u = \widehat{x}$.
Thus
$\widehat{X}$ is a closed subspace of
 $(F^1_{USCG}   (  X )^p, d_p)$.
So $(X,d)$ is isometric to
a closed subspace of
 $(F^1_{USCG}   (  X )^p, d_p)$, and then $(X,d)$ is complete.

\end{proof}

\begin{tl} \label{scpr}

 $(F^1_{USCG}   (  \mathbb{R}^m )^p, d_p)$ is complete.
\end{tl}

\begin{proof}
 Since $\mathbb{R}^m$ is complete, the desired result follows immediately from Theorem \ref{scp}.

\end{proof}

\begin{re}{\rm
  Corollary \ref{scpr} is Theorem 5.1 in \cite{huang97}.
So
Theorem \ref{scp} in this paper generalizes Theorem 5.1 in \cite{huang97}.
}
\end{re}

For $u \in F^1_{USCG} (X)$ and $\varepsilon>0$,
define $u^\varepsilon \in  F^1_{USCB} (X)$ by
  \[
  [u^\varepsilon]_{\al} =\left\{
         \begin{array}{ll}
           [u]_\al , & \al \in (\varepsilon, 1], \\
          \mbox{} [u]_\varepsilon, & \al \in  [0,\varepsilon ].
         \end{array}
       \right.
  \]

\begin{tm} \label{sln}
  $F^1_{USCB}   (  X)$ is a dense set in $(F^1_{USCG}   (X)^p, d_p)$.
\end{tm}

\begin{proof}

The desired result can be proved in a similar fashion
to Theorem 5.2 in \cite{huang97}.
In fact,
it is
 shown that
for each $v\in F^1_{USCG}   (X)^p$,
$d_p(v^{(1/n)},  v) \to 0$.

\end{proof}

\begin{tm} \label{com}
  $ (F^1_{USCG}   (\widetilde{X}   )^p, d_p)$ is a completion of  $ (F^1_{USCB}   (  X ), d_p)$.
\end{tm}

\begin{proof}

In the proof of Theorem 6.3 in \cite{huang19c}, we show the following conclusion
\begin{itemize}
  \item
for each  $v \in F^1_{USCB} (\widetilde{X})$ and each $\varepsilon>0$, there is a $w\in F^1_{USCB} (X)$ such that
$H([w]_\al, [v]_\al) \leq \varepsilon$ for all $\al\in [0,1]$.
\end{itemize}
Thus
we have
$d_p(v, w) \leq \varepsilon$. This means that $F^1_{USCB}   (  X )$ is dense
in
$(F^1_{USCB}   (  \widetilde{X}), d_p)$.

From Theorem \ref{sln}, we know that $F^1_{USCB}   (  \widetilde{X})$ is dense in $(F^1_{USCG}   (\widetilde{X})^p, d_p)$.

Combined the above conclusions, we obtain that
$F^1_{USCB}   (  X )$
is dense in
$ (F^1_{USCG}   (\widetilde{X}   )^p, d_p)$.
By Theorem \ref{scp}, $ (F^1_{USCG}   (\widetilde{X}   )^p, d_p)$ is complete.
So
 $(F^1_{USCG}   (\widetilde{X}   )^p, d_p)$ is a completion of  $ (F^1_{USCB}   (  X ), d_p)$.

\end{proof}

\begin{tl}\label{gcn}
    $ (F^1_{USCG}   (\widetilde{X}   )^p, d_p)$ is a completion of  $ (F^1_{USCG}   (  X )^p, d_p)$.
\end{tl}

\begin{proof} Since $F^1_{USCB}   (  X ) \subseteq  F^1_{USCG}   (  X )^p   \subseteq F^1_{USCG}   (\widetilde{X}   )^p$,
  the desired result follows immediately from Theorem \ref{com}.

  \end{proof}

\section{Conclusions}

In this paper, we discussed properties of the $d_p$ metrics and the spaces of fuzzy sets in a general metric space $(X,d)$ with $d_p$ metrics.

In what cases the $d_p$ metrics are well-defined is a fundamental question.
For each $u,v\in F^1_{USC}(X)$,
$d_p(u,v)= \left(\int_0^1 H ([u]_\al, [v]_\al)^p  \,   d\al   \right)^{1/p}$
 is well-defined if and only if
$H([u]_\al, [v]_\al)$ is a measurable function of $\al$ on $[0, 1]$.
In this paper, we consider this problem
not only with fuzzy sets having compact
$\al$-cuts, but also
with fuzzy sets not necessarily having compact
$\al$-cuts. The motivations are as follows.
\\
(a)
Since the Hausdorff distance
can also be considered between sets not necessarily being compact, it is natural for us to consider the measurability of Hausdorff distance function induced by fuzzy sets not necessarily having compact
$\al$-cuts.

Consider $u_0= [0,+\infty)_{F(\mathbb{R})}$ and $v_0=[1,+\infty)_{F(\mathbb{R})}$. Then $H([u_0]_\al, [v_0]_\al)= 1$
for all $\al\in [0,1]$. So $H([u_0]_\al, [v_0]_\al)$ is a measurable function of $\al$ on $[0,1]$ and $d_p(u_0,v_0)$ is well-defined.
We can see that $u_0,v_0$ are fuzzy sets with non-compact
$\al$-cuts, and their corresponding Hausdorff distance function $H([u_0]_\al, [v_0]_\al)$ is a common function.
It is natural to discuss the well-definedness of $d_p(u_0,v_0)$.
\\
(b) We think that the applications of fuzzy sets will definitely involve cases when fuzzy sets not necessarily having compact
$\al$-cuts. For example, a kind of fuzzy sets $u$ with the Gaussian membership functions are defined
as follows
$u(x;\sigma,c)=\exp
-\frac{(x-c)^2}{2\sigma}$ for each $x\in \mathbb{R}$, where $\sigma > 0$ and $c\in \mathbb{R}$. Clearly $[u]_0 = \mathbb{R}$ and $[u]_0$ is not compact in $\mathbb{R}$.
The kind of fuzzy sets $u$ with the Gaussian membership functions
are used widely in applications.

In Section \ref{meau} of this paper, we obtain the following conclusion.

(\romannumeral1) \ Let $(X, d_X)$ be a metric subspace of $(Y, d_Y)$ and $Y \setminus X$ an at most countable set.
 Let $u, v \in F^1_{USC} (X)$.
If
$u^Y \in  F^1_{USCG} (Y)$, then $d_p(u,v)= \left(\int_0^1 H_{X} ([u]_\al, [v]_\al)^p  \,   d\al   \right)^{1/p}$  is well-defined.

In the special case of $Y=X$, the above conclusion become:

Let $u\in F^1_{USC} (X)$ and let $v\in F^1_{USCG} (X)$. Then $d_p(u,v)$ is well-defined.

(\romannumeral2) \
Let
 $A$ be a nonempty subset of $\mathbb{R}^m$ with $\overline{A}^{\mathbb{R}^m}\setminus A$ being at most countable.
Let    $u,v \in  F^1_{USC} (A)$.
Then $d_p(u,v)= \left(\int_0^1 H_{A} ([u]_\al, [v]_\al)^p  \,   d\al   \right)^{1/p}$ is well-defined.

In the special case of $\overline{A}^{\mathbb{R}^m} = \mathbb{R}^m$,
 the above conclusion become:
\\
Let $S$ be an at most countable subset of $\mathbb{R}^m$.
For each   $u,v \in  F^1_{USC} (\mathbb{R}^m \setminus  S)$, $d_p(u,v)= \left(\int_0^1 H_{\mathbb{R}^m \setminus  S} ([u]_\al, [v]_\al)^p  \,   d\al   \right)^{1/p}$ is well-defined.

In the special case of $S=\emptyset$, the above conclusion become:
\\
For each   $u,v \in  F^1_{USC} (\mathbb{R}^m)$, $d_p(u,v)$ is well-defined (we point out this conclusion in \cite{huang17}).

(\romannumeral3) \
 There exists a metric space $X$ and $u,v\in F^1_{USC}(X)$ such that $d_p(u,v)$ is not well-defined (we point out this conclusion in \cite{huang17}). In \cite{huang17}, we introduce the $d_p^*$ metric on $F^1_{USC}(X)$, which is an expansion of the $d_p$ distance on $F^1_{USC}(X)$.

We have shown that
for each $u, v \in F^1_{USC} (X)$,
 $ d_p^* (u,v) \geq   \left(  \frac{H_{\rm end} (u,v)^{p+1} }{p+1}   \right)^{1/p}$
and
$
  d_p^* (u,v) \geq    H_{\rm end}' (u,v)  ^{1+1/p}.
$
If $d_p(u,v)$ is well-defined, then  $d_p(u,v) = d_p^*(u,v)$ and of course $d_p^*(u,v)$ can be replaced by $d_p(u,v)$ in the above inequalities.

We have obtained the characterizations of total boundedness, relative compactness and compactness in
$(F^1_{USCG} (X)^p, d_p)$. These conclusions generalize the corresponding conclusions
in \cite{huang9}.
Our results indicate that
for a
  subset $U$ in $F^1_{USCG}   (  X )^p$,
$U$ is total bounded (respectively, relatively
compact, compact) in $ (F^1_{USCG}   (  X )^p,   d_p)$
if and only if $U$ is total bounded (respectively, relatively
compact, compact) in $(F^1_{USCG} (X), H_{\rm end})$
and $U$ is $p$-mean equi-left-continuous.

We have shown that
 $ (F^1_{USCG}   (\widetilde{X}   )^p, d_p)$ is a completion of  $ (F^1_{USCB}   (  X ), d_p)$,
and thus a completion of $ (F^1_{USCG}   (  X )^p, d_p)$.

We believe that
the results of this paper have potential applications in the work relevant to $d_p$ distance on fuzzy sets.

This paper is essentially the paper chinaXiv:202110.00083v7 submitted to http://chinaxiv.org/ on 2022.06.20. Compared to the latter, we add
Theorem \ref{regn} and
Theorem \ref{pfn} in this paper. Theorem \ref{pfn} is an immediately corollary of
 Theorem \ref{pfne}.
 In additional, we make some minor adjustments.

\section*{Acknowledgement}

The author would like to thank the Area Editor and the four anonymous referees
for their invaluable comments and suggestions
which improve the readability of this paper.

\end{document}